\documentclass[11pt,reqno,a4]{amsart}
\usepackage[margin=2.5cm]{geometry}
\usepackage[utf8]{inputenc}
\DeclareUnicodeCharacter{00A0}{~}
\usepackage{amssymb,amsmath,amscd,amsfonts,amsthm,bbm}
\usepackage[usenames, dvipsnames]{color}
\usepackage{graphicx}
\usepackage{caption}
\usepackage{subcaption}
\usepackage{enumitem}
\usepackage[colorlinks=true]{hyperref}
\usepackage{array}

\newtheorem{theo}{Theorem}[section]

\newtheorem{lemma}[theo]{Lemma}
\newtheorem{coro}[theo]{Corollary}
\newtheorem{Def}[theo]{Definition}
\newtheorem{prop}[theo]{Proposition}

\theoremstyle{remark}
\newtheorem{rem}{Remark}[section]
\newtheorem{example}{Example}[section]

\theoremstyle{definition}

\newcommand\nc\newcommand
\newcommand\dmo\DeclareMathOperator

\nc{\dz}{{\bf d}_z}
\nc{\N}{\mathbb{N}}
\nc{\R}{\mathbb{R}}
\nc{\C}{\mathbb{C}}
\nc{\E}{\mathbb{E}}
\nc{\PP}{\mathbb{P}}
\nc{\bdm}{\begin{displaymath}}
\nc{\edm}{\end{displaymath}}
\nc{\bea}{\begin{eqnarray*}}
\nc{\eea}{\end{eqnarray*}}
\nc{\la}{\langle}
\nc{\ra}{\rangle}
\nc{\Cplus}{\mathbb{C}_+}
\nc{\Rplus}{\mathbb{R}^+}
\nc{\pitilde}{\widetilde{\pi}}
\nc{\tv}{\mathrm{tv}}
\nc{\Cnabla}{\mathbb{C}^{\nabla}}
\nc{\im}{\mathrm{Im}}
\nc{\bs}{\boldsymbol}
\nc{\ti}{\widetilde}
\nc{\Msub}{M_{\mathrm{sub}}}
\nc{\diag}{\mathrm{diag}}
\nc{\ii}{\mathrm{i}}
\numberwithin{equation}{section}
\nc{\bv}{\boldsymbol{\varepsilon}}
\nc{\dzz}{\boldsymbol{\delta}_z}
\nc{\tr}{\mathrm{tr}\,}
\nc{\cvgP}[1]{\xrightarrow[#1]{\mathcal P}}
\nc{\cvgD}{\xrightarrow[]{\mathcal D}}
\nc{\eqdef}{\stackrel{\triangle}{=}} 
\nc{\lrn}{\left|\!\left|\!\left|} 
\nc{\rrn}{\right|\!\right|\!\right|} 
\nc{\bell}{\boldsymbol{\ell}}
\nc{\blambdaR}{\boldsymbol{\lambda}_R}
\nc{\comment}[1]{\textcolor{blue}{[{\it #1}]}}

\nc{\rhoVn}{\rho(V_n)}	
\nc{\rhoV}{\rho(V)}	
\nc{\tkappa}{{\widetilde \kappa}}

\nc{\Ncenter}{\stackrel{\circ}{\mathcal N}_n}
\nc{\Sn}{{\mathcal S}_n}
\nc{\Kn}{{\mathcal K}_n}

\dmo{\Imm}{Im}
\dmo{\Real}{Re}
\dmo*{\dist}{dist}
\dmo{\var}{var}
\dmo{\trace}{Tr}
\dmo{\rank}{rank}
\dmo{\spec}{spec}

\nc{\Rd}{R_n^{1/2}}
\nc{\Rdb}{\bar{R}_n^{1/2}}
\nc{\V}{\mathcal V}
\nc{\VV}{|\mathcal V|^2}
\nc{\dlp}{d_{\mathcal LP}}
\nc{\Qij}{Q_{[ij]}}

\nc{\cU}{{\mathcal U}}

\nc{\smax}{\sigma_{\max}} 
\nc{\smin}{\sigma_{\min}} 

\nc{\Xn}{X^{\mathcal N}}
\nc{\Yn}{Y^{\mathcal N}}
\nc{\Hn}{H^{\mathcal N}}
\nc{\Gn}{G^{\mathcal N}}
\nc{\tGn}{\widetilde{G}^{\mathcal N}}
\nc{\Fn}{F^{\mathcal N}}
\nc{\cFn}{{F}^{'\mathcal N}}

\nc{\vxi}{\vec{\xi}}
\nc{\vxin}{\vec{\xi}^{\mathcal N}}
\nc{\xin}{\xi^{\mathcal N}}
\nc{\bxin}{\bar{\xi}^{\mathcal N}}

\nc{\Ec}{\E_{\{1\}}}

\nc{\Oeta}[1]{{\mathcal O}_{\eta} \left( {#1}\right)}
\nc{\vOeta}[1]{\vec{\mathcal O}_{\eta} \left( {#1}\right)}
\nc{\specnorm}[1]{\left\| {#1}\right\|_{\mathrm{sp}}}
\nc{\posneq}{\succcurlyeq_{\neq}}

\nc{\bq}{\bs{q}}
\nc{\qt}{\widetilde{q}}
\nc{\bqt}{\bs{\qt}}
\nc{\qvec}{\vec{\bs q}}
\nc{\qvecstar}{\qvec_*}
\nc{\bqstar}{\bq_*}
\nc{\bqtstar}{\bqt_*}

\nc{\bg}{\bs{g}}
\nc{\gt}{\widetilde{g}}
\nc{\bgt}{\bs{\gt}}

\nc{\rr}{r}
\nc{\rt}{\widetilde{r}}
\nc{\br}{\bs{r}}
\nc{\brt}{\bs{\rt}}
\nc{\rvec}{\vec{\bs r}}
\nc{\rvecstar}{\rvec_*}
\nc{\brstar}{\br_*}
\nc{\brtstar}{\brt_*}

\nc{\pt}{\widetilde{p}}
\nc{\bp}{\bs{p}}
\nc{\bpt}{\bs{\pt}}
\nc{\pvec}{\vec{\bp}}

\nc{\vp}{\varphi}
\nc{\vpt}{\widetilde{\varphi}}
\nc{\bphi}{\bs{\varphi}}
\nc{\bphit}{\bs{\widetilde \varphi}}
\nc{\phivec}{\bs{\vec{\varphi}}}

\nc{\q}{\bs{q}}
\nc{\qtilde}{\bs{\widetilde q}}	
\nc{\p}{\bs{p}}
\nc{\ptilde}{\bs{\widetilde p}}
\nc{\tp}{\widetilde{\varphi}}

\nc{\vecepsilon}{\vec{\boldsymbol{\varepsilon}}}

\nc{\Qa}{{\mathcal Q}(\bs{\alpha}, A, \bs{a})}
\nc{\Qb}{{\mathcal Q}(\bs{\beta}, B, \bs{b})}

\nc{\Sab}{{\mathcal S}_{AB}}
\nc{\Sa}{{\mathcal S}_{A}}
\nc{\Sb}{{\mathcal S}_{B}}

\nc{\Sig}{A}

\dmo{\eps}{\varepsilon}
\dmo{\ls}{\lesssim}
\dmo{\gs}{\gtrsim}

\nc{\expo}[1]{\exp \left( #1 \rule{0mm}{3mm}\right)}
\dmo{\e}{\mathbb{E}}
\dmo{\pr}{\mathbb{P}}
\dmo{\un}{\mathbbm{1}}
\dmo{\1}{\mathbf{1}}
\nc{\tran}{\mathsf{T}} 	
\nc{\scut}{\snot}		
\nc{\snot}{\sigma_0}	
\nc{\mN}{\mathcal{N}}
\nc{\vpmax}{\|\bphi\|_\infty}
\nc{\tpmax}{\|\bphitilde\|_\infty}
\nc{\bS}{\bs{S}}
\nc{\bd}{\bs{d}}
\nc{\bdt}{\bs{\widetilde{d}}}
\nc{\dt}{\widetilde{d}}
\nc{\HS}{\mathsf{HS}}
\nc{\Mo}{M_0}
\nc{\Res}{\bs{R}}
\nc{\Y}{\bs{Y}}
\nc{\Am}{\bs{A}^{(m)}}
\nc{\Vm}{\bs{V}^{(m)}}
\nc{\Ym}{\bs{Y}^{(m)}}
\nc{\Lm}{\check{\bs{L}}^{(m)}}
\nc{\wY}{{\widetilde Y}}
\nc{\wA}{{\widetilde A}}

\nc{\Blue}[1]{\textcolor{blue}{#1}}
\nc{\comN}[1]{\textcolor{ForestGreen}{#1}}

\newcommand{\qcheck}{\check{q}}

\title[ Non-Hermitian random matrices with a variance profile]{
Non-Hermitian random matrices with a variance profile (II): properties and examples}

\author[N. Cook, W. Hachem, J. Najim, D. Renfrew]{Nicholas Cook, Walid Hachem, Jamal Najim and David Renfrew} 

\date{\today}


\begin{document}

\begin{abstract} 
For each $n$, let $A_n=(\sigma_{ij})$ be an $n\times n$ deterministic matrix and let $X_n=(X_{ij})$ be an $n\times n$ random matrix with i.i.d.\ centered entries of unit variance. In the companion article \cite{cook2018non}, we considered the empirical spectral distribution $\mu_n^Y$ of the rescaled entry-wise product
\[
Y_n = \frac 1{\sqrt{n}} A_n\odot X_n = \left(\frac1{\sqrt{n}} \sigma_{ij}X_{ij}\right)
\]
and provided a deterministic sequence of probability measures $\mu_n$
such that the difference $\mu^Y_n - \mu_n$ converges weakly in probability to the zero measure. A key feature in \cite{cook2018non} was to allow some of the entries $\sigma_{ij}$ to vanish, provided that the standard deviation profiles $A_n$ satisfy a certain quantitative irreducibility property.

In the present article, we provide more information on the sequence $(\mu_n)$, described by a family of \emph{Master Equations}. We consider these equations in important special cases such as separable variance profiles $\sigma^2_{ij}=d_i \widetilde d_j$ and sampled variance profiles $\sigma^2_{ij} = \sigma^2\left(\frac in, \frac jn \right)$ where $(x,y)\mapsto \sigma^2(x,y)$ is a given function on $[0,1]^2$.
Associate examples are provided where $\mu_n^Y$ converges to a genuine limit.

We study $\mu_n$'s behavior at zero and provide examples where $\mu_n$'s density is bounded, blows up, or vanishes while an atom appears. As a consequence, we identify the profiles that yield the circular law.

Finally, building upon recent results from Alt et al. \cite{alt2018local,alt2019location}, we prove that except maybe in zero, $\mu_n$ admits a positive density on the centered disc of radius $\sqrt{\rho(V_n)}$, where $V_n=(\frac 1n \sigma_{ij}^2)$ and $\rho(V_n)$ is its spectral radius.

\end{abstract}

\maketitle



\section{Introduction} 
\label{sec:intro}
For an $n\times n$ matrix $M$ with complex entries and eigenvalues $\lambda_1,\dots,\lambda_n\in \C$ (counted with multiplicity and labeled in some arbitrary fashion), the \emph{empirical spectral distribution (ESD)} is given by
\begin{equation}	\label{def:esd}
\mu^M_n = \frac1n\sum_{i=1}^n \delta_{\lambda_i}\;.
\end{equation}
A seminal result in non-Hermitian random matrix theory is the \emph{circular law}, which describes the asymptotic global distribution of the spectrum for matrices with i.i.d.\ entries of finite variance -- see \cite{cook2018non} for additional references and the survey \cite{2012-bordenave-chafai-circular} for a detailed historical account. 

In the companion paper \cite{cook2018non}, we studied the limiting spectral distribution $\mu_n^Y$ for {\it random matrices with a variance profile} (see Definition \ref{def:model}). More precisely, we provided a deterministic sequence of
probability measures $\mu_n$ each described by a family of \emph{Master
Equations} (see \eqref{def:ME}), such that the difference $\mu^Y_n - \mu_n$ converges weakly in
probability to the zero measure. 
A key feature of this result was to allow a large proportion of the matrix entries to be zero, 
which is important for applications to the modeling of dynamical systems 
such as neural networks and food webs \cite{Ahmadian:2015xw, Allesina:2015ux}. 
This also presented challenges for the quantitative analysis of the Master Equations,
for which we developed the \emph{graphical bootstrapping} argument.

After the initial release of \cite{cook2018non}, a local law version of our
main statement (Theorem~\ref{thm:main}) was proven in \cite{alt2018local} under
the restriction that the standard deviation profile $\sigma_{ij}$ is uniformly
strictly positive and that the distribution of the matrix entries possesses a
bounded density and finite moments of every order.

In this article, we consider in more detail the measures $(\mu_n)$. In particular, we provide new conditions that ensure the positivity of the density of $\mu_n$ and study the behavior of $\mu_n$ at zero. This study allows us to deduce a necessary condition for the circular law.  Additionally, we specialize the standard deviation profile to separable and sampled standard deviation profiles,
which are important from a modeling perspective, and
which yield interesting examples and genuine limits. Simulations illustrate the scope of our results.  
\subsection{The model.} 
We study the following general class of random matrices with non-identically distributed entries.
\begin{Def}[Random matrix with a variance profile]		\label{def:model}
For each $n\ge 1$, let $A_n$ be a (deterministic) $n\times n$ matrix with entries $\sigma_{ij}^{(n)}\ge0$, let $X_n$ be a random matrix with i.i.d.\ entries $X_{ij}^{(n)}\in \C$ satisfying
\begin{equation}	\label{standardized}
\E X_{11}^{(n)}=0\, , \quad \E|X_{11}^{(n)}|^2=1
\end{equation}
and set
\begin{equation}	\label{def:Yn}
Y_n = \frac1{\sqrt{n}}A_n\odot X_n
\end{equation}
where $\odot$ is the matrix Hadamard product, i.e.\ $Y_n$ has entries $Y_{ij}^{(n)} = \frac1{\sqrt{n}} \sigma_{ij}^{(n)} X_{ij}^{(n)}$.
The empirical spectral distribution of $Y_n$ is denoted by $\mu_n^Y$.
We refer to $A_n$ as the \emph{standard deviation profile} and to $A_n\odot A_n = \big((\sigma_{ij}^{(n)})^2\big)$ as the \emph{variance profile}.
We additionally define the \emph{normalized variance profile} as
\[
V_n=\frac1nA_n\odot A_n. \]
When no ambiguity occurs, we drop the index $n$ and simply write $\sigma_{ij}, X_{ij}, V$, etc.  
\end{Def}

\subsection{Master equations and deterministic equivalents} 

The main result of \cite{cook2018non} states that under certain assumptions on
the sequence of standard deviation profiles $A_n$ and the distribution of the
entries of $X_n$, there exists a tight sequence of deterministic probability
measures $\mu_n$ that are \emph{deterministic equivalents} of the spectral
measures $\mu_n^Y$, in the sense that for every continuous and bounded function
$f:\C\to\C$,
\[
\int f\, d\mu_n^Y - \int f\, d\mu_n \xrightarrow[n\to\infty]{} 0 \qquad \text{in probability} .
\]
In other words, the signed measures $\mu_n^Y - \mu_n$ converge weakly in probability to zero.
In the sequel this convergence will be simply denoted by 
\[
\mu_n^Y \sim \mu_n\qquad \text{in probability} \quad (n\to\infty).
\]

The measures $\mu_n$ are described by a polynomial system of \emph{Master Equations}.
Denote by $V_n^\tran$ the transpose matrix of $V_n$, by $\rhoVn$ its spectral radius and by $[n]=\{ 1,\cdots, n\}$. 
For a parameter $s\ge0$, the Master Equations are the following system of $2n+1$ equations in $2n$ unknowns $q_1,\dots,q_n,\qt_1,\dots,\qt_n$:
\begin{equation}	\label{def:ME}	
\begin{cases}
q_i&= \dfrac{ (V_n^\tran {\bq})_i} {s^2+(V_n\bqt)_i(V_n^\tran\bq)_i  }  \\
\\
\qt_i&= \dfrac{ (V_n \bqt)_i} {s^2+(V_n\bs {\qt})_i(V_n^\tran\bq)_i  } \\ 
\\
&\sum_{i\in [n]} q_i  = \sum_{i\in [n]} {\qt}_i
\end{cases}
\ ,\qquad q_i, \qt_i\ge0,\ i\in [n],
\end{equation}
where $\bq,\bqt$ are the $n\times 1$ column vectors with components 
$q_i,\qt_i$, respectively. In the sequel, we shall write 
$\qvec = \begin{pmatrix} \bq \\ \bqt \end{pmatrix}$. 
If $s\ge \sqrt{\rhoVn}$, it can be shown that the only non-negative solution is
the trivial solution $\qvec=0$.  When $0<s<\sqrt{\rhoVn}$ and the matrix $V_n$ 
is irreducible, the Master Equations admit a unique positive solution $\qvec$ 
that depends on $s$. This solution 
$s\mapsto \qvec(s)$ is continuous on $(0,\infty)$. With this definition of $\bq(s)$ and $\bqt(s)$, 
the deterministic equivalent $\mu_n$ is defined as the radially symmetric 
probability distribution on $\C$ satisfying 
\[ 
\mu_n\{ z \in \C\, , \ |z|\le s\} = 1 - \frac 1n \bq^\tran (s) V_n \bqt(s)
  \ ,\quad s>0\ .  
\] 
It readily follows that the support of $\mu_n$ is contained in the disk of 
radius $\sqrt{\rhoVn}$.
\subsection{Contributions of this paper} In this article, we continue the study of the model initiated in  \cite{cook2018non}, where we provided existence of a $\mu_n$ such that $\mu_n \sim \mu_n^Y$ for random matrices in Definition \ref{def:model}. In particular, we study properties of $\mu_n$: positivity of its density and its behavior at zero, as well as identify variance profiles that yield the circular law. We also consider several special classes of variance profiles.  

In Section \ref{sec:results}, we recall the main results of \cite{cook2018non}. Then, in Proposition \ref{prop:LB} and Theorem \ref{th:positive} we provide sufficient conditions for which the density of $\mu_n$ is positive on the disc of radius $\sqrt{\rhoVn}$. Special attention is given to the density near zero, for which we give an explicit formula. As a consequence of our formula at zero, we have in Corollary \ref{prop:circlaw} that doubly stochastic normalized variance profiles, i.e. $V_n=\left( n^{-1} \sigma^2_{ij}\right)$ such that
$$
\frac 1n \sum_{i=1}^n \sigma_{ij}^2 ={\mathcal V}\quad \forall j\in [n]\qquad \textrm{and} \qquad  
\frac 1n \sum_{j=1}^n \sigma_{ij}^2 ={\mathcal V}\quad \forall i\in [n]\, .
$$
for some fixed ${\mathcal V}>0$, are, up to conjugation by diagonal matrices, the only profiles that give the circular law. 

In Section \ref{sec:examples}, we provide examples of variance profiles with vanishing entries. In particular, we study band matrices and give an example of a distribution with an atom and a vanishing density at zero (Proposition \ref{prop:singular.profile}).

In Section \ref{sec:separable}, we consider the Master Equations in the case of separable variance profiles.
Consider $D_n=\diag(d_i,\, 1\le i\le n)$ and $\widetilde D_n=\diag(\widetilde d_i,\, 1\le i\le n)$ two $n\times n$ diagonal matrices. Then the matrix model 
$$
Y_n =\frac 1{\sqrt{n}} D^{1/2}_n X_n \widetilde D^{1/2}_n 
$$
admits a separable variance profile in the sense that $\var(Y_{ij})= n^{-1} 
d_i \widetilde d_j$. Note that $\rho(V_n) = n^{-1} \sum_{i\in [n]} d_i \tilde
d_i$ for this model.  In this case the $2n$ Master Equations \eqref{def:ME}
simplify to a single equation, see Theorems \ref{th:separable} and \ref{th:separable-sampled}. As applications,
we recover Girko's Sombrero probability distribution and give examples with
unbounded densities at zero; see Sections \ref{sec:Sombrero} and \ref{sec:unbounded}.

In Section \ref{sec:sampled}, we consider sampled variance profiles, where the profile is obtained by evaluating a fixed continuous function $\sigma(x,y)$ on the unit square at the grid points $\{( i/n,  j/n)\colon 1\le i,j\le n\}$. Here, in the large $n$ limit the Master Equations \eqref{def:ME} turn into
an integral equation defining a genuine limit for the ESDs: 
\[
\mu_n^Y \xrightarrow[n\to\infty]{} \mu^\sigma
\]
weakly in probability; see Theorem \ref{th:sampled}.

Finally, Section \ref{sec:positivity} is devoted to the proof of the results in Section \ref{sec:results} concerning positivity and finiteness of the density of $\mu_n$. Much of this analysis will build upon results developed by Alt et al. \cite{alt2018local,alt2019location} in combination with the regularity of the solutions to the Master Equations proven in \cite{cook2018non}.

\subsection*{Acknowledgements} 
The work of NC was partially supported by NSF grants DMS-1266164 and DMS-1606310.
The work of WH and JN was partially supported by the Labex BEZOUT from the Gustave Eiffel University. 
DR was partially supported by 
Austrian Science Fund (FWF): M2080-N35. DR would also like to thank Johannes Alt, L{\'a}szl{\'o} Erd{\H o}s, and Torben Kr{\"u}ger for numerous enlightening conversations.

\section{Limiting spectral distribution: a reminder and some complements}
\label{sec:results}

In this section, we recall the main results in Cook et al. \cite{cook2018non} and then give theorems concerning the density of $\mu_n$.
\subsection{Notational preliminaries}		\label{sec:notation}
Denote by $[n]$ the set $\{1,\cdots,n\}$ and let $\C_+=\{ z\in \C\, ,\
\im(z)>0\}$. For ${\mathcal X} = \C$ or $\R$, let $C_c(\mathcal X)$
(resp.~$C_c^\infty(\mathcal X)$) the set of $\mathcal X \to\R$ continuous
(resp.~smooth) and compactly supported functions.  Let $\mathcal B(z,r)$ be the
open ball of $\C$ with center $z$ and radius $r$.  If $z\in \C$, then $\bar{z}$
is its complex conjugate; let $\ii^2=-1$.  The Lebesgue measure on $\C$ will be
either denoted by $\ell(\, dz)$ or $dx dy$. 
The cardinality of a finite set $S$ is denoted by $|S|$.
We denote by $\1_n$ the $n\times 1$ vector of 1's. Given two $n\times 1$
vectors $\bs u,\bs v$, we denote their scalar product $\langle \bs u,\bs
v\rangle=\sum_{i\in [n]} \bar{u}_i v_i$.
Let $\bs{a}=(a_i)$ an $n\times 1$ vector. We denote by $\diag(\bs{a})$ the
$n\times n$ diagonal matrix with the $a_i$'s as its diagonal elements. 
For a given matrix $A$, denote by $A^\tran$ its transpose, by $A^*$ its
conjugate transpose, and by $\|A\|$ its spectral norm. Denote by $I_n$ the
$n\times n$ identity matrix. If clear from the context, we omit the dimension.
For $a\in \C$ and when clear from the context, we sometimes write $a$ instead
of $a\, I$ and similarly write $a^*$ instead of $(aI)^*=\bar{a}I$.  For
matrices $B,C$ of the same dimensions we denote by $B\odot C$ their Hadamard,
or entry-wise, product (i.e.\ $(B\odot C)_{ij} = B_{ij} C_{ij}$). 
Notations $\succ$ and $\succcurlyeq$ refer to the element-wise inequalities for 
real matrices or vectors. Namely, if $B$ and $C$ are real matrices, 
\[
B\succ C\quad \Leftrightarrow\quad B_{ij} > C_{ij}\quad \forall i,j\qquad \text{and}\qquad B\succcurlyeq C\quad \Leftrightarrow\quad B_{ij} \ge C_{ij}\ \quad\forall i,j .
\]
The notation $B \posneq 0$ stands for $B\succcurlyeq 0$ and $B\neq 0$. 
We denote the spectral radius of an $n\times n$ matrix $B$ by
\begin{equation}	\label{def:specrad}
\rho(B) = \max\big\{\, |\lambda|\colon \text{ $\lambda$ is an eigenvalue of $B$}\,\big\}.
\end{equation}

\subsection{Model assumptions} 

We will establish results concerning sequences of matrices $Y_n$ as in
Definition \ref{def:model} under various additional assumptions on $A_n$ and
$X_n$, which we now summarize.  We note that many of our results only require a
subset of these assumptions. We refer the reader to \cite{cook2018non} for
further remarks on the assumptions.

For our main result we will need the following additional assumption on the
distribution of the entries of $X_n$.

\begin{enumerate}[leftmargin=*, label={\bf A0}]
\item\label{ass:moments} 
(Moments). We have $\E|X_{11}^{(n)}|^{4+\varepsilon}\le \Mo$ for all $n\ge1$ and some fixed $\varepsilon>0$, $\Mo<\infty$.
\end{enumerate}

We will also assume the entries of $A_n$ are bounded uniformly in $i,j\in [n]$, $n\ge1$:

\begin{enumerate}[leftmargin=*, label={\bf A1}]
\item\label{ass:sigmax}	
(Bounded variances).
There exists $\smax\in (0,\infty)$ such that
\[
\sup_n \max_{1\le i,j\le n} \sigma_{ij}^{(n)} \le \smax.
\]
\end{enumerate}

In order to express the next key assumption, we need to introduce the following
{\em Regularized Master Equations} which are a specialization of the
Schwinger--Dyson equations of Girko's Hermitized model associated to $Y_n$. 

\begin{prop}[Regularized Master Equations]
\label{prop:MEt}	
Let $n\ge 1$ be fixed, let $A_n$ be an $n\times n$ nonnegative matrix and write $V_n=\frac1nA_n\odot A_n$. Let $s,t > 0$ be fixed, and consider the following system of equations 
\begin{equation}
\label{def:MEt} 
\left\{ 
\begin{array}{ccc}
\rr_i&=& \dfrac{ (V_n^\tran  \br)_i+t} {s^2+((V_n\brt)_i+t) ((V_n^\tran \br)_i +t) }  \\
\\
\rt_i&=& \dfrac{ (V_n \brt)_i+t} {s^2+((V_n\brt)_i+t) ((V_n^\tran \br)_i +t) } \\ 
\end{array}
\right.\ ,
\end{equation}
where $\br=(\rr_i)$ and $\brt=(\rt_i)$ are $n\times 1$ vectors. Denote by $\rvec=\begin{pmatrix} \br \\ \brt\end{pmatrix}$. Then this 
system admits a unique solution $\rvec=\rvec(s,t) \succ 0$. This solution
satisfies the identity 
\begin{equation}	\label{q_t:trace}
\sum_{i\in [n]} \rr_i \ =\ \sum_{i\in [n]} \rt_i \, .
\end{equation}
\end{prop}

\begin{enumerate}[leftmargin=*,label={\bf A2}]
\item\label{ass:admissible}		
(Admissible variance profile).
Let $\rvec(s,t)=\rvec_n(s,t)\succ 0$ be the solution of the Regularized Master Equations for given $n\ge 1$. For all $s>0$, there exists a constant $C=C(s)>0$ such that
$$
\sup_{n\ge 1} \sup_{t\in (0,1]} \frac 1n \sum_{i\in [n]} \rr_i(s,t)\ \le \ C\ . 
$$
A family of variance profiles (or corresponding standard deviation/normalized variance profiles) for which the previous estimate holds is called {\em admissible}.
\end{enumerate}
\begin{rem} 
After restating the main theorems we list concrete conditions under which we verify \ref{ass:admissible}, namely \ref{ass:sigmin} (lower bound on $V_n$), \ref{ass:symmetric} (symmetric $V_n$) and \ref{ass:expander} (robust irreducibility for $V_n$), cf. section \ref{sec:sufficient}. 
\end{rem}

\subsection{Results from \cite{cook2018non}} 	 \label{results} 

Recall the Master Equations \eqref{def:ME}, and notice that these equations are
obtained from the Regularized Master Equations \eqref{def:MEt} by letting the
parameter $t$ go to zero. Notice however that condition $\sum q_i=\sum
\widetilde q_i$ is required for uniqueness and not a consequence of the
equations as in \eqref{def:MEt}.

In what follows, we will always tacitly assume the standard deviation profile $A_n$ is
{\it irreducible}.  This will cause no true loss of generality, as we can conjugate
the matrix $Y_n$ by an appropriate permutation matrix to put $A_n$ in
block-upper-triangular form with irreducible blocks on the diagonal. The 
spectrum of $Y_n$ is then the union of the spectra of the corresponding block
diagonal submatrices.

\begin{theo} [Cook et al. \cite{cook2018non}] 
\label{thm:master} 
Let $n\ge 1$ be fixed, let $A_n$ be an $n\times n$ nonnegative matrix and write $V_n=\frac1nA_n\odot A_n$. Assume that $A_n$ is irreducible. 
Then the following hold:
\begin{enumerate}
\item\label{q:sys1} 
For $s \geq \sqrt{\rho(V_n)}$ the system \eqref{def:ME} has the unique solution
$\qvec(s) = 0$.
\item\label{q:sys2} 
For $s\in (0,\sqrt{\rhoV})$ the system \eqref{def:ME} has a unique non-trivial
solution $\qvec(s) \posneq 0$. Moreover, this solution satisfies 
$\qvec(s)\succ 0$. 
\item\label{q=limr} $\qvec(s) = \lim_{t\downarrow 0} \rvec(s,t)$ for 
  $s\in (0, \infty)$. 
\item\label{q:contdif} 
The function $s\mapsto \qvec(s)$ defined in parts (1) and (2) is continuous on 
$(0,\infty)$ and is continuously differentiable on 
$(0,\sqrt{\rhoV}) \cup (\sqrt{\rhoV},\infty)$. 
\end{enumerate} 
\end{theo}

\begin{rem}[Convention]
Above and in the sequel we abuse notation and write $\qvec=\qvec(s)$ to mean a solution of the equation \eqref{def:ME}, understood to be the nontrivial solution for $s\in (0,\sqrt{\rho(V)})$. 
\end{rem}
The main result of \cite{cook2018non} is the following.
\begin{theo}[Cook et al. \cite{cook2018non}]	\label{thm:main}	
Let $(Y_n)_{n\ge1}$ be a sequence of random matrices as in Definition \ref{def:model}, and assume \ref{ass:moments}, \ref{ass:sigmax} and \ref{ass:admissible} hold. Assume moreover that $A_n$ is irreducible for all $n\ge 1$.
\begin{enumerate}
\item There exists a sequence of deterministic measures $(\mu_n)_{n\ge 1}$ on $\C$ such that
\[
\mu_n^Y \sim \mu_n\quad \mbox{ in probability}.
\]
\item Let $\bq(s),\bqt(s)$ be as in Theorem \ref{thm:master}, and for $s\in (0,\infty)$ 
let
\begin{equation}	\label{expF}
F_n(s) = 1-\frac1n\langle \bq(s),V_n \bqt(s)\rangle.
\end{equation}
Then $F_n$ extends to an absolutely continuous function on $[0,\infty)$ which is the CDF of a probability measure with support contained in $[0,\sqrt{\rho(V_n)}]$ and continuous density on $(0,\sqrt{\rho(V_n)})$.
\item For each $n\ge1$ the measure $\mu_n$ from part (1) is the unique radially symmetric probability measure on $\C$ with $\mu_n(\{z:|z|\le s\})= F_n(s)$ for all $s\in (0,\infty)$.
\end{enumerate}
\end{theo}

This theorem calls for some comments. Using the fact that $\mu_n$ is radially 
symmetric along with the properties of $F_n(s) = \mu_n(\{z:|z|\le s\})$, it is
straightforward that $\mu_n$ has a density $f_n$ on $\C\setminus\{ 0 \}$ which 
is given by the formula 
\begin{equation} 
\label{eq:density} 
f_n(z) \ =\ \frac 1{2\pi |z|} \frac{d}{ds}  F_n(s)
\Big|_{s=|z|}\ =\  - \frac 1{2\pi n |z|} 
  \frac{d}{ds} \langle \bq(s),  V \bqt(s) \rangle \Big|_{s=|z|} 
\end{equation} 
for $|z| \not\in \{ 0, \sqrt{\rho(V_n)} \}$. We use the convention  
$f_n(z) = 0$ for $|z| = \sqrt{\rho(V_n)}$.

\subsection{Sufficient conditions for admissibility}\label{sec:sufficient}

We now recall a series of assumptions that enforce \ref{ass:admissible} and are
directly checkable from the variance profiles $(V_n)$ without solving a priori
the regularized master equations.
\begin{enumerate}[leftmargin=*, label={\bf A3}]
\item\label{ass:sigmin}		
(Lower bound on variances).
There exists $\smin>0$ such that
$$
\inf_n \min_{1\le i,j\le n} \sigma_{ij}^{(n)} \ge \smin.
$$
\end{enumerate}
\begin{enumerate}[leftmargin=*, label={\bf A4}]
\item\label{ass:symmetric}		
(Symmetric variance profile). For all $n\ge 1$, the normalized variance profile (or equivalently the standard deviation profile) is symmetric: 
$V_n= V_n^\tran\ .$
\end{enumerate}

The following assumption is a quantitative form of irreducibility that considerably generalizes \ref{ass:sigmin}, allowing a broad class of sparse variance profiles . We refer the reader to \cite{cook2018non} for the definition.

\begin{enumerate}[leftmargin=*,label={\bf A5}]
\item\label{ass:expander}		
(Robust irreducibility).
There exists constants $\scut,\delta,\kappa\in (0,1)$ such that for all $n\ge
1$, the matrix $A_n(\scut) = \left( \sigma_{ij} \1_{\sigma_{ij} \geq \scut}
\right)$ is $(\delta,\kappa)$-robustly irreducible.
\end{enumerate}

We gather in the following theorem some results from \cite{cook2018non}, namely Propositions 2.5 and 2.6, as well as Theorem 2.8.

\begin{theo}[Cook et al. \cite{cook2018non}]  Let $(A_n)$ be a family of standard deviation profiles for which \ref{ass:sigmax} holds. If either \ref{ass:sigmin}, \ref{ass:symmetric}, or \ref{ass:expander} holds then \ref{ass:admissible} also holds, in other words, the family $(A_n)$ is admissible. 
\end{theo}

\subsection{Positivity of the density of $\mu_n$} 
\label{sec:behavior-zero}

In \cite[Lemma 4.1]{alt2018local}, it is shown that under Assumption \ref{ass:sigmin}, the density of $\mu_n$ is strictly positive on the disk $\sqrt{\rho(V)}$. We begin by giving a more general assumption under which the density of $\mu_n$, is uniformly bounded from below on its support.

We recall the following definition used in \cite{AEKgram}:
\begin{Def}
A $K \times K$ matrix $T = (t_{ij})_{i,j=1}^K$ with nonnegative entries is called fully indecomposable if for any two subsets $I, J \subset \{1, \ldots, K   \}$ such that $|I| + |J| \geq K$, the submatrix $(t_{ij})_{i\in I,j \in J}$ contains a nonzero entry.
\end{Def}
See \cite{bap-rag-(book)97} for a detailed account on these matrices. 

\begin{enumerate}[leftmargin=*,label={\bf A6}]
\item\label{ass:BFID} (Block fully indecomposable) For all $n \geq 1$, the normalized variance profiles $V_n$ are {\em block fully indecomposable}, i.e.  there are constants $\phi
> 0$, $K \in \N$ independent from $n\ge 1$, a fully indecomposable matrix $Z =
(z_{ij})_{i,j\in [K]}$, with $z_{ij} \in \{0,1\}$ and a partition $(I_j)_{j\in
[K]}$ of $[n]$ such that \[|I_i| \, =\,  \frac{n}{K}, \qquad V_{xy} \, \geq\,
\frac{\phi}{n} z_{ij}, \qquad x \in I_i \quad\text{and}\quad y \in I_j \] for
all $i,j \in [K]$.  
\end{enumerate}
Assumption \ref{ass:BFID} can be seen as a robust version of the full
indecomposability of the matrix $V$. It is well known that the full
indecomposability implies the irreducibility of a matrix. Therefore, one can 
expect that the block full indecomposability implies the robust irreducibility. 
Indeed, the following is an immediate consequence of \cite[Lemma 2.4]{cook2018non}.
\begin{prop}
\label{BFID=>RI} 
\ref{ass:BFID} implies~\ref{ass:expander}.  
\end{prop} 
\begin{rem} \label{rem:sinkhorn} 
In \cite{Sinkhorn2}  full indecomposability is shown to be equivalent to the existence and the uniqueness, up to scaling, of positive diagonal matrices $D_1$ and $D_2$ such that $D_1 V D_2$ is doubly stochastic. Below, in Proposition~\ref{prop:LB} and in particular \eqref{eq:sink}, we see under Assumption \ref{ass:BFID}, $\diag(\bq)V\diag(\widetilde \bq)$ is doubly stochastic. Under Assumption \ref{ass:BFID} an optimal local law for square Gram matrices was proven in \cite{alt2018local}. The boundedness of the density near zero for Hermitian random matrices under the analogous conditions was proven in \cite{AEKquadequations}. \end{rem}
The behavior of $\mu_n$ near zero is an interesting problem.  By Theorem
\ref{thm:main}, $F_n$ admits a limit as $s\downarrow 0$. Is this limit positive
(atom) or equal to zero (no atom)? Is its derivative finite at $z=0$ (finite
density), zero (vanishing density), or does it blow up at $z=0$? In Section \ref{sec:examples}, we give various examples that shed additional light on these questions. 

The next proposition provides an explicit formula for the density $f_n$ at zero under Assumption \ref{ass:BFID}. In Section \ref{sec:singular.profile}, Proposition \ref{prop:singular.profile} provides an example of a simple variance profile with large zero blocks where $\mu_n$ admits a closed-form expression with an atom and a vanishing density at $z=0$. Section \ref{sec:unbounded} and Proposition \ref{prop:blow-up-density} provide an example of a symmetric separable sampled variance profile $\sigma_{ij}= d(i/n) d(j/n)$ where function $d$ is continuous and vanishes at zero. Depending on the function $d$, the density may or may not blow up at $z=0$.
\begin{prop}[No atom and bounded density near zero]
\label{prop:LB} 
Consider a sequence $(V_n)$ of normalized variance profiles and assume that
\ref{ass:sigmax} and \ref{ass:BFID} hold. Let $\qvec(s)$ be as in Theorem
\ref{thm:master}, let $\mu_n$ be as in Theorem \ref{thm:main}, and let 
$\rvec(s,t)=\begin{pmatrix} \br(s,t) \\ \brt(s,t) \end{pmatrix}$ be as in 
Proposition~\ref{prop:MEt}. Then, 
\begin{enumerate}
\item \label{item1:LB} 
The limits $\lim_{t\downarrow 0}\rvec(0,t)$ and $\lim_{s\downarrow 0} \qvec(s)$ 
exist and are equal. 
Writing $\q(0)= (q_i(0)) = \lim_{s\downarrow 0} \q(s)$ and 
$\qtilde(0)= (\tilde q_i(0)) = \lim_{s\downarrow 0} \qtilde(s)$, it holds that 
\begin{equation} 
\label{eq:sink}
q_i(0) (V_n\qtilde(0))_i=1 \quad \textrm{and}\quad 
\tilde q_i(0) (V_n^\tran \q(0))_i=1\,, \quad i\in [n]\ .
\end{equation}
In particular, the probability measure $\mu_n$ has no atom at zero: 
$\mu_n(\{0\})=0\ .$
\item \label{item2:LB} The density $f_n$ of $\mu_n$ on $\C\setminus\{ 0 \}$ 
admits a limit as $z\to 0$. This limit $f_n(0)$ is given by 
$$
f_n(0) = \frac 1n \sum_{i\in[n]} 
 \frac 1{(V_n^\tran\q(0))_i (V_n\qtilde(0))_i} 
 = \frac 1n \sum_{i\in[n]} q_i(0) \widetilde q_i(0)\ .
$$
In particular, there exist finite constants $\kappa,K$ independent of $n\ge 1$ such that 
\begin{equation}\label{eq:unif-bounds-density-0}0< \kappa \le f_n(0) \le K\ .
\end{equation}
\end{enumerate}
\end{prop}
This proposition will be proven in Section~\ref{app:proof-prop:LB}.

In the following theorem, we adapt an argument from \cite{alt2018local} to
bound the density of $\mu_n$ from below. 
\begin{theo} \label{th:positive}  
Assume that \ref{ass:sigmax} holds true and that $A_n$ is irreducible. Then,  
\begin{enumerate}
\item \label{th:positive-i1} Assuming \ref{ass:admissible}, if 
$|z| \in (0, \sqrt{\rho(V)})$, then the density $f_n$ of $\mu_n$ is bounded 
from below by a positive constant that depends on $|z|$ and is independent of
$n$. 
\item \label{th:positive-i2} Assuming \ref{ass:BFID}, then for 
$|z| \in[0, \sqrt{\rho(V)})$, the density $f_n$ of $\mu_n$ (which existence at 
zero is stated by Proposition~\ref{prop:LB}) is bounded from below by a 
positive constant that depends on $|z|$ and is independent of $n$.     
\end{enumerate}
\end{theo}
The proof of Theorem \ref{th:positive} is postponed to Section \ref{proof:th-positive}.
Part \eqref{th:positive-i2} will follow easily by noting that the proof of Proposition \ref{prop:LB}-(\ref{item2:LB}) shows the lower bound in \eqref{eq:LBqq} is bounded away from zero. Finally, we remark the examples in Section \ref{sec:examples} show that one cannot expect $z$ independent lower bounds in general. We do note that our lower bounds only depend on the solution to \eqref{def:ME}.
\subsection{Revisiting the circular law}\label{sec:circular-law-revisited}
Example 2.1 in \cite{cook2018non} uses Theorem \ref{thm:main} to rederive the classical circular law. In \cite[Example 2.2]{cook2018non} and  \cite[Theorem 2.4]{cook2018non} the circular law is shown to also hold for any doubly stochastic variance profile that satisfies Assumption \ref{ass:sigmax}. In both these cases the master equations \eqref{def:MEt}, \eqref{def:ME} simplify to: 
\begin{equation}	\label{ME:scalar}
\rr_i \equiv r = \frac{r+t}{s^2 + (r+t)^2}\ ,\quad r>0\qquad \textrm{and} \qquad  q_i \equiv q = \frac{q}{s^2 + q^2}\ , \quad q \ge 0\ .
\end{equation}

\begin{rem}\label{rem:necessary-circular}
Beyond doubly stochastic variance profiles, it is not hard to see that the circular law also holds for any variance profile of the form $D S D^{-1}$, where $D$ is a diagonal, positive matrix and $S$ is a doubly stochastic matrix. 
Indeed, a random matrix with such a variance profile can be represented as $D C D^{-1}$, where $C$ is a random matrix with a doubly stochastic variance profile. As the matrices $D C D^{-1}$ and $C$ have the same eigenvalues, we see the circular law is the deterministic equivalent for both. 
\end{rem}
We illustrate this observation by recovering a result by Aagaard and Haagerup \cite[Section 4]{AH:DT}.

\begin{example}
Let $\epsilon >0 $ and consider the variance profile $\widetilde{C}$ with entries:
\[ \sigma_{ij}^2 = \begin{cases} 
      \epsilon & \text{ if } i \geq j \\
      \epsilon + 1 & \text{ if } i < j 
   \end{cases}.
\]
Let $A$ be the associated standard deviation profile and consider the random matrix model $n^{-1/2} A\odot X$. Then its deterministic equivalent is given by $\mu_n$, the uniform measure on the disk of radius square root of $\frac{\epsilon}{n} \sum_{i=0}^{n-1}  \left( \frac{1+ \epsilon}{\epsilon} \right)^{\frac{i}{n}} $. In the limit $n\to \infty$, the expression for the radius converges to $\left(1/\log(1+ 1/\epsilon)\right)^{1/2}$.
\end{example}

To prove this, we begin by conjugating the variance profile by $D$, the diagonal matrix with diagonal element
$
D_{ii} = \left( \frac{1+\epsilon}{\epsilon} \right)^{\frac{i-1}{n}}
$. 
Matrix $n^{-1} D \widetilde{C} D^{-1}$ is a circulant matrix with positive entries. Since the row and column sums of a circulant matrix are all equal it follows immediately from Theorem \ref{thm:main} and Section~\ref{sec:circular-law-revisited} that the deterministic equivalent for the ESD is uniform on a disk. The radius of this disk follows from computing the first eigenvalue of the circulant variance profile.

It turns out that the variance profiles mentioned in Remark \ref{rem:necessary-circular} are the only ones that yield the circular law, as the following corollary of Proposition \ref{prop:LB} shows. Its proof is given in Section \ref{proof:circlaw}.

\begin{coro} \label{prop:circlaw}
Let $V$ satisfy Assumptions \ref{ass:sigmax} and \ref{ass:BFID}. Then the density of $\mu_n$ at zero is greater than or equal to $1/(\pi \rho(V))$, with equality if and only if $V = D^{-1} S D$ for some diagonal matrix $D$ and doubly stochastic matrix $S$. In the latter case, $\mu_n=\mu_{\textrm{circ}}$, the circular law.   
\end{coro}
\section{Examples and simulations}
\label{sec:examples}
In this section we study variance profiles with many vanishing entries. We provide simulations for band matrix models in Section \ref{sec:numerics} and exhibit a model with vanishing density and an atom at zero in Section \ref{sec:singular.profile}.  
\subsection{A remark about the numerical computation of the Master Equations}
\label{numerics} 
We first introduce the following compact notation for the Regularized Master Equations \eqref{def:MEt}.
For two $n\times 1$ vectors $\bs{a}$ and $\bs{\widetilde a}$ with nonnegative components, let $\bs{\vec{a}}^\tran=\big( \bs{a}^\tran \ \bs{\widetilde a}^\tran \big)$. Define 
\begin{equation}
\label{def:Psi} 
\Psi(\bs{\vec{a}},s,t) = 
\diag \left( 
 \frac 1{s^2+[(V_n\bs{\widetilde a})_i+t][ (V_n^\tran\bs{a})_i +t]} 
   ;\, i\in [n]\right),  
\end{equation}
and let 
$$
\mathcal I(\bs{\vec{a}},s, t) =
\begin{pmatrix} 
\Psi(\bs{\vec{a}},s,t) V_n^\tran & 0\\
  0 & \Psi(\bs{\vec{a}},s,t) V_n
\end{pmatrix} \bs{\vec{a}} 
  + t \begin{pmatrix} \Psi(\bs{\vec{a}},s,t) \bs{1}_n \\
 \Psi(\bs{\vec{a}},s,t) \bs{1}_n
\end{pmatrix}  
.
$$
Then \eqref{def:MEt} can alternatively be expressed 
$\rvec = {\mathcal I}(\rvec,s,t)$.
The proof of Theorem \ref{thm:master} involves the study of the solution
$\rvec={\mathcal I}(\rvec,s,t)$ to the Regularized Master Equations, where
$t>0$ is a regularization parameter. These equations also provide a numerical
means of obtaining an approximate value of the solution of \eqref{def:ME} via
the iterative procedure 
$$
\rvec_{k+1} = \mathcal I(\rvec_k,s,t)\ ,
$$ 
obtained for a small value of $t$. However, the convergence of this procedure
becomes slower as $t\downarrow 0$. To circumvent this issue, one can solve the
system for relatively large $t$ and then increment $t$ down to zero, using the
previous solution as the new initial vector $\rvec_0$. Additionally, as pointed
out in \cite[Section 4]{speicher-et-al-2007}, considering the average
$\rvec_{k+1} = 2^{-1}\rvec_{k} + 2^{-1} \mathcal I(\rvec_k,s,t)$ leads to 
faster numerical convergence.
\subsection{Band matrix models}
\label{sec:numerics} 
We now provide some numerical illustrations of the results of 
Theorem~\ref{thm:main} in the case of band matrix models. In these cases, closed-form expressions for the density seem out of reach but plots can be  
obtained by numerics.
We consider two probabilistic matrix models with
complex entries (with independent Bernoulli real and imaginary parts) and sampled variance profiles associated to the following functions:
$$
\begin{array}{|c|c|}
\hline
\phantom{\Big(}\textrm{Model A}\phantom{\Big)} & \textrm{Model B}\\
\hline
\phantom{\bigg(} \sigma^2(x,y) = \un_{\left\{|x-y|\leq \frac 1{20}\right\}} \phantom{\bigg(}& \sigma^2(x,y) = (x + 2y)^2 \un_{\left\{|x-y|\leq \frac 1{10}\right\}}\\
\hline
\end{array}
$$
Clearly, the function associated to Model A yields a symmetric variance profile, admissible by Theorem 2.4. Model B satisfies  
the broad connectivity hypothesis (see \cite[Remark 2.8]{cook2018non}), hence \ref{ass:expander} (which is weaker than the broad connectivity assumption).
\begin{lemma}
\label{lem:band-profile} 
Given $\alpha \in (0,1)$ and $a>0$, consider the standard deviation profile matrix 
$A_n = (\sigma(i/n, j/n))_{i,j=1}^n$ where 
$\sigma^2(x,y) = (x+ay)^2\un_{|x-y|\leq \alpha}$. Then, there exists a cutoff 
$\bs\sigma_0 \in (0,1)$ such that for all $n$ large enough, the matrix 
$A_n(\bs\sigma_0)$ satisfies the broad connectivity hypothesis 
with $\delta  = \kappa 
= c\alpha $ for a suitable absolute constant $c>0$. 
\end{lemma}

\begin{proof}
One can take the cutoff parameter $\scut$ sufficiently small that the entries $\sigma_{ij}<\scut$ within the band are confined to the top-left corner of $A$ of dimension $n/100$, say, at which point the argument of \cite[Corollary 1.17]{cook2018lower} applies with minor modification.
\end{proof}

Eigenvalue realizations for models A and B are shown on Figure~\ref{plot-eig}.
On Figure~\ref{plot-dens}, the densities of $\mu_n$ are shown. Plots of the
functions $F_n$ given by~\eqref{expF} are shown on Figure~\ref{plot-F} along with
their empirical counterparts.  

\begin{figure}[ht] 
\centering
\begin{subfigure}{.5\textwidth}
  \centering
  \includegraphics[width=\linewidth]{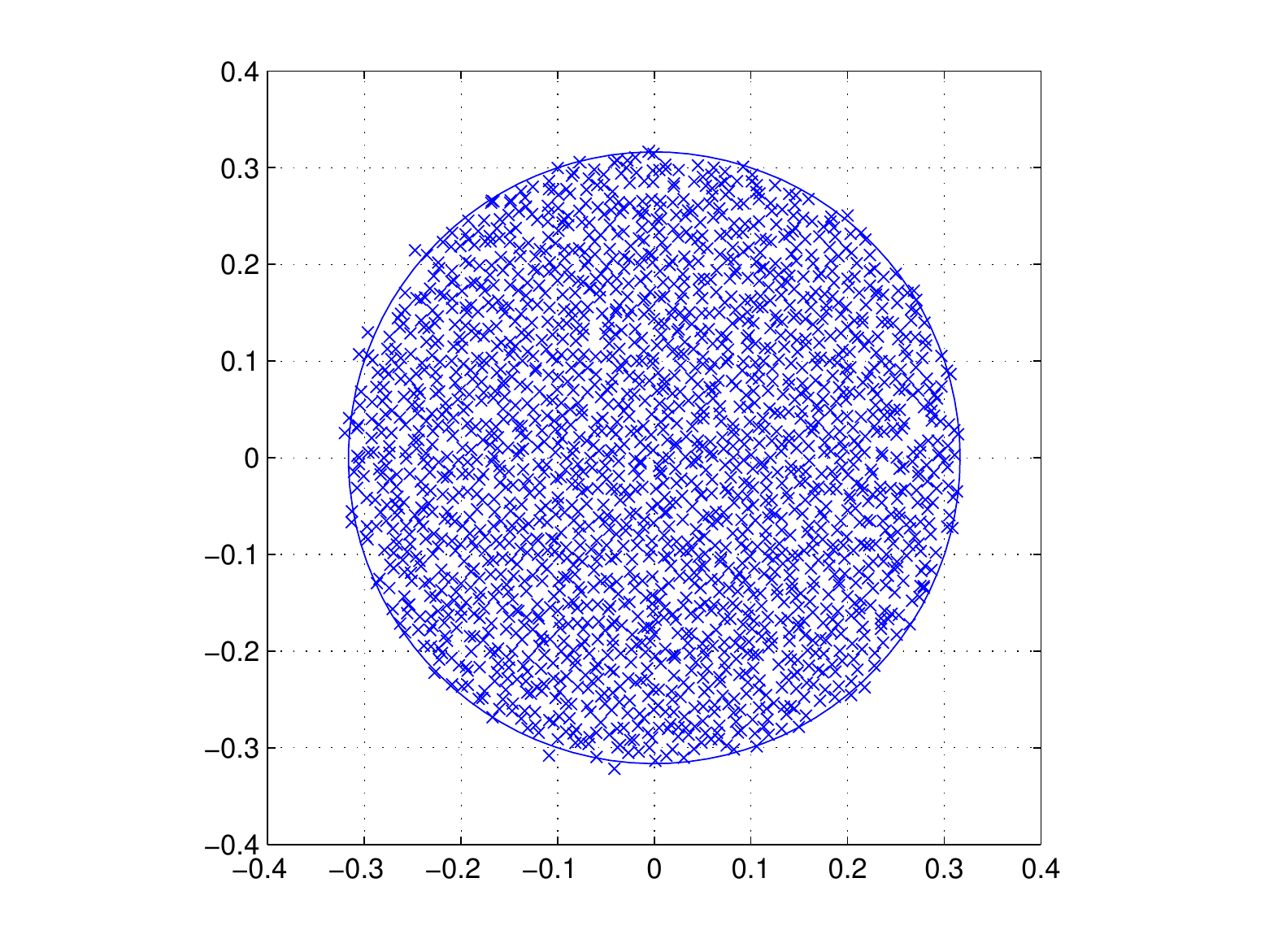}
  \caption{Model A}
  \label{eig-A}
\end{subfigure}%
\begin{subfigure}{.5\textwidth}
  \centering
  \includegraphics[width=\linewidth]{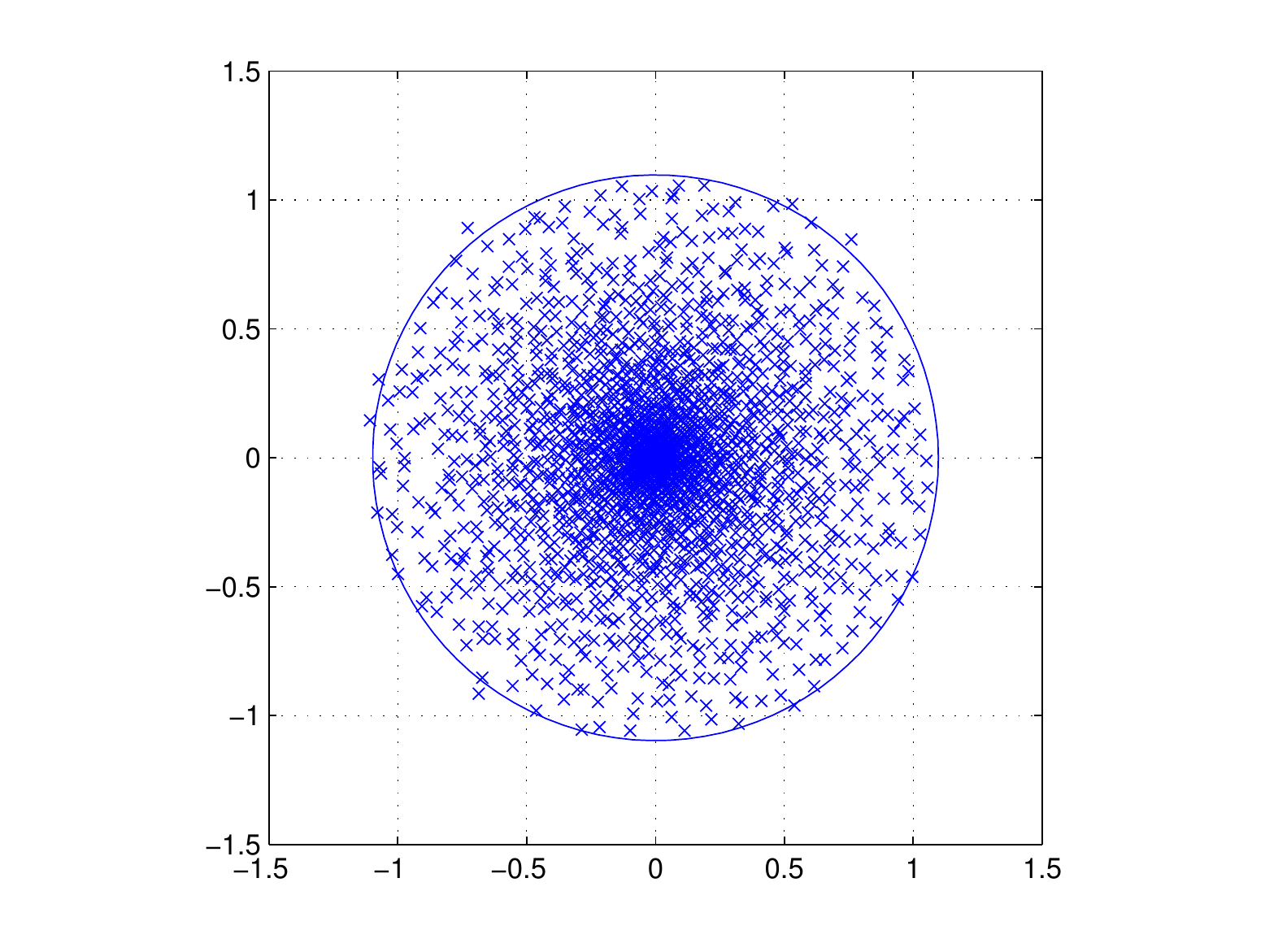}
  \caption{Model B}
  \label{eig-B}
\end{subfigure}
\caption{Eigenvalues realizations. Setting: $n=2000$; the circles' radii are $\sqrt{\rhoV}$.}
\label{plot-eig}
\end{figure}

\begin{figure}[ht] 
\centering
\begin{subfigure}{.5\textwidth}
  \centering
  \includegraphics[width=0.95\linewidth]{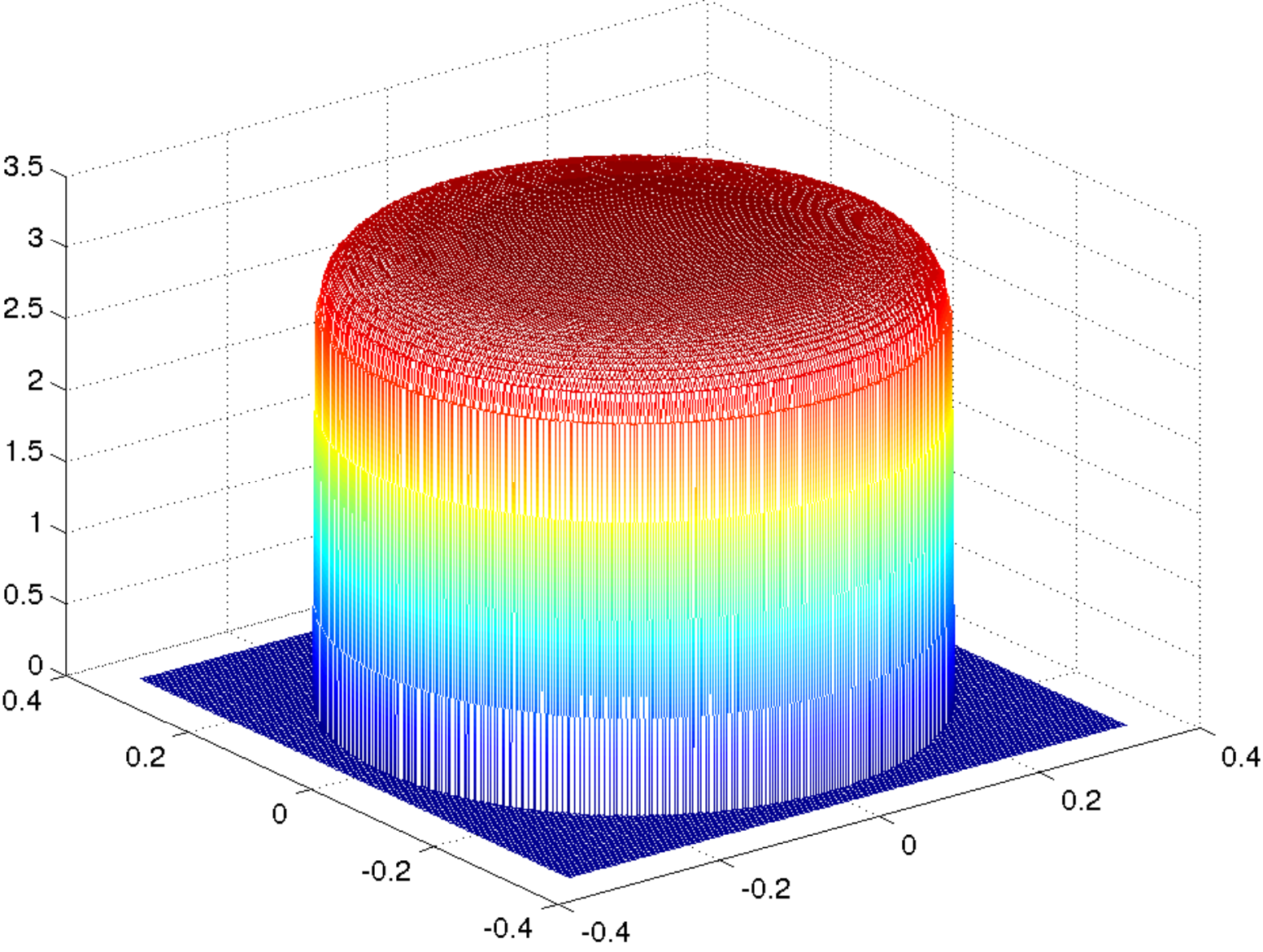}
  \caption{Model A}
  \label{dens-A}
\end{subfigure}%
\begin{subfigure}{.5\textwidth}
  \centering
  \includegraphics[width=0.95\linewidth]{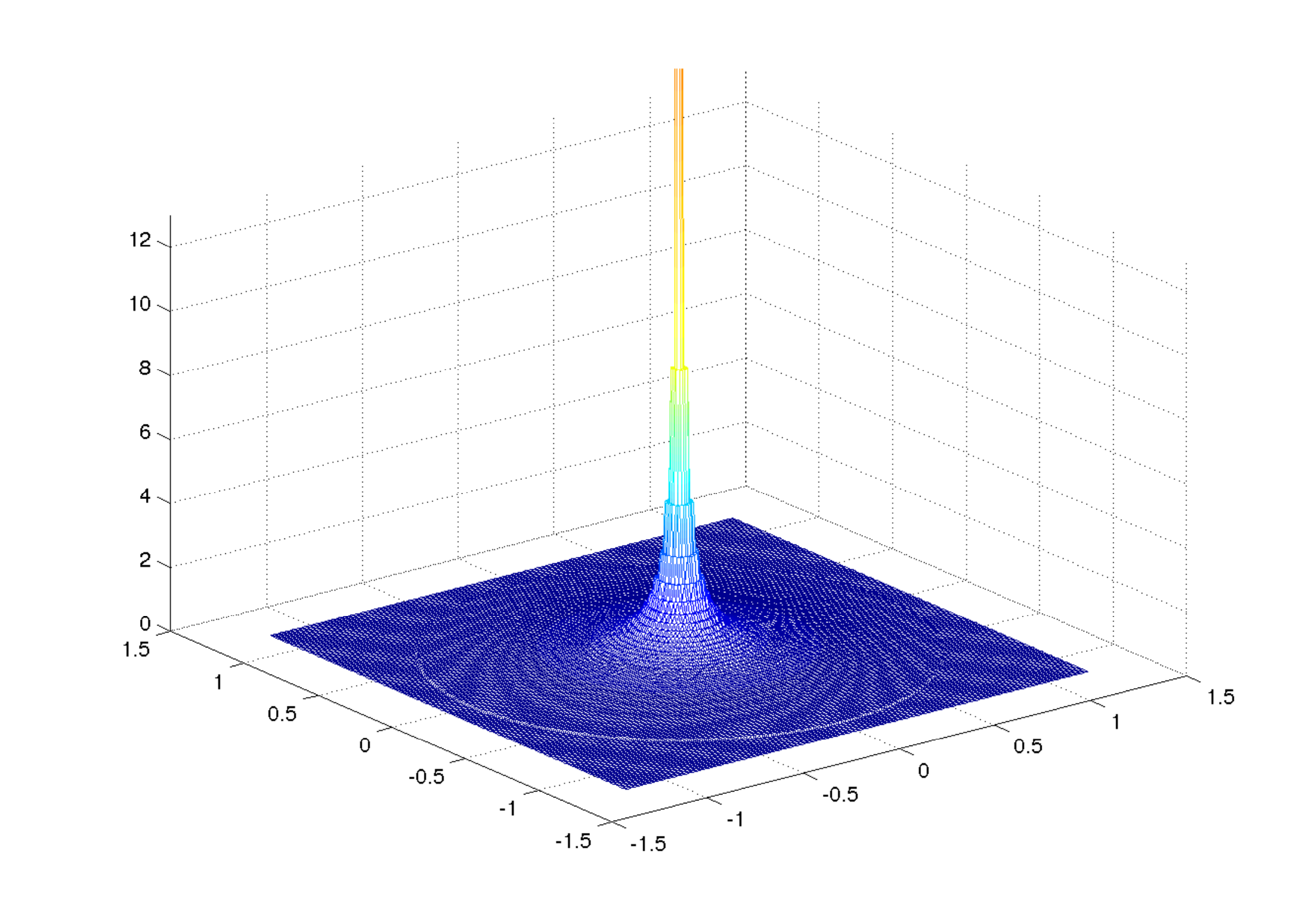}
  \caption{Model B}
  \label{dens-B}
\end{subfigure}
\caption{Densities of $\mu_n$.}
\label{plot-dens}
\end{figure}

\begin{figure}[ht] 
\centering
\begin{subfigure}{.5\textwidth}
  \centering
  \includegraphics[width=0.8\linewidth]{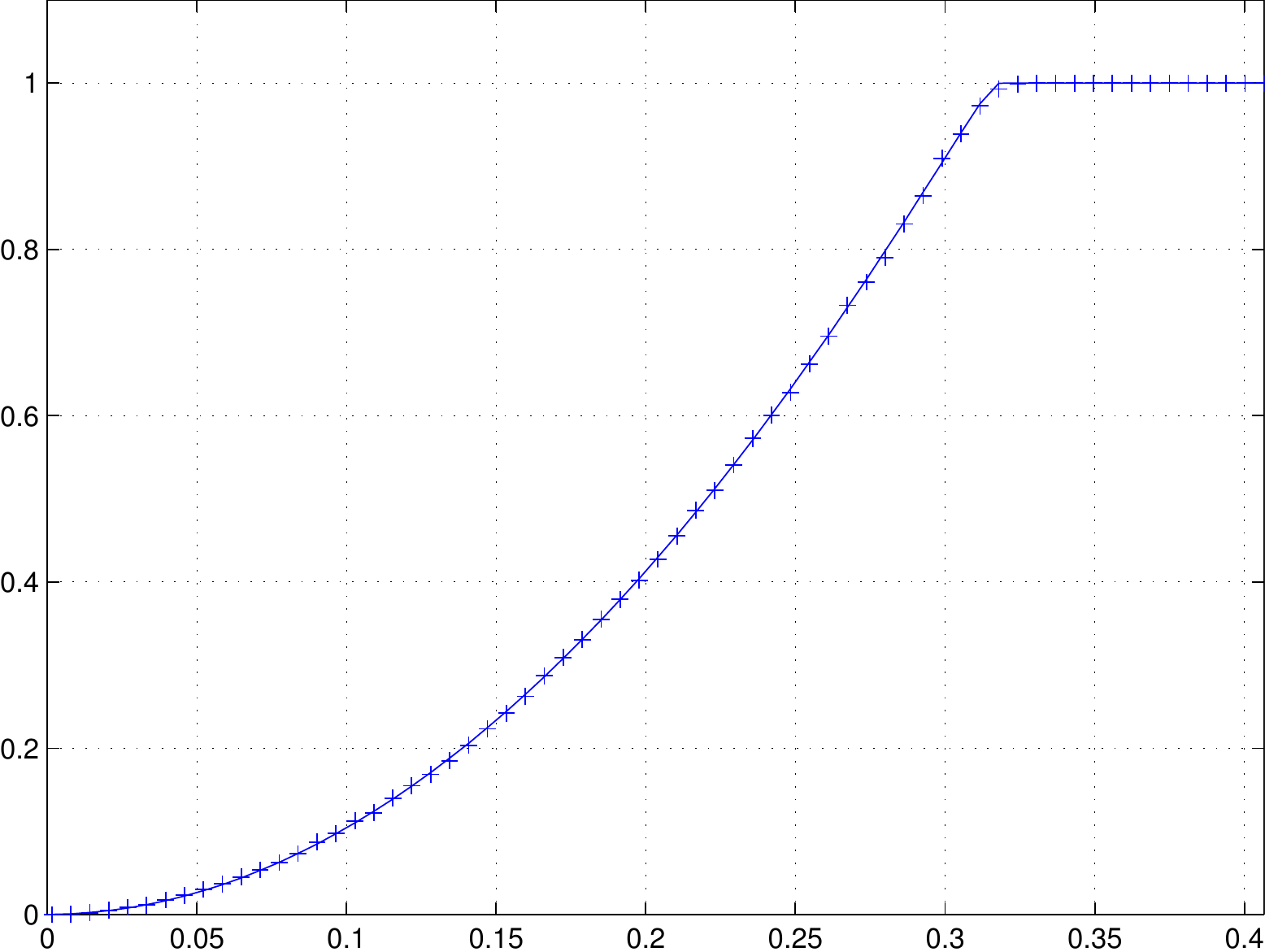}
  \caption{Model A\\ \phantom{x}}
  \label{F-A}
\end{subfigure}%
\begin{subfigure}{.5\textwidth}
  \centering
  \includegraphics[width=0.8\linewidth]{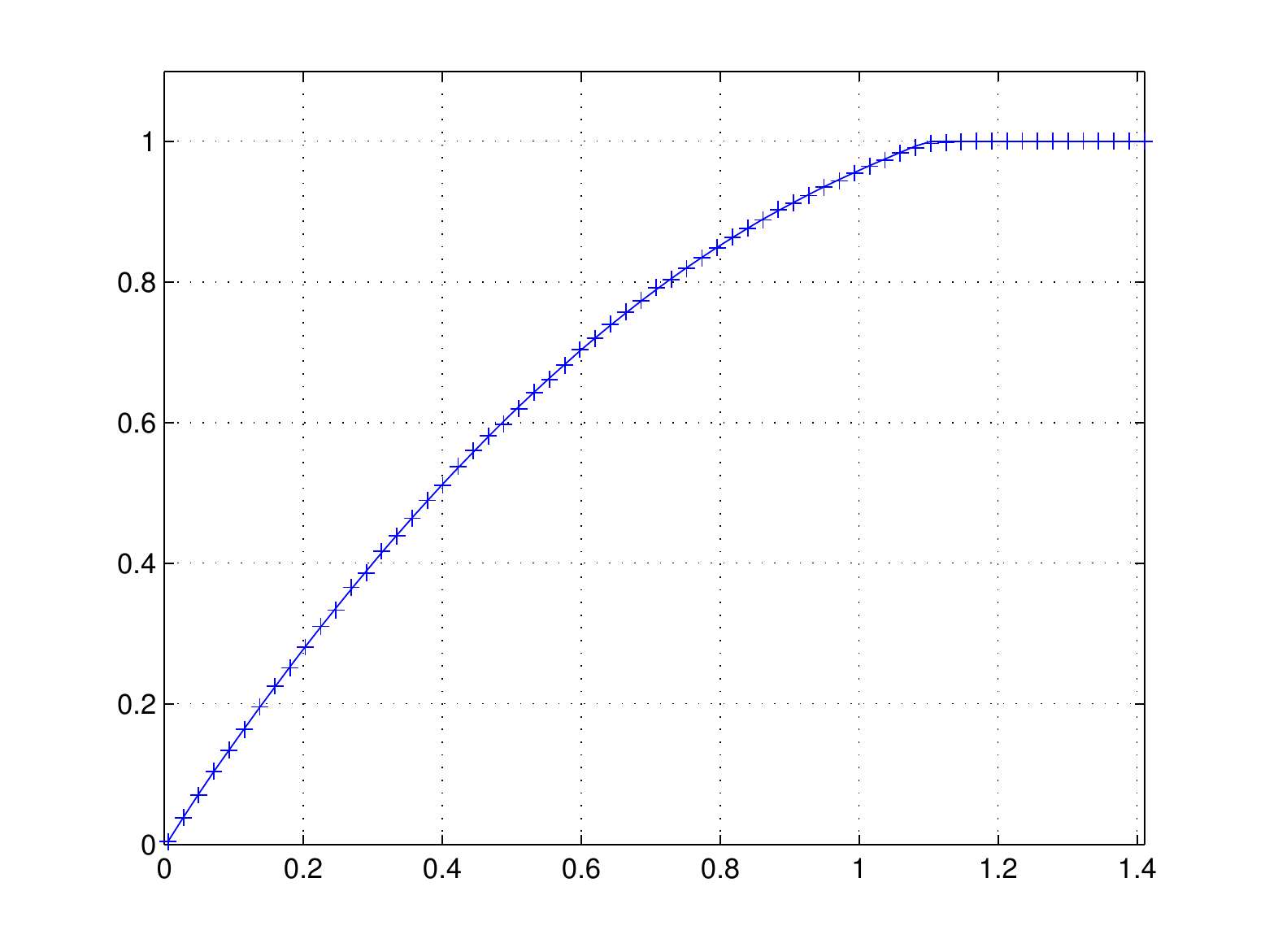}
  \caption{Model B\\ \phantom{x}}
  \label{F-B}
\end{subfigure}
\caption{Plots of $F_n(s)$ (plain lines) and their empirical realizations 
(``+''). The setting is the same as in Figure~\ref{plot-eig}.} 
\label{plot-F}
\end{figure}

Up to the ``corner effects'', the variance profile for Model~A is a scaled 
version of the doubly stochastic variance profile considered in 
Section~\ref{sec:circular-law-revisited}. It is therefore expected that the density for
Model~A is ``close'' to the density of the circular law. This is confirmed
by Figures \ref{eig-A}, \ref{dens-A} and \ref{F-A}. Note in particular that $F_n$ depicted on 
Figure~\ref{F-A} is close to a parabola, which is the radial marginal of the circular law. 

Due to the form of the variance profile of Model~B, a good proportion of
the rows and columns of the matrix $Y_n$ have small Euclidean norms.  We can
therefore expect that many of the eigenvalues of $Y_n$ will concentrate towards
 zero. This phenomenon is particularly visible on the plot of the
density in Figure~\ref{dens-B}. 
\subsection{A limiting distribution with an atom at $z=0$}  \label{sec:singular.profile}
The following Proposition gives an example of a variance profile with a deterministic equivalent that has an atom at zero.

\begin{prop}[Example with an atom and vanishing density at zero] \label{prop:singular.profile} Denote by $J_m$ the $m\times m$ matrix whose elements are
all equal to one. Let $k\ge 1$ be a fixed integer, assume that $n = km$ ($m\ge 1$) and consider the $n\times n$ matrix
\begin{equation}	\label{Asingular}
A_n = \begin{pmatrix} 0 & J_m&\cdots  & J_m \\ J_m & 0 & \cdots & 0 \\ \vdots \\ J_m & 0 &\cdots & 0
\end{pmatrix} . 
\end{equation} 
Associated to matrix $A_n$ is the sequence of normalized variance profiles $V_n=\frac 1n A_n\odot A_n$ with spectral radius $\rho(V_n)= \frac{\sqrt{k-1}}k$. Denote by $\rho^*=\sqrt{\rho(V_n)} = \frac{\sqrt[4]{k-1}}{\sqrt{k}}$. Then 
\begin{enumerate}
\item Assumptions \ref{ass:sigmax} and \ref{ass:admissible} hold true.
\item The function $F_n$ defined in Theorem \ref{thm:main} does not depend on $n$ and is given by 
$$
F_n(s)= F_\infty(s) = \frac 1k \sqrt{(k-2)^2 +4k^2 s^4} \quad \textrm{if}\quad 0\le s\le \rho^*\ ,\\
$$
and $F_{\infty}(s)=1$ if $s>\rho^*$. In particular, $F_{\infty}({0})=1- \frac 2k$ and $\lim_{s\uparrow \rho^*} F_{\infty}(s)=1$. 
\item The density $f_n(=f_{\infty})$ and the measure $\mu_n(=\mu_{\infty})$ do not depend on $n$ and are given by 
\begin{eqnarray*}
f_{\infty}(z) & =& \frac {4k}\pi \frac{|z|^2}{\sqrt{(k-2)^2+ 4k^2 |z|^4}}{\bf 1}_{\{ |z|\le \rho^*\}}\ ,\\
\mu_{\infty}(\, dz)&=& \left( 1 - \frac 2k\right) \delta_0(\, dz) +\frac{4k}\pi \frac{|z|^2}{\sqrt{(k-2)^2+ 4k^2 |z|^4}}{\bf 1}_{\{ |z|\le \rho^*\}} \ell(dz)\ .
\end{eqnarray*}
In particular, $f_{\infty}(0)=0$.
\end{enumerate}
\end{prop}

The definition of $F_n$ readily implies that measure $\mu_n$ admits an atom at zero of weight $1-\frac 2k$ since $\mu_n(\{0\})=F_n(0)=1 - \frac 2k$. This result can (almost) be obtained by simple linear algebra: Note that 
$\rank( Y_n) = \rank( n^{-1/2} A_n \odot X_n) \leq (m-2)k$ for any $X_n$. 
Indeed, since the top-right $m\times (k-1)m$ submatrix of $Y_n$ has row-rank at most $m$, its kernel, and hence the kernel of $Y_n$, has dimension at least $m(k-2)$.
Therefore, $\mu^Y_n$ has an atom at zero with
the weight $\frac{m(k-2)}{mk}= 1-\frac 2k$ (at least) when $n$ is a multiple of $k$.

\begin{rem}[Typical spacing for the random eigenvalues near zero]
We heuristically evaluate the typical spacing for the random eigenvalues in a small disk centered at zero.
$$
\mu_n^Y(B(0,\varepsilon)) \simeq \left( 1 - \frac 2k\right) +\int_{B(0,\varepsilon)} f_{\infty}(z) \ell(dz)
$$
If we remove the $n\left( 1 - \frac 2k\right)=km\left( 1 - \frac 2k\right)=(k-2)m$ deterministic zero eigenvalues, the typical number of random eigenvalues in $B(0,\varepsilon)$ is 
\begin{eqnarray*}
\#\{\lambda_i\ \textrm{random}\ \in B(0,\varepsilon)\} \ =\ n \times \int_{B(0,\varepsilon)} f_{\infty}(z) \ell(dz)
\ =\ 2\pi n \int_0^\varepsilon s h(s)\, ds
\ \propto\  n\varepsilon^4, 
\end{eqnarray*}
with $h(|z|) = f_\infty(z)$. 
Hence, if we want the number of random eigenvalues in $B(0,\varepsilon)$ to be of order ${\mathcal O}(1)$, we need to tune $\varepsilon= n^{-1/4}$ and the typical spacing should be $n^{-1/4}$ near zero. On the other hand, the typical spacing at any point $z$ where $f_{\infty}(z)>0$ is $n^{-1/2}$. Notice that 
$
n^{-1/4} \gg n^{-1/2}
$. This is confirmed by the simulations which show some repulsion phenomenon at zero, cf.\ Figure \ref{plot-dens-block-model}. In particular, the optimal scale for a local law near zero should be $n^{-1/4}$.
\end{rem}

\begin{figure}[ht] 
\centering
\begin{subfigure}{0.49\textwidth}
  \centering
  \includegraphics[width=\linewidth]{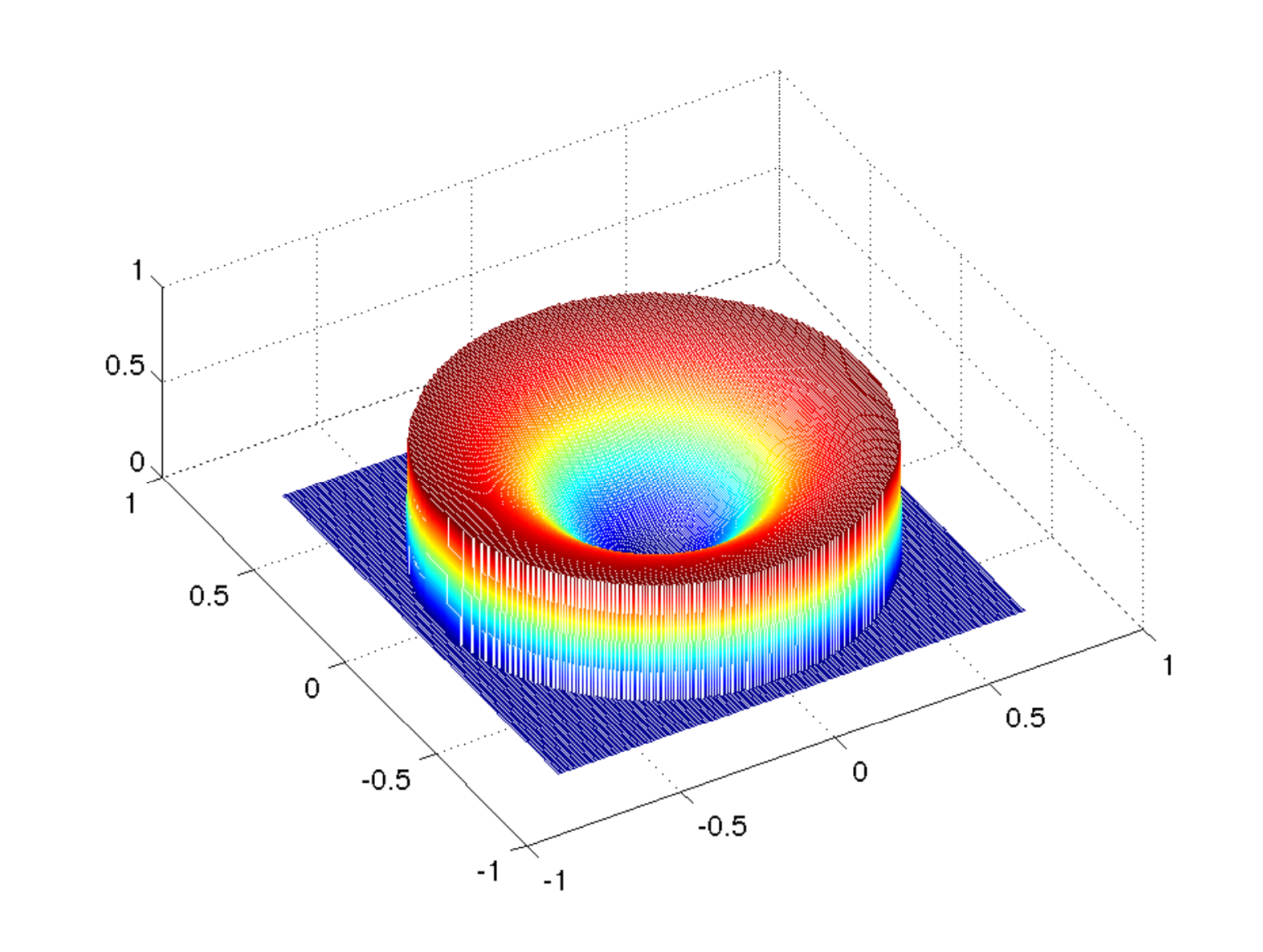}
\end{subfigure}
\begin{subfigure}{.49\textwidth}
  \centering
  \includegraphics[width=\linewidth]{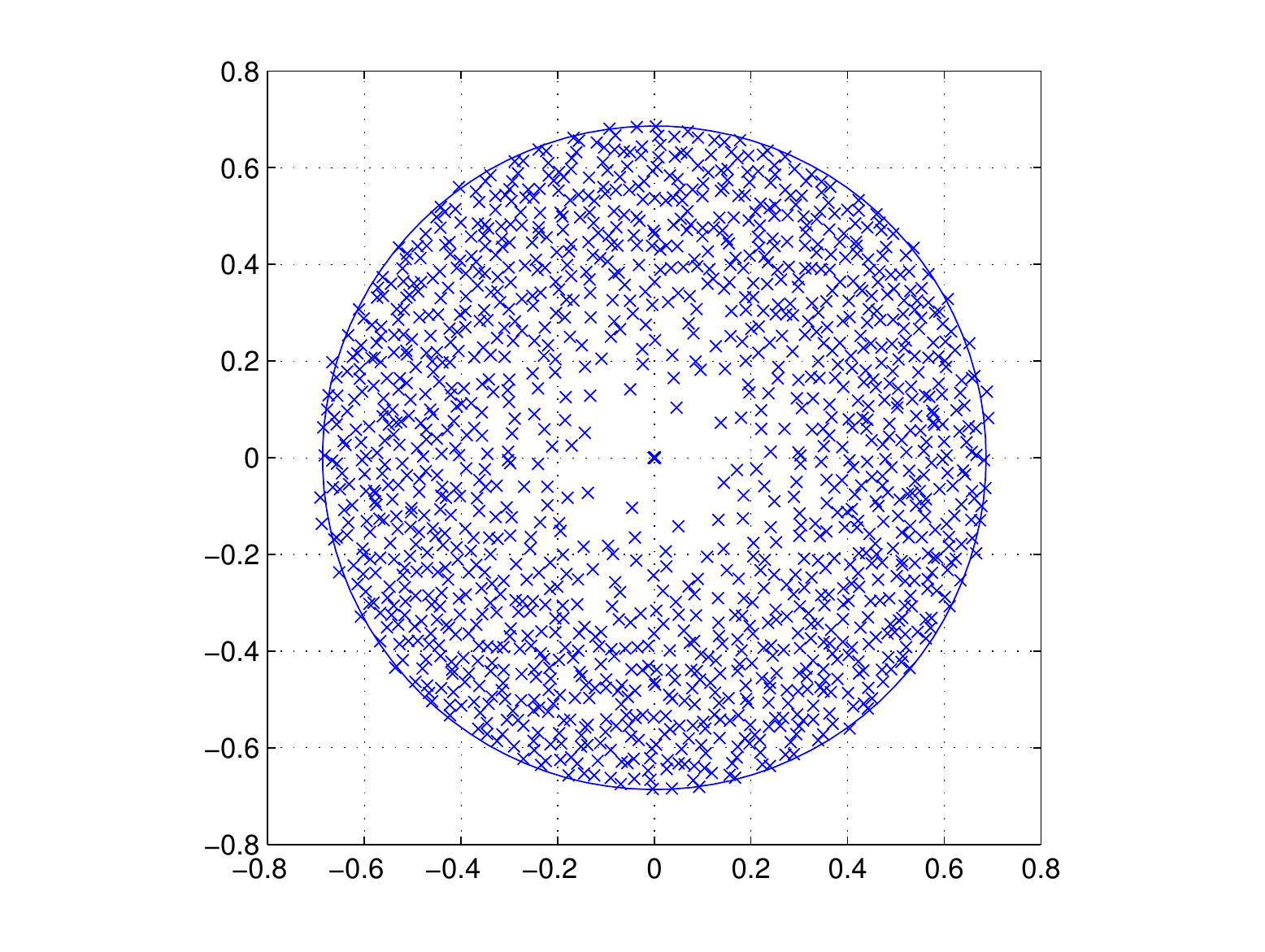}
  \end{subfigure}
\caption{Density $f_{\infty}$ and eigenvalue realizations of a $2001\times 2001$ matrix for the model studied in Proposition \ref{prop:singular.profile} in the case $k=3$. 
A repulsion phenomenon can be observed near zero.}
\label{plot-dens-block-model}
\end{figure}

\begin{proof}[Proof of Proposition \ref{prop:singular.profile}] Simple computations yield $\rho(V_n)=\sqrt{k-1}/k$ and that $V_n$'s spectral measure features a Dirac mass at zero with weight $1-\frac 2k$. Assumption \ref{ass:sigmax} is immediately satisfied, so is \ref{ass:admissible} as the variance profile is symmetric. Item (1) is proved.
We now prove item (2) and first solve the master equations. Since the variance profile is symmetric, we have $\boldsymbol{q}=\boldsymbol{\widetilde q}$ and obviously 
$$
\boldsymbol{q}^\tran=(\underbrace{q\,,\cdots, q}_{m\ \textrm{times}},\underbrace{\check{q}\,,\cdots, \check{q}}_{(k-1)m\ \textrm{times}})
$$
The equations satisfied by $q,\qcheck$ are
$$
q\ =\ \dfrac{k(k-1)\qcheck}{k^2s^2 + (k-1)^2\qcheck^2}\\
\quad \textrm{and} \quad 
\qcheck\ =\ \dfrac{k q}{k^2 s^2 + q^2}
$$
Set $\alpha=\frac{(k-1)(k-2)}2$ and $\beta=\frac{k(k-1)}2$.
We end up with the following equation for $X=q^2$:
$$
s^2 X^2 + 2(k^2 s^4+\alpha)X + k^4 s^6 - k^2(k-1)s^2=0
$$
Hence
$$
\Delta'=\alpha^2 +k^2(k-1)^2s^4\quad \textrm{and}\quad q^2 = \frac{2k\beta s^2 - k^4 s^6}{\sqrt{\Delta'} + k^2s^4 +\alpha}\ .
$$
Now 
$$
\frac 1{mk}\langle \boldsymbol{q},V_n\boldsymbol{q}\rangle = \frac{2(k-1)}{k^2} \frac{kq^2}{k^2s^2 +q^2} =\frac{2(k-1)}k \frac{(k-1) - k^2 s^4}{\sqrt{\Delta'} +\underbrace{\alpha +(k-1)}_{=\beta}} \ .
$$
We finally compute $F_{\infty}$ for $s\le \rho^*$ and notice that a simplification occurs:
\begin{eqnarray*}
F_{\infty}(s)&=& 1 -\frac 1{nk}\langle \boldsymbol{q},V_n\boldsymbol{q}\rangle\\
&=& \frac{k\sqrt{\Delta'} +k\beta - 2(k-1)^2 + 2(k-1) k^2 s^4} {k(\sqrt{\Delta'} + \beta)}
\ =\  \frac{k\sqrt{\Delta'} +\frac{(k-1)(k-2)^2}2 + 2(k-1)k^2 s^4} {k(\sqrt{\Delta'} + \beta)}\\
&=& \frac{k\sqrt{\Delta'} +\frac 2{k-1} \Delta'} {k(\sqrt{\Delta'} + \beta)}\ =\ \frac{2\sqrt{\Delta'}}{k(k-1)} \ =\ \frac 1k \sqrt{(k-2)^2 + 4 k^2 s^4}\ .
\end{eqnarray*}
Part (3) follows by a straightforward computation.
\end{proof}

\section{Separable variance profile}
\label{sec:separable}
The family of separable variance profiles is an interesting instantiation of general variance profiles as it yields simpler and more explicit master equations (see for instance Theorems \ref{th:separable} and \ref{th:separable-sampled}). It is also an abundant source of examples of limiting distributions with density diverging at zero (see Proposition \ref{prop:blow-up-density}).

\subsection{Separable variance profile}

Here we are interested in the following matrix model
\begin{equation}\label{eq:RM-separable}
Y_n=\frac 1{\sqrt{n}} D_n^{1/2} X_n \widetilde D_n^{1/2}\, ,
\end{equation}
where $D_n$ and $\widetilde D_n$ are $n\times n$ diagonal matrices with positive entries.

\begin{enumerate}[leftmargin=*, label={\bf A7}]
\item \label{ass:separable} 
(Separable variance profile).
For each $n\ge 1$ there are deterministic vectors 
$\bd_n, \bdt_n\in (0,\infty)^n$ with components $d_i^{(n)}, \dt_i^{(n)}$ 
respectively such that
\[
A_n\odot A_n = \Big( \big(\sigma_{ij}^{(n)}\big)^2\Big)= \big( d_i^{(n)}\dt_j^{(n))}\big) = \bd_n\bdt_n^\tran.
\]
Moreover there exists $d_{\max} \in (0,\infty)$ such that:
$$
\sup_{n\ge 1} \max(d^{(n)}_i,\widetilde d^{(n)}_i,i\in [n])\ \le\ d_{\max}\, .
$$
\end{enumerate}
\begin{rem}
Notice that in \ref{ass:separable}, we do assume that the $d_i^{(n)},\widetilde d_i^{(n)}$'s are positive but do not assume that they are bounded away from zero.
\end{rem}

Denote by $D_n=\diag(\bd_n)$ and $\widetilde D_n =\diag(\bdt_n)$, then the variance profile in \ref{ass:separable} corresponds to the random matrix model in \eqref{eq:RM-separable}.
This type of model was considered in the context of linear dynamics on structured random networks in \cite{Ahmadian:2015xw}.
\begin{rem}[equivalence with single-sided separable variance profile] \label{rem:single-sided}
Let $\lambda$ be an eigenvalue of $Y_n$ and $\boldsymbol{u}$ its corresponding eigenvector, then 
$$
Y_n \boldsymbol{u} = \lambda \boldsymbol{u} \quad \Rightarrow \quad \frac 1{\sqrt{n}} \left( D_n \widetilde D_n\right)^{1/2}  X_n \left(\widetilde D_n^{1/2}\boldsymbol{u}\right)
= \lambda \left( \widetilde D_n^{1/2} \boldsymbol{u}\right)
$$
In other words, $\lambda$ is also an eigenvalue of matrix $n^{-1/2} \Delta_n^{1/2} X_n$, with $\Delta_n= D_n \widetilde D_n$, corresponding to a single-sided separable variance profile.
\end{rem}
In the sequel, we drop the dependence in $n$ and simply write $A, V,\bd,\bdt,d_i,\dt_i$. As will be shown in the next theorem, the system \eqref{def:ME} of $2n$ equations simplifies into a single equation.
\begin{theo}[Separable variance profile]
\label{th:separable} 
For each $n\ge 1$, let $A_n=(\sigma_{ij})$ be a $n\times n$ matrix with nonnegative elements. Assume that \ref{ass:separable} holds. In this case $V_n=\frac 1n \bs d_n \bs{\widetilde d}_n^\tran$ and $\rho(V)=\frac 1n \langle\bs d, \bs{\widetilde d}\, \rangle$.\begin{enumerate}
\item For each $s\in (0,\sqrt{\rho(V)})$ there exists a unique positive solution $u_n(s)$ to the equation
\begin{eqnarray*}
 \frac 1n \sum_{i\in [n]} \frac {d_i\widetilde d_i}{s^2 + d_i\widetilde d_i u_n(s)}\ =\ 1 .
\end{eqnarray*}
Moreover, the limit $\lim_{s\downarrow 0} u_n(s)$ exists and is equal to one:
$
u_n(0)=\lim_{s\downarrow 0} u_n(s)=1 .
$
If one sets $u_n(s)=0$ for $s\ge \sqrt{\rho(V)}$, then $s\mapsto u_n(s)$ is continuous on $(0,\infty)$ and continuously differentiable on $(0,\sqrt{\rho(V)})$. 
\item The function 
$
F_n(s) = 1 -u_n(s)$, $s \ge 0$
defines a rotationally invariant probability measure $\mu_n$ by
\[
\mu_n( \{ z \, : \, 0 \leq |z| \leq s \}) = F_n(s), \quad s \ge  0 . 
\]
In particular, 
$
\mu_n(\{0\})~=~0\ .
$
\item On the set $\{ z: |z| < \sqrt{\rho(V)}\}$, $\mu_n$ admits the density 
\[
f_n(z) = \frac 1\pi \left(\sum_{i\in[n]} \dfrac {d_i\widetilde d_i}{(|z|^2 +d_i\widetilde d_i u_n(|z|))^2}\right)
\left(\sum_{i\in[n]} \dfrac {d^2_i\widetilde d^2_i}{(|z|^2 +d_i\widetilde d_i u_n(|z|))^2}\right)^{-1} ,
\]
and the support of $\mu_n$ is exactly $\{ z: |z| \le \sqrt{\rho(V)}\}$.
\item In particular, the density is bounded at $z=0$ with value
\[
f_n(0) = \frac 1{n\pi} \sum_{i\in[n]} \dfrac {1}{d_i\widetilde d_i}\ .
\]
\end{enumerate}
Let $(Y_n)_{n\ge1}$ be as in Definition \ref{def:model} and assume that \ref{ass:moments} holds. 
\begin{enumerate}
\setcounter{enumi}{4}
\item Asymptotically,
\[
\mu^Y_n \sim \mu_n \quad \text{in probability}\quad (\text{as}\ n\to \infty) .
\]
\end{enumerate}
\end{theo}
\begin{proof} This theorem is essentially a specification of Theorems \ref{thm:master} and \ref{thm:main} to the case of the variance profile $\bs d\bs{\widetilde d}^\tran$. Introduce the quantities 
$\alpha_n=\frac 1n \langle \bs d, \bs q\, \rangle$ and $\widetilde \alpha_n =\frac 1n \langle \bs{\widetilde d}, \bs{\widetilde q}\, \rangle$
which satisfy the system
\begin{eqnarray*}
1=\frac 1n \sum_{i\in [n]} \frac {d_i\widetilde d_i}{s^2 + d_i\widetilde d_i \alpha_n \widetilde \alpha_n}\quad \textrm{and}
\quad 
\alpha_n\sum_{i\in [n]}\frac{\widetilde d_i}{s^2 + d_i\widetilde d_i \alpha_n \widetilde \alpha_n} =
\widetilde \alpha_n\sum_{i\in [n]}\frac{d_i}{s^2 + d_i\widetilde d_i \alpha_n \widetilde \alpha_n}
\end{eqnarray*}
for $s\in (0,\sqrt{\rho(V)})$ and are equal to zero if $s\ge \sqrt{\rho(V)}$. The function $F_n$ given in \eqref{expF} becomes $F_n(s) = 1 - \alpha_n(s)\widetilde \alpha_n(s)$. Set $u_n(s)=\alpha_n(s)\widetilde \alpha_n(s)$. Notice that $u_n$ satisfies the first equation in Theorem \ref{th:separable}-(1). 
All the other properties of $u_n$ follow from those of $\bs q,\bs{\widetilde q}$, except that $u_n(0)=1$. To prove the later introduce 
$
\xi_{\min}=n^{-1}\sum_{i\in [n]} d_i \widetilde d_i\, >\, 0
$ and recall the definition of $d_{\max}$ in \ref{ass:separable}. Then
$$
\frac{\xi_{\min}}{s^2+ d_{\max}^2 u_n(s)} \ \le\  1\ \le\  \frac{1}{u_n(s)}\ .
$$ 
We deduce that $u_n(s)$ is bounded away from zero and upper bounded as $s\downarrow 0$. Taking the limit in the equation satisfied by $u_n(s)$ as $s\downarrow 0$ along a converging subsequence finally yields that $u_n(s)\xrightarrow[s\to0]{} 1$. 

We do not prove items (3)--(4) since they can be proved as in Theorem \ref{th:separable-sampled}-(3)--(4) below. 

In order to prove Item (5), we need to verify assumption \ref{ass:admissible}. Consider the following random matrix models:
$$
Y_n=\frac 1{\sqrt{n}} D_n^{1/2} X_n \widetilde D_n^{1/2}\ ,\quad 
Y_n^{(2)}= \frac 1{\sqrt{n}} \left( D_n \widetilde D_n\right)^{1/2} X_n\ ,\quad 
Y_n^{(3)}= \frac 1{\sqrt{n}} \Delta_n^{1/2} X_n \Delta_n^{1/2} \, ,
$$
where $\Delta_n =  \left( D_n \widetilde D_n\right)^{1/2}$. Applying Remark \ref{rem:single-sided} twice, these matrix models all have the same spectrum. Moreover, 
the variance profile of matrix $Y_n^{(3)}$ writes $V_n=\frac 1n \left( \sqrt{d_i \widetilde d_i} \right) \left(\sqrt{d_i \widetilde d_i}\right)^\tran$ which is symmetric, fulfilling Assumption \ref{ass:symmetric}, and is hence \ref{ass:admissible} by \cite[Proposition 2.6]{cook2018non}.   
\end{proof}
If the quantities $d_i,\widetilde d_i$ correspond to regular evaluations of continuous functions $d,\widetilde d:[0,1]\to [0,\infty)$, then one obtains a genuine limit. Notice that in this case \ref{ass:sigmax} and \ref{ass:sigmin} and hence \ref{ass:admissible} hold.
\begin{theo}[Sampled and separable variance profile]
\label{th:separable-sampled} Let $d,\widetilde d:[0,1] \to [0,\infty)$ be continuous functions satisfying 
$$
d(0),\widetilde d(0)\ge 0\qquad \textrm{and}\qquad d(x),\widetilde d(x)>0 \qquad \textrm{for}\quad x\in (0,1]\, .
$$
Define a variance profile $(\sigma^2_{ij})$ by
\[
\sigma_{ij}^2=d\left( \frac in\right) \widetilde d\left( \frac jn\right) .
\]
Denote by $\rho_\infty=\int_0^1 d(x)\widetilde d(x)\, dx$.
\begin{enumerate}
\item For any $s\in (0,\sqrt{\rho_\infty})$  there exists a unique positive solution $u_\infty(s)$ to the equation
\begin{eqnarray*}
 \int_0^1\frac {d(x)\widetilde d(x)}{s^2 + d(x)\widetilde d(x) u_\infty(s)}\ dx= 1 .
\end{eqnarray*}
If one sets $u_\infty(s)=0$ for $s\ge \sqrt{\rho_\infty}$, then $s\mapsto u_\infty(s)$ is continuous on $(0,\infty)$.
Moreover, the limit $u_\infty(0)=\lim_{s\downarrow 0} u_\infty(s)$ exists and 
$
u_\infty(0) = 1 .
$
\item The function 
\[
F_\infty(s) = 1 -u_\infty(s)  , 
\quad s \ge 0
\]
defines a rotationally invariant probability measure $\mu_\infty$ by
\[
\mu_\infty( \{ z \, : \, 0 \leq |z| \leq s \}) = F_\infty(s), \quad s \ge 0\, ,\quad\textrm{and}\quad  \mu_\infty(\{ 0\})=0\ .
\]
\item The function $s\mapsto u_\infty(s)$ is continuously differentiable on $(0,\sqrt{\rho(V)})$ and $\mu_\infty$ admits the density
\[
f_\infty(z) = \frac 1\pi  \left(\int_0^1 \dfrac{d(x)\widetilde d(x)}{(|z|^2 +d(x)\widetilde d(x) u_\infty(|z|))^2}\, dx\right) 
\left(\int_0^1 \dfrac{d^2(x)\widetilde d^2(x)}{(|z|^2 +d(x)\widetilde d(x) u_\infty(|z|))^2}\, dx \right)^{-1}
\] 
on the set $\{ z:\ |z|<\sqrt{\rho_\infty}\}$ and $f_{\infty}=0$ for $|z|>\sqrt{\rho_\infty}$. In particular, 
the support of $\mu_\infty$ is equal to $\{ z;\ |z|\le \sqrt{\rho_\infty}\}$.
\item If the integral 
$
\int_0^1 ( d(x)\widetilde d(x))^{-1}dx
$ 
is finite, then the density $f_{\infty}$ is bounded at $z=0$ 
with value
\[
f_\infty(0) = \frac 1\pi \int_0^1 \frac{\, dx}{d(x)\widetilde d(x)}\ .
\]
\end{enumerate}
Let $(Y_n)_{n\ge1}$ be as in Definition \ref{def:model} and assume that \ref{ass:moments} holds. 
\begin{enumerate}
\setcounter{enumi}{4}
 \item Asymptotically,
\[
\mu^Y_n \xrightarrow[n\to \infty]{w} \mu_\infty \quad \text{in probability}\ .
\]
\end{enumerate}
\end{theo}

\begin{proof}
Denote by $d_i=d(i/n), \widetilde d_j=d(j/n)$. The associated vectors $\bs d, \bs{\widetilde d}$ meet the conditions of Assumption \ref{ass:separable}. In order to study the existence of $s\mapsto u_\infty(s)$ and its properties, we introduce the solution $u_n\in [0,1]$ defined in Theorem \ref{th:separable}-(1). Notice that $\rho(V)=\frac 1n \bs{\widetilde d}^\tran\bs d\to \rho_\infty$ as $n\to \infty$.

We prove parts (1) and (2). We establish the existence of $s\mapsto u_\infty(s)$ by relying on Arzela--Ascoli's theorem. 

Denote by 
$$
\delta_{\min}=\liminf_{n\ge 1} \frac 1n \sum_{i\in [n]} d^2_i\widetilde d^2_i \ =\ \int_0^1 d^2(x)\widetilde d^2(x)\, dx\ >\ 0\, .
$$ 
and let $d_{\max}=\sup_{x\in [0,1]} \max( d(x), \widetilde d(x))$.
Let $s,t>0$ be such that $s,t<\sqrt{\rho_\infty}$. For $n$ large enough, $s,t<\sqrt{\rho(V)}$ and 
\begin{multline}\label{eq:discrete-derivative}
(t^2-s^2) \frac 1n \sum_{i\in [n]} \frac{ d_i \widetilde d_i}{(s^2 +d_i \widetilde d_i u_n(s))(t^2 +d_i \widetilde d_i u_n(t))}\\
=-(u_n(t)-u_n(s))\frac 1n \sum_{i\in [n]} \frac{ d^2_i \widetilde d^2_i}{(s^2 +d_i \widetilde d_i u_n(s))(t^2 +d_i \widetilde d_i u_n(t))}\ ,
\end{multline}
where the latter follows by simply subtracting equation in Theorem \ref{th:separable}-(i) evaluated at $s$ to itself evaluated at $t$.
Let $\eta\in (0,\delta_{\min})$. Then for $n$ large enough,
\begin{align*}
\frac 1n \sum_{i\in [n]} \frac{ d_i \widetilde d_i}{(s^2 +d_i \widetilde d_i u_n(s))(t^2 +d_i \widetilde d_i u_n(t))}&\le \frac {\boldsymbol{d}^2_{\max}}{s^2 t^2}\ ,\\
\frac 1n \sum_{i\in [n]} \frac{ d^2_i \widetilde d^2_i}{(s^2 +d_i \widetilde d_i u_n(s))(t^2 +d_i \widetilde d_i u_n(t))} &\ge \frac{\delta_{\min}-\eta}{(\rho(V) +\boldsymbol{d}^2_{\max})^2}\ .
\end{align*}
Plugging these two estimates into Eq. \eqref{eq:discrete-derivative} yields for $n$ large enough
$$
\left| u_n(t) - u_n(s)\right| \ \le\  K \times \frac{t+s}{t^2 s^2} \times |t-s|\ ,
$$
where $K$ depends on $\delta_{\min}$ and $d_{\max}$. Notice in particular that $u_n$ being Lipschitz in any interval $[a,b]\subset (0,\sqrt{\rho_\infty})$ is an equicontinuous family. By Arzela--Ascoli's theorem, the sequence $(u_n)$ is relatively compact for the supremum norm on any interval $[a,b]\subset (0,\sqrt{\rho_\infty})$. Let $u$ be an accumulation point for $s\in [a,b]\subset (0,\sqrt{\rho_\infty})$, then $u(s)\in (0,1)$ and $s\mapsto u(s)$ is non-increasing. By continuity  
\begin{equation}\label{eq:continuous}
\int_0^1 \frac{d(x)\widetilde d(x)}{s^2 + d(x)\widetilde d(x) u(s)}\, dx =1 \quad \textrm{for}\quad s\in [a,b]\ ,
\end{equation}
hence the existence. If $u$ and $\widetilde u$ are two accumulation points of $(u_n)$ on $[a,b]$, then 
$$
(\widetilde u(s) - u(s) ) \int_0^1 \frac{d^2(x)\widetilde d^2(x)}{(s^2 +d(x)\widetilde d(x) u(s))(s^2 +d(x) \widetilde d(x) \widetilde u(s))}\, dx =0\ .
$$
By relying on the same estimates as in the discrete case, one proves that the integral on the l.h.s. is positive and hence $u=\widetilde u=u_\infty$. The uniqueness and the continuity of a solution to \eqref{eq:continuous} is established for $s\in(0,\sqrt{\rho_\infty})$. Using similar arguments, one can prove that $u_\infty(s)>0$ for $s\in(0,\sqrt{\rho_\infty})$, that $s\mapsto u_\infty(s)$ satisfies the Cauchy criterion for functions as $s\downarrow 0$ and $s\uparrow \sqrt{\rho_{\infty}}$. In particular, $u$ admits a limit as  $s\downarrow 0$ and $s\uparrow \sqrt{\rho_{\infty}}$ and it is not difficult to prove that 
$$
\lim_{s\downarrow 0} u_\infty(s) =1\quad \textrm{and}\quad \lim_{s\uparrow \sqrt{\rho_{\infty}}} u_\infty(s) =0\ .
$$ 
We now prove (3) and establish that $s\mapsto u_\infty(s)$ is differentiable on $(0,\sqrt{\rho_\infty})$. By considering the continuous counterpart of equation \eqref{eq:discrete-derivative}, we obtain
$$
\frac{u_\infty(t) - u_\infty(s)}{t-s} =- (t+s) \frac{\int_0^1 \dfrac{d(x)\widetilde d(x)}{(s^2 +d(x)\widetilde d(x) u_\infty(s))(t^2 +d(x)\widetilde d(x) u_\infty(t))}\, dx }
{\int_0^1 \dfrac{d^2(x)\widetilde d^2(x)}{(s^2 +d(x)\widetilde d(x) u_\infty(s))(t^2 +d(x)\widetilde d(x) u_\infty(t))}\, dx }\ .
$$
The r.h.s. of the equation admits a limit as $t\to s$, hence the existence and expression of $u_\infty$'s derivative:
$$
u_\infty'(s) = - 2s  \left(\int_0^1 \dfrac{d(x)\widetilde d(x)}{(s^2 +d(x)\widetilde d(x) u_\infty(s))^2}\, dx \right)
\left(\int_0^1 \dfrac{d^2(x)\widetilde d^2(x)}{(s^2 +d(x)\widetilde d(x) u_\infty(s))^2}\, dx \right)^{-1}
$$ 
for $s\in (0,\sqrt{\rho_\infty})$. This limit is continuous in $s$. The density follows from Equation~\eqref{eq:density}: 
$$
f_\infty(z)=-\frac 1{2\pi|z|} u'_{\infty}(|z|)\ .
$$
Item (4) follows from the fact that $u_{\infty}(0)=1$ and by a continuity argument.

We now establish (5). Notice first that $F_\infty(s)=1-u_\infty(s)$ is the cumulative distribution function of a rotationally invariant probability measure on $\C$. Since $u_n\to u_\infty$ for $s\ge 0$ (some care is required to prove the convergence for $s=\sqrt{\rho_\infty}$ but we leave the details to the reader), one has $\mu_n\xrightarrow[n\to\infty]{w} \mu_\infty$. Combining this convergence with Theorem \ref{th:separable}-(3) yields the desired convergence.  
The proof of Theorem \ref{th:separable-sampled} is complete.\end{proof}
\subsection{Example: Girko's Sombrero distribution} \label{sec:Sombrero}
Consider the separable variance profile $\bs d \bs{\widetilde d}^\tran$ with the first $k$ entries of $\bs d$ equal to $a>0$, the last $n-k$ equal to $b>0$ and all the entries of $\bs{\widetilde d}$ equal to $1$. Denote by $\alpha=\frac kn$, by $\beta=\frac{n-k}n$ and by $\rho= \alpha a +\beta b$ the spectral radius of 
$V_n=\frac 1n \bs d\bs{\widetilde d}^\tran$. Below, we apply Theorem \ref{th:separable} and obtain
\begin{equation}	\label{sombrero}
f_n(z) = \frac 1{2\pi ab}\left( (a+b) - \frac{|z|^2(a-b)^2 +ab[2(\alpha a +\beta b) - (a+b)]}{\sqrt{|z|^4(a-b)^2 +2|z|^2ab[2(\alpha a +\beta b) - (a+b)] + a^2b^2}}\right)
\end{equation}
for $|z|<\sqrt{\rho}$ and $f_n(z)=0$ elsewhere. This formula was also derived and further studied in \cite[Eq. (2.63)]{Ahmadian:2015xw}. In the case where $\alpha=\beta =\frac 12$, we recover Girko's Sombrero probability distribution \cite[Section 26.12]{girko-canonical-equations-I}:
\begin{align*}
f_n(z)&= \frac 1{2\pi ab}\left( (a+b) - \frac{|z|^2(a-b)^2}{\sqrt{|z|^4(a-b)^2  + a^2b^2}}\right)\quad \textrm{for} \quad s<\sqrt{\frac {a+b}2}\ .
\end{align*}
In the case $a=b$, we recover the circular law.
\begin{proof}[Computation of $f_n$]
To compute $f_n$,  we proceed as follows: Theorem \ref{th:separable} yields the equation
$$
 \frac {\alpha a}{s^2 + au_n(s)} +   \frac{\beta b}{s^2+ bu_n(s)} =1
$$
with positive solution for $s<\sqrt{\rho}$ 
$$
u_n(s) = \frac{-[s^2(a+b) -ab] +\sqrt{s^4(a-b)^2 +2s^2ab[2(\alpha a +\beta b) - (a+b)] + a^2b^2}}{2ab}\ .
$$
After applying Theorem \ref{th:separable}-(2) and \eqref{eq:density}, 
a short computation now yields \eqref{sombrero}.
\end{proof}

\subsection{Examples of unbounded densities near $z=0$}\label{sec:unbounded}
We now consider a family of separable variance profiles that yield deterministic equivalents with a wide variety of behaviors at zero. 

Consider a separable variance profile $\sigma^2_{ij} = d(i/n) d(j/n)$ where $d:[0,1]\to [0,\infty)$ with $d(0)\ge 0$ and $d(x)>0$ for $x\in(0,1]$.
From Theorem \ref{th:separable-sampled}, we have that if the integral $\int_0^1 d^{-2}(x)dx$ is finite, then so is the density at zero with value:
$$
f_{\infty}(0) = \frac 1\pi \int_0^1 \frac {d\, x}{d^2(x)}\ .
$$
In order to build a distribution $\mu_{\infty}$ whose density at zero blows up, we consider cases where $\int_0^1 d^{-2}(x)dx=\infty$. The following proposition gives the density in a neighborhood of zero.

\begin{prop}\label{prop:blow-up-density}
Let $d:[0,1]\to [0,\infty)$ be a continuous function satisfying $d(0)=0$ and
$d(x)>0$ for $x>0$. Define a variance profile by $\sigma_{ij}^2 = d( i/n)
d(j/n)$ for $i,j\in [n]$ and denote by $\rho_{\infty} =\int_0^1 d^2(x)\, dx$.
Let $f_{\infty}(z)$ be the density defined for $|z|\in (0,\sqrt{\rho_\infty})$.
Then
$$
f_{\infty} (z) \ \sim\  
\frac 1\pi \int_0^1 \frac{d^2(x)}{[|z|^2 + d^2(x) u(|z|)]^2} dx
 \quad \textrm{as}\quad z\to 0\ .
$$
\end{prop}
Proposition \ref{prop:blow-up-density} whose proof is omitted can be proved as Theorems \ref{th:separable} and \ref{th:separable-sampled}.

We denote $u(s)\sim v(s)$ as $s\to0$ if $\lim_{s\to 0} \frac{u(s)}{v(s)}=1$. Then applying Proposition \ref{prop:blow-up-density} to specific functions $d(\cdot)$ yields the following examples:
\begin{example}[Unbounded and bounded densities near $z=0$] 

\begin{enumerate}
\item Let $d(x)=x$ then 
$$
\int_0^1 \frac{x^2}{[s^2 + x^2 u(s)]^2} dx\sim \frac \pi{4s}\quad 
 \textrm{hence}\quad 
 f_{\infty}(z) \sim \frac 1{4|z|}\quad \textrm{as} \quad z\to 0\, .
$$
\item Let $d(x)=\sqrt{x}$ then 
$$
\int_0^1 \frac{x}{[s^2 + x u(s)]^2} dx\sim -2\log(s)
 \quad \textrm{hence}\quad 
 f_{\infty}(z) \sim -\frac{2\log(|z|)}{\pi}\quad \textrm{as} \quad z\to 0\, .
$$
\item Let $d(x)=x^a$ with $a\in (0,\frac 12)$, then 
 $f_{\infty}(0)= \frac 1{\pi(1-2a)}$.
\end{enumerate}
\end{example}

\section{Sampled variance profile}
\label{sec:sampled} 

\subsection{Sampled variance profile} 
Here, we are interested in the case where 
$$
\sigma^2_{ij}(n) = \sigma^2\left( \frac in,\frac jn \right)\ ,
$$
where $\sigma$ is a continuous nonnegative function on $[0,1]^2$. In this situation, the deterministic equivalents will converge to a genuine limit as $n \to \infty$. Notice that \ref{ass:sigmax} holds and denote by
$$
\smax=\max_{x,y\in [0,1]} \sigma(x,y)\quad \textrm{and}\quad \smin=\min_{x,y\in [0,1]} \sigma(x,y)\,\,.
$$ 
For the sake of simplicity, we will restrict ourselves to the case where
$\sigma$ takes its values in $(0,\infty)$, i.e.\ where $\smin>0$, which implies that
\ref{ass:sigmin} holds.

We will use some results from the Krein--Rutman theory (see for instance \cite{deim-livre85}), which generalizes the
spectral properties of nonnegative matrices to positive operators on
Banach spaces.  To the function $\sigma^2$ we associate the linear operator
$\bs V$, defined on the Banach space $C([0,1])$ of continuous
real-valued functions on $[0,1]$ as
\begin{equation}
\label{pos-op} 
(\bs Vf)(x) = \int_0^1 \sigma^2(x,y) f(y) \, dy .
\end{equation} 
By the uniform continuity of $\sigma^2$ on $[0,1]^2$ and the Arzela--Ascoli
theorem, it is a standard fact that this operator is compact
\cite[Ch.~VI.5]{ree-sim-livre-1}. Let $C^+([0,1])$ be the convex cone of 
nonnegative elements of $C([0,1])$:
$$
C^+([0,1])=\left\{ 
f\in C([0,1])\, , \ f(x)\ge 0\quad \textrm{for}\ x\in [0,1]\right\} \; .
$$
Since 
$
\smin>0$,
the operator $\bs V$ is strongly positive, i.e.\ it sends any element of
$C^+([0,1])\setminus \{0\}$ to the interior of $C^+([0,1])$, the set of continuous and positive functions on $[0,1]$. Under these
conditions, it is well-known that the spectral radius $\rho({\bs V})$ of 
$\bs V$ is non zero, and it coincides with the so-called Krein--Rutman
eigenvalue of $\bs V$ \cite[Theorem~19.2 and 19.3]{deim-livre85}. 

To be consistent with our notation for nonnegative finite dimensional vectors,
we write $f \posneq 0$ when $f \in C^+([0,1]) \setminus \{ 0 \}$,
and $f \succ 0$ when $f(x)>0$ for all $x\in [0,1]$.

\begin{theo}[Sampled variance profile]
\label{th:sampled} 
Assume that there exists a continuous function 
$\sigma:[0,1]^2\to (0,\infty)$ such that 
$$
\sigma_{ij}^{(n)} =\sigma\left(\frac in,\frac jn\right) .
$$
Let $(Y_n)_{n\ge1}$ be a sequence of random matrices as in Definition \ref{def:model} and assume that \ref{ass:moments} holds. 
Then, 
\begin{enumerate}

\item\label{cvg-spradius}
The spectral radius $\rho(V_n)$ of the matrix
$V_n = n^{-1}(\sigma_{ij}^2)$ converges to
$\rho({\bs V})$ as $n\to\infty$, where $\bs V$ is the operator on $C([0,1])$
defined by \eqref{pos-op}.  
\item\label{limit-eqs}  
Given $s > 0$, consider the system of equations:  
\begin{equation}
\label{eq:sys-infty} 
\left\{ 
\begin{split} 
&Q_\infty(x,s) = \frac{\int_0^1 \sigma^2(y,x) Q_\infty(y,s)\, dy}
{s^2 +\int_0^1 \sigma^2(y,x) Q_\infty(y,s)\, dy
    \int_0^1 \sigma^2(x,y) \widetilde Q_\infty(y,s)\, dy}\ ,\\ 
&\widetilde Q_\infty(x,s) = 
\frac{\int_0^1 \sigma^2(x,y) \widetilde Q_\infty(y,s)\, dy}
{s^2 +\int_0^1 \sigma^2(y,x) Q_\infty(y,s)\, dy
  \int_0^1 \sigma^2(x,y) \widetilde Q_\infty(y,s)\, dy} \, , \\ 
&\int_0^1 Q_\infty(y,s)\, dy = \int_0^1 \widetilde Q_\infty(y,s)\, dy  .
\end{split} 
\right. 
\end{equation} 
with unknown parameters  $Q_\infty(\cdot, s), \widetilde Q_\infty(\cdot, s) \in C^+([0,1])$. 
Then, 
\begin{enumerate}
\item for $s\geq \sqrt{\rho(\bs V)}$, 
$Q_\infty(\cdot,s) = \widetilde Q_\infty(\cdot,s) = 0$ is the unique
solution of this system. 
\item for $s \in (0, \sqrt{\rho(\bs V)})$, the system has a unique solution 
$Q_\infty(\cdot,s) + \widetilde Q_\infty(\cdot,s) \posneq 0$. This solution
satisfies 
$$
Q_\infty(\cdot,s),\widetilde Q_\infty(\cdot,s)\, \succ 0\ .
$$ 
\item The functions 
$
Q_\infty,\widetilde Q_\infty\,:\,[0,1]\times (0, \infty) \ \longrightarrow \  [0,\infty)
$ 
are continuous, and continuously extended to 
$[0,1]\times [0, \infty)$, with 
$$Q_\infty(\cdot,0)\,,\, \widetilde Q_\infty(\cdot,0) \succ 0\,.
$$ 
\end{enumerate}
\item\label{Finfty} 
The function 
$$
F_\infty(s) = 
1 -\int_{[0,1]^2} Q_\infty(x,s) \, \widetilde Q_\infty(y,s) \, \sigma^2(x,y)
  \, dx\, dy \ , \quad s \in (0,\infty)  
$$
converges to zero as $s\downarrow 0$. Setting $F_\infty(0)=0$, the 
function $F_\infty$ is an absolutely continuous function on $[0,\infty)$ which 
is the CDF of a probability measure whose support is contained in 
$[0, \sqrt{\rho(\bs V)}]$, and whose density is continuous on 
$[0, \sqrt{\rho(\bs V)}]$. 
\item\label{cvg-muinfty} 
Let $\mu_\infty$ be the rotationally invariant probability measure on 
$\C$ defined by the equation 
$$
\mu_\infty( \{ z \, : \, 0 \leq |z| \leq s \}) 
 = F_\infty(s), \quad s \ge 0\, . 
$$
Then, 
$$
\mu^Y_n \xrightarrow[n\to \infty]{w} \mu_\infty \quad \text{in probability}\ .
$$
\end{enumerate}
\end{theo}

The proof of Theorem \ref{th:sampled} is an adaptation of the proofs of Lemmas 4.3 and 4.4 from \cite{cook2018non} to the context of Krein--Rutman's 
theory for positive operators in Banach spaces.
\subsection{Proof of Theorem~\ref{th:sampled}}
\label{proof:th:sampled}
Extending the maximum norm notation from vectors to functions, we also denote by $\|f\|_\infty=\sup_{x\in[0,1]} |f(x)|$ the norm on the Banach space $C([0,1])$.
Given a positive integer $n$, the linear operator $\bs V_n$ defined on $C([0,1])$ as
\[
\bs V_n f(x) =  \frac 1n \sum_{j=1}^n \sigma^2(x,j/n) \, f(j/n)  
\]
is a finite rank operator whose eigenvalues coincide with those of the matrix
$V_n$. It is easy to check that $\bs V_n f \to \bs V f$ in $C([0,1])$ for all
$f\in C([0,1])$, in other words, $\bs V_n$ converges strongly to $\bs V$ in
$C([0,1])$, denoted by 
$$
\bs V_n \xrightarrow[n\to\infty]{str} \bs V
$$  
in the sequel. 
However, $\bs V_n$ does not converge to $\bs V$ in norm, in which
case the convergence of $\rho(\bs V_n)$ to $\rho({\bs V})$ would have been
immediate.  Nonetheless, the family of operators $\{ \bs V_n \}$ satisfies the
property that the set 
$\{ \bs V_n f \, : \, n \ge 1,\, \| f \|_{\infty} \leq 1 \}$ has a compact
closure, being a set of equicontinuous and bounded functions thanks to the
uniform continuity of $\sigma^2$ on $[0,1]^2$.  Following \cite{ans-pal-68},
such a family is named \emph{collectively compact}.  

We recall the following important properties, cf.\ \cite{ans-pal-68}.
If a sequence $(\bs T_n)$ of collectively compact operators on a Banach space
converges strongly to a bounded operator $\bs T$, then:
\begin{itemize}
\item[i)] The spectrum of $\bs T_n$ is eventually contained in any 
neighborhood of the spectrum of $\bs T$. Furthermore, 
$\lambda$ belongs to the spectrum of $\bs T$ if and only if there exist
$\lambda_n$ in the spectrum of $\bs T_n$ such that $\lambda_n \to \lambda$; 
\item[ii)] $(\lambda - \bs T_n)^{-1} \xrightarrow[n\to\infty]{str}  (\lambda - \bs T)^{-1}$ for 
 any $\lambda$ in the resolvent set of $\bs T$. 
\end{itemize}
The statement~\eqref{cvg-spradius} of the theorem follows from \emph{i)}. 
We now provide the main steps of the proof of the statement \eqref{limit-eqs}. 
Given $n\geq 1$ and $s > 0$, let 
$(\q^n(s)^\tran\; \qtilde^n(s)^\tran)^\tran \in \R^{2n}$ be the solution of
the system~\eqref{def:ME} that is specified by Theorem~\ref{thm:master}. Denote by $\q^n(s)= (q^n_1(s),\ldots, q^n_n(s))$ and $\qtilde^n = (\widetilde q^n_1,\ldots, \widetilde q^n_n)$
and introduce the quantities 
\begin{equation}\label{def:Phi-and-tilde}
\Phi_n(x,s) = \frac 1n \sum_{i=1}^n \sigma^2\left(x, \frac in\right) \, \widetilde q^n_i(s)
 \quad \textrm{and}\quad  \widetilde\Phi_n(x,s) = \frac 1n \sum_{i=1}^n \sigma^2\left(\frac in, x\right) \, q^n_i(s) \ .
\end{equation}
By Proposition~2.5 of \cite{cook2018non} (recall that \ref{ass:sigmin} holds), we know that the average 
$$
\langle \q^n(s)\rangle_n = \frac 1n \sum_{i=1}^n q_i^n(s)$$ satisfies
$\langle \q^n(s)\rangle_n \leq \smin^{-1}$.
Therefore, we get from~\eqref{def:ME} that
\begin{equation}\label{eq:max-bound}
\| \q^n(s) \|_\infty\ \leq\ \frac{\smax^2 \langle \q^n(s) \rangle_n}{s^2}\ \leq\ \frac{\smax^2}{\smin s^2}\,\,.
\end{equation}
Consequently the family $\{ \widetilde\Phi_n(\cdot, s) \}_{n\geq 1}$ is an equicontinuous and bounded subset of $C([0,1])$. Similarly, an
identical conclusion holds for the family 
$\{ \Phi_n(\cdot, s) \}_{n\geq 1}$.
By Arzela--Ascoli's theorem, there exists a subsequence (still denoted by $(n)$, with a small abuse of notation) along which
$\widetilde\Phi_n(\cdot, s)$ and $\Phi_n(\cdot, s)$ respectively converge to
given functions 
$\widetilde\Phi_\infty(\cdot, s)$ and $\Phi_\infty(\cdot, s)$ in $C([0,1])$.  Denote 
$$
\Psi_n(x,s) = \frac 1{s^2 + \Phi_n(x,s)\widetilde\Phi_n(x,s)}\quad \textrm{and}\quad 
\Psi_\infty(x,s) = \frac 1{s^2 + \Phi_\infty(x,s)\widetilde\Phi_\infty(x,s)}\ .
$$
and introduce the auxiliary quantities
$$
Q_n(x,s) = \Psi_n(x,s) \widetilde\Phi_n(x,s)\qquad \textrm{and}\qquad \widetilde Q_n(x,s) = \Psi_n(x,s) \Phi_n(x,s)\,.
$$ 
Then there exists $Q_\infty(x,s)$ and $\widetilde Q_\infty(x,s)$ such that 
$Q_n(\cdot,s) \to Q_\infty(\cdot, s)$ and 
$\widetilde Q_n(\cdot,s) \to \widetilde Q_\infty(\cdot, s)$ in 
$C([0,1])$. These limits satisfy 
$$
Q_\infty(x,s)\ =\  \frac{\widetilde \Phi_\infty(x,s)}{s^2 + \Phi_\infty(x,s)\widetilde \Phi_\infty(x,s)}\qquad \textrm{and}\qquad 
\widetilde Q_\infty(x,s)\ =\  \frac{\widetilde \Phi_\infty(x,s)}{s^2 + \Phi_\infty(x,s)\widetilde \Phi_\infty(x,s)}\ .
$$
Moreover, the mere definition of $\q^n$ and $\qtilde^n$ as solutions of \eqref{def:ME} yields that 
\begin{equation}\label{eq:property-Q}
\begin{cases}
Q_n\left(\frac in,s\right) = q_i^n(s) & 1\le i\le n\\
\widetilde Q_n\left( \frac in,s\right) = \widetilde q_i^n(s) & 1\le i\le n.
\end{cases}
\end{equation}
Combining \eqref{def:Phi-and-tilde}, \eqref{eq:property-Q} and the convergence of $Q_n$ and $\widetilde Q_n$, we finally obtain the useful representation
\begin{equation}\label{eq:useful-representation}
\Phi_\infty(x,s)  = \int_0^1 \sigma^2(x,y) \, \widetilde Q_\infty(y,s) \, dy \qquad \textrm{and}
\qquad 
\widetilde \Phi_\infty(x,s)  = \int_0^1 \sigma^2(y,x) \, Q_\infty(y,s) \, dy \ .
\end{equation}
which yields that $Q_\infty$ and $\widetilde Q_\infty$ satisfy the system \eqref{eq:sys-infty}.

To establish the first part of the statement~\eqref{limit-eqs}, we show that
these limits are zero if $s^2 \geq \rho(\bs V)$ and positive if 
$s^2 < \rho(\bs V)$, then we show that they are unique.  It is known that
$\rho(\bs V)$ is a simple eigenvalue, it has a positive eigenvector, and there
is no other eigenvalue with a positive eigenvector. If $\bs T$ is a bounded
operator on $C([0,1])$ such that $\bs T f - \bs V f \succ 0$ for $f \posneq 0$,
then $\rho(\bs T) > \rho(\bs V)$ \cite[Theorem~19.2 and 19.3]{deim-livre85}. 

We first establish \eqref{limit-eqs}-(a). Fix $s^2 \geq \rho(\bs V)$, and assume that $Q_\infty(\cdot, s) \posneq 0$. 
Since $Q_\infty(\cdot, s) = \Psi_\infty \bs V Q_\infty(\cdot, s)$, where
$\Psi_\infty(\cdot, s)$ is the limit of $\Psi_n(\cdot, s)$ along the 
subsequence $(n)$, it holds that $Q_\infty(\cdot, s) \succ 0$, and by the 
properties of the Krein--Rutman eigenvalue, that $\rho(\Psi_\infty \bs V) = 1$. 
From the identity 
$\int Q_\infty(x, s) \, dx  = \int \widetilde Q_\infty(x, s) \, dx$, 
we get that $\widetilde Q_\infty(\cdot, s) \posneq 0$, hence 
$\widetilde Q_\infty(\cdot, s) \succ 0$ by the same argument. 
By consequence, $s^{-2} \bs V f-\Psi_\infty \bs V f \succ 0$ for all 
$f \posneq 0$. 
This leads to the contradiction 
$1 \geq \rho(s^{-2} \bs V) > \rho(\Psi_\infty \bs V) = 1$. 
Thus, $Q_\infty(\cdot, s) = \widetilde Q_\infty(\cdot, s) = 0$.

We now establish \eqref{limit-eqs}-(b). Let $s^2 < \rho(\bs V)$. By an argument based on collective compactness, it holds that 
$$
\rho(\Psi_n \bs V_n) \xrightarrow[n\to\infty]{} \rho(\Psi_\infty \bs V)\ 
$$
and moreover, that 
$\rho(\Psi_n \bs V_n) = 1$ (see e.g.\ the proof of 
Lemma~4.3 of \cite{cook2018non}). Thus, $Q_\infty(\cdot, s) \posneq 0$ and 
$\widetilde Q_\infty(\cdot, s) \posneq 0$, otherwise 
$\rho(\Psi_\infty \bs V) = \rho(s^{-2} \bs V) > 1$. Since
$Q_\infty(\cdot, s) = \Psi_\infty \bs V Q_\infty(\cdot, s)$, we get that 
$Q_\infty(\cdot, s) \succ 0$ and similarly, that 
$\widetilde Q_\infty(\cdot, s) \succ 0$. 

It remains to show that the accumulation point 
$(Q_\infty, \widetilde Q_\infty)$ is unique. The proof of this fact is similar
to its finite dimensional analogue in the proof of Lemma 4.3 from \cite{cook2018non}. In
particular, the properties of the Perron--Frobenius eigenvalue and its
eigenspace are replaced with their Krein--Rutman counterparts, and the matrices
$K_{\vec{\bs q}}$ and $K_{\vec{\bs q},\vec{\bs q}'}$ in that proof are replaced
with continuous and strongly positive integral operators. Note that the end of
the proof is simpler in our context, thanks to the strong positivity assumption
instead of the irreducibility assumption. We leave the details to the reader.

We now address \eqref{limit-eqs}-(c) and first prove the continuity of $Q_\infty$ and $\widetilde Q_\infty$
on $[0,1] \times (0,\infty)$. This is equivalent to proving the continuity
of $\Phi_\infty$ and $\widetilde\Phi_\infty$ on this set. 
Let $(x_k, s_k) \to_k (x,s) \in [0,1] \times (0,\infty)$. The bound
$$
0\ \le\ \widetilde Q_\infty(y,s)\  \le\ \frac{\smax^2}{\smin\, s^2}
$$
follows from \eqref{eq:property-Q} and the convergence of $\widetilde Q_n$ to $\widetilde Q_\infty$.
As a consequence of \eqref{eq:useful-representation}, the
family $\{ \Phi_\infty(\cdot , s_k) \}_k$ is equicontinuous for $k$ large. 
By Arzela--Ascoli's theorem and the uniqueness of the solution of the system, we get that 
$\Phi_\infty(\cdot , s_k) \to_k \Phi_\infty(\cdot , s)$ in $C([0,1])$. 
Therefore, writing 
\[
| \Phi_\infty(x_k , s_k) - \Phi_\infty(x , s)  |
\leq 
\| \Phi_\infty(\cdot , s_k) - \Phi_\infty(\cdot , s)  \|_\infty + 
| \Phi_\infty(x_k , s) - \Phi_\infty(x , s)  |  
\]
and using the continuity of $\Phi_\infty(\cdot, s)$, we get that 
$\Phi_\infty(x_k , s_k) \to_k \Phi_\infty(x , s)$.

The main steps of the proof for extending the continuity of $Q_\infty$ and 
$\widetilde Q_\infty$ from $[0,1] \times (0,\infty)$ to 
$[0,1] \times [0,\infty)$ are the following. 
Following the proof of Proposition \ref{prop:LB}, we can establish that
$$
\lim\inf_{s\downarrow 0} \int_0^1 Q_\infty(x,s)\, dx \ >\  0\ .
$$ 
The details are omitted. Since 
\[
\frac{1}{\widetilde Q_\infty(x,s)} = \frac{s^2}{\Phi_\infty(x,s)} 
+ \widetilde\Phi_\infty(x,s) > 
\smin \int_0^1 Q_\infty(y,s)\, dy \, , 
\]
we obtain that $\| \widetilde Q_\infty(\cdot, s) \|_\infty$ is bounded when 
$s \in (0,\varepsilon)$ for some $\varepsilon > 0$. 
Thus, $\{ \Phi_\infty(\cdot, s) \}_{s\in (0,\varepsilon)}$ is equicontinuous by \eqref{eq:useful-representation}, 
and it remains to prove that the accumulation point $\Phi_\infty(\cdot, 0)$ is
unique. 

This can be done by working on the system 
\eqref{eq:sys-infty} for $s=0$, along the lines of the proof of
Lemma~4.3 of \cite{cook2018non} and Proposition \ref{prop:LB}. Details are omitted.

Turning to Statement~\eqref{Finfty}, the assertion $F(s) \to 0$ as 
$s\downarrow 0$ can be deduced from the proof of Proposition \ref{prop:LB} and a
passage to the limit, noting that the bounds in that proof are independent from
$n$.

Consider the Banach space $\mathcal B = C([0,1];\R^2)$ of continuous functions
$$
\vec f=( f, \widetilde f)^\tran: [0,1] \longrightarrow \R^2
$$
endowed with
the norm $\| \vec f \|_{\mathcal B} = \sup_{x\in[0,1]}  
\max ( | f(x) |, |\widetilde f(x)| )$. 
In the remainder of the proof, we may use the notation shortcut $\Psi_{\infty}^s$ instead of $\Psi_{\infty}(\cdot,s)$ and corresponding shortcuts for quantities $\Phi_\infty(\cdot,s)$, $\widetilde \Phi_{\infty}(\cdot,s)$, $Q_\infty(\cdot,s)$ and $\widetilde Q_\infty(\cdot,s)$.

Given 
$s, s' \in (0, \sqrt{\rho(\bs V)})$ with $s\neq s'$, consider the function 
\[
\Delta \vec Q_\infty^{s,s'} = 
\frac{\bigl( Q^s_\infty - Q^{s'}_\infty , 
 \widetilde Q^s_\infty -  \widetilde Q^{s'}_\infty \bigr)^{\tran}} 
 {s^2 - s'^{\,2}} \in \mathcal B. 
\]
Let $\bs{ V}^\tran$ be the linear operator associated to the kernel $(x,y)\mapsto \sigma^2(y,x)$, and defined as
$$
\bs{V}^\tran f(x) =\int_0^1 \sigma^2(y,x)f(y)\, dy\ .
$$
Then, mimicking the proof of Lemma~4.4 of \cite{cook2018non}, it is easy to prove that
$\Delta \vec Q_\infty^{s,s'}$ satisfies the equation 
\[
\Delta \vec Q_\infty^{s,s'} = \bs M_\infty^{s,s'} \Delta \vec Q_\infty^{s,s'} 
+ \bs a_\infty^{s,s'} , 
\] 
where $\bs M_\infty^{s,s'}$ is the operator acting on $\mathcal B$ and defined in a matrix 
form as 
\begin{equation*}
\bs M_\infty^{s,s'} = 
\begin{pmatrix} 
s^2 \Psi^s_\infty \Psi^{s'}_\infty \bs{  V}^\tran & 
       - \Psi^s_\infty \Psi^{s'}_\infty
    \widetilde\Phi^s_\infty \widetilde\Phi^{s'}_\infty \bs V  \\
- \Psi^s_\infty\Psi^{s'}_\infty
   \Phi^s_\infty\Phi^{s'}_\infty \bs{ V}^\tran
& s^2 \Psi^s_\infty \Psi^{s'}_\infty \bs V
\end{pmatrix} \, , \\ 
\end{equation*}
and $\bs a_\infty^{s,s'}$ is a function ${\mathcal B}$ defined as
\begin{gather*} 
\bs a_\infty^{s,s'} = 
- \begin{pmatrix} \Psi^s_\infty \Psi^{s'}_\infty \bs{ V}^\tran Q^s_\infty \\
\Psi^s_\infty \Psi^{s'}_\infty \bs V \widetilde Q^s_\infty
\end{pmatrix} \, . 
\end{gather*} 
To proceed, we rely on a regularized version of this
equation. 
Denoting by $\1$ the constant function $\bs 1(x) = 1$ in $C([0,1])$, 
and letting $v = ( \bs 1, - \bs 1)^\tran \in \mathcal B$, the kernel operator
$v v^\tran$ on $\mathcal B$ is defined by the matrix 
\[
(v v^\tran)(x,y) = \begin{pmatrix} 
\bs 1(x) \bs 1(y) & - \bs 1(x) \bs 1(y) \\ 
- \bs 1(x) \bs 1(y) & \bs 1(x) \bs 1(y) \end{pmatrix} .
\]
By the constraint 
$\int Q^s_\infty = \int \widetilde Q^s_\infty$, it holds
that $(v v^\tran) \Delta \vec Q_\infty^{s,s'} = 0$. Thus, 
$\Delta \vec Q_\infty^{s,s'}$ satisfies the identity 
\begin{equation}
\label{DQss} 
\bigl( ( I - (\bs M_\infty^{s,s'})^\tran ) 
( I - \bs M_\infty^{s,s'}) + v v^\tran \bigr) 
\Delta \vec Q_\infty^{s,s'} = 
( I - (\bs M_\infty^{s,s'})^\tran )\bs a_\infty^{s,s'} \, . 
\end{equation} 
We rewrite the left side of this identity as 
$(I - \bs G^{s,s'}_\infty) \Delta \vec Q_\infty^{s,s'}$ where  
\[
\bs G_\infty^{s,s'} =
 \bs M_\infty^{s,s'} + (\bs M_\infty^{s,s'})^\tran 
 - (\bs M_\infty^{s,s'})^\tran \bs M_\infty^{s,s'} - v v^\tran \, , 
\]
and we study the behavior of $\bs M_\infty^{s,s'}$ and $\bs G_\infty^{s,s'}$
as $s'\to s$. 

Let $s\in (0,\sqrt{\rho(\bs V)})$ and $s'$ belong to a small compact neighborhood $\mathcal K$
of $s$. Then the first component of $\bs M_\infty^{s,s'} \vec f(x)$ has
the form 
\[
\int \bigl( h_{11}(x,y,s') f(y) 
   + h_{12}(x,y,s') \widetilde f(y)\bigr) \, dy \, , 
\]
where $h_{11}$ and $h_{12}$ are continuous on the compact set 
$[0,1]^2 \times \mathcal K$ by the previous results. A similar argument holds for the other component of $\bs M_\infty^{s,s'} \vec f(x)$.
By the uniform continuity of these functions on this set, we get that 
the family $\{ \bs M_\infty^{s,s'} \vec f \, : \, s' \in \mathcal K, 
\, \| \vec f \|_{\mathcal B} \leq 1 \}$ is equicontinuous, and by the Arzela--Ascoli theorem,
the family $\{ \bs M_\infty^{s,s'} \, : \, s' \in \mathcal K \}$ is 
collectively compact. Moreover, 
\[
\bs M_\infty^{s,s'} \xrightarrow[s'\to s]{str} 
M_\infty^s = \begin{pmatrix} I &0 \\0 & -I \end{pmatrix} 
N_\infty^s
\begin{pmatrix} I & 0\\0 & -I \end{pmatrix} \, , 
\]
where
\[
N_\infty^{s} = 
\begin{pmatrix} 
s^2 \Psi^2_\infty(\cdot,s) \bs{V}^\tran & 
       \Psi^2_\infty(\cdot,s) \widetilde\Phi^2_\infty(\cdot,s) \bs V  \\
  \Psi^2_\infty(\cdot,s) \Phi^2_\infty(\cdot,s) \bs{V}^\tran
& s^2 \Psi^2_\infty(\cdot,s) \bs V
\end{pmatrix} \, .  
\]
By a similar argument, 
$\{ \bs G_\infty^{s,s'} \, : \, s' \in \mathcal K \}$ is collectively compact,
and $\bs G_\infty^{s,s'} \xrightarrow[s'\to s]{str} G^s_\infty$, where 
\[
G_\infty^s = 
 M_\infty^{s} + (M_\infty^{s})^\tran - (M_\infty^{s})^\tran M_\infty^{s} 
 - v v^\tran \, . 
\]
We now claim that $1$ belongs to the resolvent set of the compact operator
$G_\infty^s$. 

Repeating an argument of the proof of Lemma 4.4 from \cite{cook2018non}, we can
prove that the Krein--Rutman eigenvalue of the strongly positive operator
$N_\infty^{s}$ is equal to one, and its eigenspace is generated by the vector
$\vec Q^s_\infty = 
\bigl( Q^s_\infty, \widetilde Q^s_\infty \bigr)^\tran$.  From the
expression of $M_\infty^s$, we then obtain that the spectrum of this compact
operator contains the simple eigenvalue $1$, and its eigenspace is generated
by the vector 
$\bigl( Q^s_\infty, - \widetilde Q^s_\infty \bigr)$. 

We now proceed by contradiction. If $1$ were an eigenvalue of $G_\infty^s$, there would exist a non zero vector
$\vec f \in \mathcal B$ such that $(I - G_\infty^s) \vec f = 0$, or, 
equivalently, 
\[
(I - (M_\infty^{s})^\tran ) ( I - M_\infty^{s}) \vec f 
+ v v^\tran \vec f = 0 \, .
\]
Left-multiplying the left hand side of this expression by $\vec f^\tran$ and 
integrating on $[0,1]$, 
we get that $( I - M_\infty^{s}) \vec f = 0$ and $\int f = \int \widetilde f$,
which contradicts the fact the $\vec f$ is collinear with 
$\bigl( Q_\infty(\cdot, s), - \widetilde Q_\infty(\cdot, s) \bigr)$. 

Returning to \eqref{DQss} and observing that 
$\{ \bs M_\infty^{s,s'} \, : \, s' \in\mathcal K \}$ is bounded, we get from
the convergence 
$(\bs M_\infty^{s,s'})^\tran \xrightarrow[s'\to s]{str} (M_\infty^s)^\tran$ 
that 
\[
( I - (\bs M_\infty^{s,s'})^\tran ) \bs a_\infty^{s,s'} 
\xrightarrow[s'\to s]{} 
( I - (M_\infty^{s})^\tran ) a_\infty^{s} \, , 
\]
where 
\[
a_\infty^s(\cdot) = 
- \begin{pmatrix} \Psi_\infty(\cdot,s)^2 \bs{V}^\tran Q_\infty(\cdot,s) \\
\Psi_\infty(\cdot,s)^2 \bs V \widetilde Q_\infty(\cdot,s)  
\end{pmatrix} \, .
\]
From the aforementioned results on the collectively compact operators, it holds
that there is a neighborhood of $1$ where 
$\bs G_\infty^{s,s'}$ has no eigenvalue for all $s'$ close enough to $s$
(recall that $0$ is the only possible accumulation point of the spectrum of 
$G_\infty^s$). Moreover, 
\[
(I - \bs G_\infty^{s,s'})^{-1} \xrightarrow[s'\to s]{str} 
(I - G^s_\infty)^{-1} \, . 
\]
In particular, for $s'$ close enough to $s$, the family 
$\{ (I - \bs G_\infty^{s,s'})^{-1} \}$ is bounded by the Banach-Steinhaus 
theorem. Thus, 
\begin{eqnarray*}
\Delta \vec Q_\infty^{s,s'} &\xrightarrow[s'\to s]{}& 
\bigl( (I - (M^s_\infty)^\tran) (I - M^s_\infty) + v v^\tran \bigr)^{-1} 
 (I - (M^s_\infty)^\tran) a_\infty^s \\
&& \quad = (\partial_{s^2} Q^s_\infty, 
  \partial_{s^2} \widetilde Q^s_\infty)^\tran \ . 
\end{eqnarray*}
Using this result, we straightforwardly obtain from the expression of 
$F_\infty$ that this function is differentiable on $(0, \sqrt{\rho(\bs V)})$.
The continuity of the derivative as well as the existence of a right
limit as $s\downarrow 0$ and a left limit as $s\uparrow \sqrt{\rho(\bs V)}$
can be shown by similar arguments involving the behaviors of the operators
$M_\infty^s$ and $G_\infty^s$ as $s$ varies. The details are skipped.

Since $\mu_n^Y \sim \mu_n$ in probability and since we have the straightforward
convergence $\mu_n \xrightarrow[n\to\infty]{w} \mu_\infty$, the statement
\eqref{cvg-muinfty} of the theorem follows. 

\section{Positivity of the density}
\label{sec:positivity}

In this section we prove Proposition \ref{prop:LB}, Theorem \ref{th:positive}
and Corollary \ref{prop:circlaw}.

\subsection{Proof of Proposition \ref{prop:LB}}
\label{app:proof-prop:LB}

Most of the work will go into showing that the limits $\lim_{t\downarrow
0}\rvec(0,t)$ and $\lim_{s\downarrow 0} \qvec(s)$ exist and are equal.  To that
end, we rely on some of the results of~\cite{AEKgram}, from which we start by
borrowing some notations. Given to sequences $(a_n)$ and $(b_n)$ of real
numbers, $a_n \lesssim b_n$ refers to the fact that there exists a constant
$\kappa>0$ independent of $n\ge 1$ such that $a_n \le \kappa\, b_n$. The
notation $a_n \sim b_n$ stands for $a_n \lesssim b_n$ and $b_n \lesssim a_n$.
Given a real vector $\bs{x}$, the notation $\min \bs{x}$ refers to the smallest element
of $\bs{x}$. 

\begin{lemma}[Lemmas 3.11, 3.13 and Eq. (3.56) of \cite{AEKgram}]
\label{ME0bound}
Let \ref{ass:sigmax} and \ref{ass:BFID} hold true, and recall that 
$\vec{\br}(0,t)$ is the unique positive solution of~\eqref{def:MEt} for $s=0$ and 
$t > 0$. Then, 
$$
1 \quad \lesssim\quad   \inf_{t \in (0, 10]} \min \vec{\br}(0,t)  
  \quad \leq  \quad 
 \sup_{t  > 0} \|\,  \vec{\br}(0,t)\|_{\infty} \quad \lesssim \quad  1\, .
$$
The limit 
$\rvec_0=\begin{pmatrix} \br_0 \\ \brt_0 \end{pmatrix} = 
 \lim_{t\downarrow 0} \vec{\br}(0,t)
$
exists and satisfies 
$1 \lesssim \min \rvec_0 \leq \| \rvec_0 \|_{\infty} \lesssim 1$. 
Moreover, writing $\br_0 = (r_{0,i})$ and $\brt_0 = ( \tilde r_{0,i} )$, it 
holds that 
\begin{equation}
\label{r00-sink} 
r_{0,i} (V_n \brt_0)_i=1, \quad \textrm{and}\quad 
 \tilde r_{0,i} (V_n^\tran \brt_0)_i=1\,, \quad i\in [n]\ .
\end{equation} 
\end{lemma}

\begin{prop}[Proposition 3.10 (ii) of \cite{AEKgram}] \label{prop:stabs0}
Let \ref{ass:sigmax} and  \ref{ass:BFID} hold. Suppose the functions 
$$
\vec{\bd}=\begin{pmatrix}\bd\\ \widetilde \bd \end{pmatrix} 
  =\begin{pmatrix} (d_i)_{i\in[n]} \\ (\tilde d_i)_{i\in[n]} \end{pmatrix} 
: \R^+ \to \C^{2n}, 
  \qquad\textrm{and}\qquad  
 \vec \bg  =\begin{pmatrix} \bg\\ \widetilde \bg \end{pmatrix} 
  =\begin{pmatrix} (g_i)_{i\in[n]} \\ (\tilde g_i)_{i\in[n]}  \end{pmatrix} 
    : \R^+ \to (\C\setminus\{ 0\})^{2n}
$$ 
satisfy
\begin{equation} 
\label{eq:pertME} 
\frac{1}{g_i(t)} = (V_n \bgt(t))_i + t + d_i(t)\ , 
\quad  \frac{1}{\widetilde g_i(t)} = (V_n^\tran \bg(t))_i + t + \widetilde d_i(t)\quad 
\textrm{and}\quad  
\sum_{i\in [n]} g_i(t) \ =\ \sum_{i\in [n]} \gt_i(t)  
\end{equation}
for all $t \in \R^+$. Then, there exist $\lambda^* > 0$ and $C > 0$, depending 
on $V$, such that
\[ \| \vec \bg(t) - \vec  \br(0,t) \|_{\infty}\ \bs{1}_{\left\{  \|\vec  \bg(t) - \vec \br(0,t) \|_{\infty} \leq \lambda^*  \right\}} \quad \leq\quad  C \| \vec  \bd(t)\|_{\infty}  
\qquad \textrm{for all}\quad  |t| < 10\, .
\]
\end{prop}

Let us outline the proof of Proposition~\ref{prop:LB}--\eqref{item1:LB}. 
Lemma~\ref{ME0bound} shows that $\vec{\br}(0,t)$
converges as $t\downarrow 0$.  In parallel, we know from
Theorem~\ref{thm:master}--\eqref{q=limr} that for each $s > 0$, it holds that
$\vec{\br}(s,t) \to_{t\downarrow 0} \qvec(s)$ under the irreducibility 
assumption, which is implied by~\ref{ass:BFID}. To prove that 
$\qvec(s) \to_{s\downarrow 0} \vec{\br}_0$, we fix $s > 0$ small enough and 
find a sequence $t_k \downarrow 0$ such that 
$\| \rvec(s, t_k) - \rvec(0, t_k) \|_\infty \leq \text{Constant} \times s^2$.  
This inequality will be established iteratively on $k$. Specifically, we start 
with a $t_0$ large enough so that the inequality is satisfied, then we 
apply a bootstrap procedure on $k$, controlling 
$\| \rvec(s, t_k) - \rvec(0, t_k) \|_\infty$ at each step with the help of 
Proposition~\ref{prop:stabs0} with $\vec \bg(t) = \vec{\br}(s,t)$. We now 
begin the proof.

\begin{proof}[Proof of Proposition \ref{prop:LB}]
 Letting $\vec \bg(t) = \vec{\br}(s,t)$, we get 
from~\eqref{def:MEt} that $\vec \bg(t)$ satisfies~\eqref{eq:pertME} with 
$$
 d_i(s,t) = \frac{s^2}{((V_n^\tran \br(s,t))_i + t} 
 \quad\text{and}\qquad  
 \widetilde d_i(s,t) = \frac{s^2}{((V_n \widetilde\br(s,t))_i + t}\, .
 $$
We now start our iterative procedure by choosing properly the initial value
$t_0$. 
Using the bound $\|\vec \br(0,t)\|_\infty \leq t^{-1}$ and 
$\|\vec \br(s,t)\|_\infty \leq t^{-1}$ from \eqref{def:MEt}, and
$\|  \vec \bd(s,t) \|_{\infty}  \leq s^2 t^{-1}$ we get that for $t_0$ 
sufficiently large, 
$\|\vec \br(s,t_0) - \vec \br(0,t_0) \|_{\infty} \leq \lambda^* $ and
thus Proposition \ref{prop:stabs0} gives the bound 
\begin{equation} 
\label{eq:initbound} 
 \|\vec \br(s,t_0) - \vec \br(0,t_0) \|_{\infty} \leq C s^2 t_0^{-1}.
\end{equation}
We now fix this $t_0$ and let 
$K = \sup_{0 < t < t_0} \|\rvec(0,t) \|_{\infty}$, which is finite by
Lemma~\ref{ME0bound}. We also introduce $\ell^*,s^*>0$ such that 
\begin{equation}\label{tuned-l-s}
\ell^* \quad \le\quad \min\left( \lambda^* \, ,\,  \frac 1{2\sigma^2_{\max}K}\right)
\qquad \textrm{and}\qquad (s^*)^2 \quad \le\quad  
 \min\left( \frac {\ell^*}{8CK}\, ,\, \frac {t_0 \ell^*}{4C} \right)\, .
\end{equation}
Fix $s$ such that $0<s<s^*$. From the choice of $s^*$ and 
\eqref{eq:initbound}, we get that 
$$
 \|\vec \br(s,t_0) - \vec \br(0,t_0) \|_{\infty} \quad \le\quad  \frac{\ell^*}4\, .
$$

By Lemma \ref{ME0bound} and Theorem~\ref{thm:master}--\eqref{q=limr}, the 
functions $t\mapsto \vec \br(0,t)$ and $t\mapsto \vec \br(s,t)$ extend continuously 
to $t=0$ and hence are uniformly continuous on the compact interval
$[0,t_0]$. Thus, there exists $\eta>0$ such that for $0\leq t,t' \leq t_0$ and 
$|t-t'|\le \eta$, we have 
$$
\| \vec{\br}(0,t) - \vec{\br}(0,t')\|_\infty \le \frac{\ell^*}4\, ,\ \| \vec{\br}(s,t) - \vec{\br}(s,t')\|_\infty \le \frac{\ell^*}4\ ,\
\left| (V^\tran\br(s,t))_i + t -(V^\tran\br(s,t'))_i-t' \right| \le \frac 1{4K}\, .  
$$
Consider a sequence of real numbers $(t_k)_{k\ge 0}$ such that $t_k\downarrow
0$ and $|t_{k+1}-t_k|<\eta$ for $k\ge 0$. We shall prove inductively that  
\begin{equation}
\|\vec{\br}(s,t_k) - \vec{\br}(0,t_k)\|_\infty \quad \le\quad  \frac {\ell^*}4\, .
\end{equation}
Using the uniform continuity and the inductive assumption, we obtain
\begin{eqnarray}
\lefteqn{\|\vec{\br}(s,t_{k+1}) - \vec{\br}(0,t_{k+1})\|_\infty} \nonumber \\
&\le& \|\vec{\br}(s,t_{k+1}) - \vec{\br}(s,t_k)\|_\infty +\|\vec{\br}(s,t_k) - \vec{\br}(0,t_k)\|_\infty +\|\vec{\br}(0,t_k) - \vec{\br}(0,t_{k+1})\|_\infty\ ,\nonumber \\
&\le & \frac {\ell^*}4 +  \frac {\ell^*}4 +  \frac {\ell^*}4\quad <\quad \ell^* \quad <\quad \lambda^* \ ,\label{inductive-estimate}
\end{eqnarray}
thus, Proposition \ref{prop:stabs0} leads to the bound 
\[  \|\vec \br(s,t_{k+1}) - \vec \br(0,t_{k+1}) \|_{\infty} \quad \leq\quad  C \|  \vec \bd(s,t_{k+1}) \|_{\infty}\, . \]
We now upper bound $\|  \vec \bd(s,t_{k+1}) \|_{\infty}$. We have:
\begin{eqnarray*}
(V_n^\tran \br(s,t_{k+1}))_i + t_{k+1} & \ge&  (V_n^\tran \br(0,t_{k+1}))_i + t_{k+1} - \left(((V_n^\tran \br(0,t_{k+1}))_i-(V_n^\tran \br(s,t_{k+1}))_i\right)\ , \\
&\stackrel{(a)}\ge& (V_n^\tran \br(0,t_{k+1}))_i + t_{k+1} - \sigma^2_{\max} \ell^*\ , \\
&\stackrel{(b)}=& \frac 1{r_i(0,t_{k+1})} - \sigma^2_{\max} \ell^*\quad\ge\quad  \frac 1K - \sigma^2_{\max} \ell^*\ ,\\
&\stackrel{(c)}\ge & \frac 1{2K}\ ,
\end{eqnarray*}
where $(a)$ follows from \eqref{inductive-estimate}, $(b)$ from the system satisfied by $\vec{\br}(0,t_{k+1})$ and $(c)$ from the constraint \eqref{tuned-l-s} of $\ell^*$.
We finally end up with the estimation $\| \vec{\bd}(s,t_{k+1})\|_\infty\le 2K s^2$. Applying Proposition \ref{prop:stabs0} together with \eqref{inductive-estimate}, we obtain
\begin{eqnarray*}
\|\vec{\br}(s,t_{k+1}) - \vec{\br}(0,t_{k+1})\|_\infty &\le&  C \| \vec{\bd}(s,t_{k+1})\|_\infty\quad \le\quad  2CK s^2\quad \stackrel{(a)}\le \quad \frac{\ell^*}{4}\, , 
\end{eqnarray*}
where $(a)$ follows from the fact that $s<s^*$ and the constraint \eqref{tuned-l-s} on $s^*$. Hence the induction step is verified. As a byproduct of the induction,
we have, after taking $t_k\downarrow 0$, 
\begin{equation}
\label{q(s)-r0} 
\forall s\in (0,s^*)\, ,\quad \| \vec{\bq}(s) - \rvec_0 \|_\infty \le 2CK\, s^2
\end{equation} 
and in particular, $\qvec(s)$ converges to $\qvec(0) = \rvec_0$ as 
$s\downarrow 0$. 

Combining $q_i(0)(V\qtilde(0))_i=1$ and $\widetilde q_i(0)(V^\tran \q(0))_i=1$ with the definition of $\mu_n$, we obtain
\[
\mu_n(\{0\})= 1 - \lim_{s\downarrow 0} \frac{1}{n} \langle \bq(s),  V \bqt(s) \rangle\ = 1 -  \frac{1}{n} \sum_{i\in [n]} q_i(0)  (V \bqt(0))_i  =0\ .
\]
Proposition \ref{prop:LB}-\eqref{item1:LB} is proven. 

We now turn to Proposition~\ref{prop:LB}-\eqref{item2:LB}. To establish the
existence of the limit of $f(z)$ as $z\to 0$, we first show that
$\partial_{s^2} \qvec(s)$ can be continuously extended to $s = 0$ as
$s\downarrow 0$. This can be done by considering \cite[Lemma 4.4]{cook2018non}. 
Using the shorthand notation $\Psi(s) =\Psi(\bs{\vec{q}},s,0)$ from 
\eqref{def:Psi}, let us define 
\begin{gather*} 
M(s) = \begin{pmatrix} 
s^2 \Psi(s)^2 V^\tran  & - \diag(\bq(s))^2 V  \\
- \diag(\bqt(s))^2 V^\tran  & s^2 \Psi(s)^2 V
\end{pmatrix}, \\ 
A(s) = \begin{pmatrix} I - M(s) \\ 
 ( \bs 1_n^\tran  \ \ - \bs 1_n^\tran ) 
\end{pmatrix}  
\in \R^{(2n+1) \times  2n} ,  
\quad \text{and} \quad 
{b}(s) = - \begin{pmatrix} \Psi(s) \bq(s) \\ 
  \Psi(s) \bqt(s) \\ 
  0 
  \end{pmatrix} \in \R^{2n+1} . 
\end{gather*} 
Then, it is shown in \cite[Lemma 4.4]{cook2018non} that $A(s)$ is a full
column-rank matrix for $s \in (0, \sqrt{\rho(V)})$, and that 
$\partial_{s^2} \qvec(s) = A(s)^{-\text{L}} b(s)$, where $A(s)^{-\text{L}}$ is 
the left inverse of $A(s)$. Now, the important observation here is that if we 
make $s\downarrow 0$, then $A(s)$ converges to the full column-rank matrix 
\[
A(0) = \begin{pmatrix} I - M(0) \\ 
 ( \bs 1_n^\tran  \ \ - \bs 1_n^\tran ) 
\end{pmatrix} , \quad \text{with} \quad  
M(0) = 
\begin{pmatrix} 
 0 &  - \diag(\bq(0))^2 V \\
 - \diag(\bqt(0))^2 V^\tran & 0   
\end{pmatrix} . 
\]
The convergence to $A(0)$ is an immediate consequence of the 
convergence of $\qvec(s)$ that we just established, and of 
Lemma~\ref{ME0bound}. To show that $A(0)$ is full column-rank, consider 
the matrix non-negative matrix $N = - M(0)$.
We show that $\qvec(0)$ is the unique eigenvector of $N$, up to scaling, such 
that $N \qvec(0) = \qvec(0)$.
For any non zero vector 
$\vec{\bs x} = \begin{pmatrix} \bs x \\ \bs{\tilde x} \end{pmatrix}$ 
such that $\vec{\bs x} = N \vec{\bs x}$, we have 
\begin{align} 
\diag(\bq(0)) V \diag(\bqt(0))  \diag(\bqt(0))^{-1} \bs{\tilde x} &= 
 \diag(\bq(0))^{-1} \bs x,  \quad\text{and} \nonumber \\ 
\diag(\bqt(0)) V^\tran \diag(\bq(0))  \diag(\bq(0))^{-1} \bs{x} &= 
 \diag(\bqt(0))^{-1} \bs{\tilde x} , 
\label{xtilde=qtilde} 
\end{align} 
thus, writing $Q = \diag(\bq(0)) V \diag(\bqt(0))^2 V^\tran \diag(\bq(0))$,
we get that 
\begin{equation}
\label{pf-Q} 
Q  \diag(\bq(0))^{-1} \bs x =  \diag(\bq(0))^{-1} \bs x . 
\end{equation} 
We know from Proposition~\ref{prop:LB}--\eqref{item1:LB} that 
$Q$ is doubly stochastic (see also Remark~\ref{rem:sinkhorn}). 
Moreover, since $V$ is fully indecomposable, $Q$ is also 
fully indecomposable, see, \emph{e.g.} \cite[Theorem 2.2.2]{bap-rag-(book)97}. 
Thus, it is irreducible, which implies that the only non zero vectors $\bs x$ 
that satisfy~\eqref{pf-Q} take the form $\bs x = \alpha \bq(0)$ for 
$\alpha \neq 0$. Plugging this identity into~\eqref{xtilde=qtilde}, we also get
that $\bs{\tilde x} = \alpha \bqt(0)$, which shows that $\vec{\bs x}$ exists
and is equal to $\alpha \qvec(0)$. 
 
As a consequence, the right null space of the matrix 
$I - M(0)$ is spanned by the vector $\begin{pmatrix} \bq(0) \\
 - \bqt(0) \end{pmatrix}$. Since the inner product of the last row of 
$A(0)$ with this vector is non zero, $A(0)$ is full column-rank. 
By the right continuity of $A(s)$ and $b(s)$ at zero and the fact that 
$A(s)$ is full column-rank on $[0, \sqrt{\rho(V)})$, we conclude that  
$\partial_{s^2} \qvec(s)$ can be continuously extended to 
$s = 0$ as $s\downarrow 0$. 

Now, from the expression~\eqref{eq:density} of the density and 
Equations~\eqref{def:ME}, we have for $|z|$ near zero 
\begin{align} \label{eq:densityexatzero}
f_n(z) &= - \frac 1{2\pi n |z|} 
  \frac{d}{ds} \langle \bq(s),  V \bqt(s) \rangle \Big|_{s=|z|} 
 = - \frac 1{\pi n} 
  \frac{d}{ds^2} \langle \bq(s),  V \bqt(s) \rangle \Big|_{s=|z|} \\
&= - \frac 1{\pi n} \sum_{i\in[n]} \partial_{s^2} 
 \frac{(V_n\bqt(s))_i(V_n^\tran {\bq}(s))_i} 
  {s^2+(V_n\bqt(s))_i(V_n^\tran\bq(s))_i}  \Big|_{s=|z|}\nonumber \\
&= \frac 1{\pi n} \sum_{i\in[n]} 
 \frac{(V_n\bqt(|z|))_i(V_n^\tran {\bq}(|z|))_i - 
 |z|^2 
  \partial_{s^2} \left((V_n\bqt(s))_i(V_n^\tran {\bq}(s))_i\right) |_{s=|z|}}  
  {\left(|z|^2+(V_n\bqt(|z|))_i(V_n^\tran\bq(|z|))_i\right)^2}  \nonumber
\end{align} 
Since $\|\partial_{s^2} \qvec(s) \|_\infty$ is bounded near zero by what we
have just shown, it is easily seen that 
\[
|z|^2 
\partial_{s^2} \left((V_n\bqt(s))_i(V_n^\tran {\bq}(s))_i\right) |_{s=|z|} 
 \xrightarrow[z\to 0]{} 0 .
\]
We therefore get that 
\[
f_n(z) \xrightarrow[z\to 0]{} 
 \frac 1{\pi n} \sum_{i\in[n]} 
 \frac{1}{(V_n\bqt(0))_i(V_n^\tran\bq(0))_i} 
\]
as well as the inequalities~\eqref{eq:unif-bounds-density-0}
by using Lemma~\ref{ME0bound} again, which completes the proof of 
Proposition~\ref{prop:LB}-\eqref{item2:LB}. 
\end{proof}

\subsection{Proof of Theorem \ref{th:positive}} 
\label{proof:th-positive}
The positivity of the density has been established under Assumptions
\ref{ass:sigmax} and \ref{ass:sigmin} in \cite[Lemma 4.1]{alt2018local}. We
will follow a similar strategy. The proof of \cite[Lemma 4.1]{alt2018local} relies on two crucial steps: the
existence and regularity of solutions to the master equations \eqref{def:ME},
and an expression for the density \eqref{eq:density} in terms of a certain
operators whose spectrum can be controlled. In \cite[Section 5]{cook2018non},
the first step is established, as long as $|z|$ is away from $0$, under the
more general Assumption \ref{ass:expander}. Following the calculations from
\cite{alt2018local}, we now carry out the second step, occasionally referring
the reader to \cite{alt2018local} for details. We note that while the calculations can be closely followed, the weaker assumptions on the variance profile $V$ introduces new complications. 

In all this section, we follow the notational convention of \cite{alt2018local}
stating that if $\bs u = (u_i)$ and $\bs v = (v_i)$ are $n\times 1$ vectors,
then $\frac 1{\bs u}$ is the vector $(\frac 1{u_i})_{i\in [n]}$, 
$\sqrt{\bs u}= (\sqrt{u_i})_{i\in [n]}$, $\bs u\bs v= (u_iv_i)_{i\in [n]}$, 
and so on. 

In what follows, ${\mathcal O}(t)$ refers to error terms that are bounded in
magnitude by $C t $ for small $t$, where the constant $C$ can depend on $n$ or
on $|z|$.  We use the notation $a(t)\lesssim b(t)$ if there exists a constant
$C$ that might depend on $n$ or on $|z|$, such that $ a(t) \leq C b(t)$. The
notation $a(t)\sim b(t)$ refers to $a(t)~\lesssim~b(t)~\lesssim~a(t)$.

\begin{proof}[Proof of Theorem \ref{th:positive}]

We now prove part (1), in particular in this section we will always assume
Assumption \ref{ass:admissible} holds and that $s = |z|^2$ is in the interval
$(0, \sqrt{\rho(V)})$. As mentioned in the introduction, we will prove a lower bound that depends on $\bs q$ and $\bqt$. By Proposition \ref{prop:LB}, we have that under Assumption \ref{ass:BFID} these vectors are continuous in a neighborhood of $0$, therefore can continuously extend our lower bound to zero and match it with the bound in the previous section, ensuring the lower bound stays away from 0 for all $z$ in the support, verifying  part (2).

We start with the expression of the density in~\eqref{eq:density}.  
In what follows it will be more convenient to work on the regularized master
equations provided by the system~\eqref{def:MEt} rather than those given 
by the system~\eqref{def:ME}, recalling from 
Theorem~\ref{thm:master}--\eqref{q=limr} that 
$\qvec(s) = \lim_{t\downarrow 0} \rvec(s,t)$ for $s>0$. 
In \cite[Section 7]{cook2018non}, it is indeed proven that we can switch 
$d/ds^2$ and $\lim_{t\downarrow 0}$, and write
\[
f_n(z) = 
  - \frac 1{\pi n} \frac{d}{ds^2} \left(\lim_{t\downarrow 0} 
   \langle \br(s,t),  V \brt(s,t) \rangle\right) \Big|_{s=|z|} = 
 - \frac 1{\pi n} \lim_{t\downarrow 0} 
  \frac{d}{ds^2} \langle \br(s,t),  V \brt(s,t) \rangle \Big|_{s=|z|}. 
\]
Introducing the notation 
\[
 \bphi(s,t) = V \brt(s,t) + t\ ,\quad 
\bphit(s,t) = V^\tran \br(s,t) + t, 
 \qquad\textrm{and} \qquad 
  \phivec(s,t) =\begin{pmatrix} \bphi(s,t) \\ \bphit(s,t) 
 \end{pmatrix} , 
\] 
we can rewrite the expression of the density as 
\[
f_n(z) = 
 - \frac 1{\pi n} \lim_{t\downarrow 0} 
  \langle \phivec(s,t), \frac{d}{ds^2} \rvec(s,t) \rangle \Big|_{s=|z|}. 
\]

We now use the shorthand $\Psi(s,t)=\Psi(\vec{\bs{r}}(s,t),s,t)$ from \eqref{def:Psi}
and let 
\[
\bs\Psi(s,t) = \begin{pmatrix} \Psi(s,t) \\ & \Psi(s,t) \end{pmatrix} , 
\quad
\vec{\brt}(s,t) =
 \begin{pmatrix} \bs{\widetilde r}(s,t)\\ \bs{r}(s,t) \end{pmatrix}.\] 
In what follows we will often drop the dependence on $s$ and $t$. In expressions with $t$ taken to zero we will use $\bs{q}$ instead of $ \bs{r}$.
With this notation, we reformulate ~\eqref{def:MEt} as
\begin{equation}
\label{eq:phivec} 
\phivec(s,t) = \bs\Psi(s,t)^{-1} \vec{\brt}(s,t) .
\end{equation} 
We now turn to the derivative $d \rvec(s,t) / ds^2$. A straightforward adaption of \cite[Lemma 4.4]{cook2018non} 
with $\qvec(s)$ replaced by $\rvec(s,t)$ yields:
\begin{equation}
\frac{d}{ds^2} \rvec(s,t) = \bs A(s,t)^{-\text{1}} \bs{b}(s,t).  
\end{equation}  
where  
\begin{gather*} 
\bs M(s,t) = \begin{pmatrix} 
s^2 \Psi(s,t)^2 V^\tran  & - \diag(\br(s,t)^2) V  \\
- \diag(\widetilde{\br}(s,t)^2) V^\tran  & s^2 \Psi(s,t)^2 V
\end{pmatrix}, \\ 
\bs A(s,t) = I - \bs M(s,t) 
\in \R^{2n \times  2n} ,  
\quad \text{and} \quad 
\bs{b}(s,t) = - \bs\Psi(s,t) \rvec(s,t) 
\in \R^{2n} . 
\end{gather*}
We note that from \cite{cook2018non}, $\bs A(s,t)$ is invertible.

In \cite{alt2018local}, a fine analysis of the spectrum of $\bs A(s,t)$ is done
for the purpose of establishing an optimal local law on the eigenvalues of 
$Y_n$. Here we borrow some of the results of \cite{alt2018local} in order to
control the inverse of this matrix. 
Following the proof of \cite[Lemma 4.1]{alt2018local}, the matrix $\bs A(s,t)$ 
can be factored as 
\begin{equation}
\label{eq:A-Alt}   
\bs A(s,t) = \bs W ( I- \bs T\bs F) \bs W^{-1}, 
\end{equation} 
where $\bs W, \bs T$ and $\bs F$ are the $2n\times 2n$ symmetric matrices 
given as 
\begin{gather*}
 \bs T = \bs \Psi^{-1} \begin{pmatrix} - \diag(\br \brt) & s^2 \Psi^2 \\  
     s^2 \Psi^2 & - \diag(\br \brt) \end{pmatrix} , \ 
 \bs W =\begin{pmatrix} W &  \\
 & \widetilde W \end{pmatrix} , \  
\bs F = \begin{pmatrix} &   
 W V \widetilde W \\
     \widetilde W  V^\tran W \end{pmatrix} 
  = \begin{pmatrix}  & F \\ F^\tran \end{pmatrix} , 
\\ 
 W = \sqrt{\diag \left(\frac{ \br}{ \brt}\right) \Psi} , \qquad\text{and}\qquad 
{\widetilde W} = \sqrt{\diag \left(\frac{ \brt}{ \br}\right) \Psi} .  
  \end{gather*}
We note that $\bs T, \bs F, \bs W$ each depend on $s,t$ but we omit the
notation for readability. 
From Equations \eqref{eq:phivec}--\eqref{eq:A-Alt}, we have 
\begin{align} 
\label{eq:dens1}
f_n(z)&=
\lim_{t \to 0} \frac{1}{\pi n} 
\left\langle {\bs \Psi}^{-1}  \vec{\brt} ,  
\bs W (I-\bs T \bs F)^{-1}  \bs W^{-1} 
  \bs \Psi \vec{\br} \right\rangle
&=  \lim_{t \to 0} \frac{1}{\pi n} 
   \left\langle \sqrt{\vec{\bs{ r}  } \vec{\brt}  }, 
    \bs \Psi^{-1/2}  (I-\bs T \bs F)^{-1} \bs \Psi^{1/2}  
   \sqrt{\vec{\bs{ r}  } \vec{\brt}  } \right\rangle . 
\end{align}

In order to exploit this decomposition, the will need the following lemmas,
which all hold under the assumptions of
Theorem~\ref{th:positive}--\eqref{th:positive-i1}. 

\begin{lemma} 
\label{lemma:boundsonq} 
$r_i(s,t) \sim 1$ and $\rt_i(s,t) \sim 1$ uniformly in $i \in [n]$.  
\end{lemma} 
\begin{proof} 
Under \ref{ass:admissible}, the average of $\bs{r}$ is bounded. Since each term
is positive, we trivially have each term is bounded by an ($n$-dependent)
constant. For the ($n$-dependent) lower bounds on $r_i$ and $\widetilde r_i$,
we refer to \cite[Eq. (5.17) and (5.31)]{cook2018non}.  
\end{proof}

The following two lemmas provide control on the spectrum of the symmetric
operators $\bs T$ and $\bs F$.  While the proofs appeal to arguments from
\cite{alt2018local}, we point out that we only use the parts of their theorems
that hold without that work's assumption of \ref{ass:sigmin}.

\begin{lemma} 
\label{lemma:altT} 
Let $s > 0$ and $t\in (0, 1)$. Then, there exists a constant $\varepsilon > 0$
such that the spectrum $\spec(\bs T)$ of $\bs T$ satisfies 
\[ 
\min(\spec({\bs T})) = -1 \qquad\text{and}\qquad 
\spec(\bs T) \subset \{-1\} \cup (-1+\varepsilon,1-\varepsilon) 
\] 
Moreover, the eigenspace for the eigenvalue $-1$ is the span of all vectors of 
the form $(-\bs y^\tran ,\bs y^\tran)^\tran$. 
\end{lemma}
This lemma follows from the definition of $\bs T$, \eqref{def:MEt}, and the
bound in Lemma \ref{lemma:boundsonq}, see \cite[Lemma 3.6]{alt2018local} for
details.  

The following lemma gives bounds on the spectrum of $\bs F$. Unlike in
\cite{alt2018local}, our assumptions on $V$ do not imply the matrix $\bs F$ is
irreducible, but we will not need its Perron-Frobenius subspace to be
one-dimensional. Although we will use that the vector $\bs\Psi^{-1/2}
\sqrt{\rvec \vec{\brt}}$ is near this Perron-Frobenius subspace. In particular
in the following lemma, we compute the ``correction'' term.
\begin{lemma} 
\label{lemma:altF} 
Let $s > 0$ and $t\in (0, 1)$. There exists a $ c_t \sim t $ such that 
$\| \bs{F}\| = 1 - c_t$. 
Let $\mathcal{V}$ be the subspace spanned by all eigenvalues with magnitude 
greater than $1 - C t$ for some $C > 0$. Then for all $t$ sufficiently small,  
$\| \bs F|_{\mathcal{V}^{\perp}} \| \leq 1 - \eps$, for some small $\eps$. 
Moreover, there exists an eigenvector $f_{-}$ such that 
\begin{equation} 
\label{eq:Feigenvector}
\bs{F} f_{-} = - \|\bs{F}\|f_{-}, \quad \text{and} \quad 
f_{-} =  \bs \Psi^{-1/2} \sqrt{\rvec \vec{\brt}}\bs e_{-} + \bs\varepsilon(t),
\end{equation}
where $\bs e_{-} =\begin{pmatrix} 1 \\  -1 \end{pmatrix}$, and 
$\| \bs\varepsilon(t) \| = {\mathcal O}(t)$. 
Finally, it holds that 
\begin{equation}
\label{I+F} 
( I + \bs F)^{-1} \left( \bs\Psi^{-1/2} \sqrt{\rvec \vec{\brt}} 
 - \frac t2 \bs W \bs 1 \right) = 
 \frac 12 \bs\Psi^{-1/2} \sqrt{\rvec \vec{\brt}} . 
\end{equation} 
\end{lemma}

\begin{proof}
The bound on the norm and the spectral gap can be obtained by combining 
Lemma~\ref{lemma:boundsonq} with the proof of \cite[Lemma 3.4]{alt2018local}, in particular \eqref{eq:Feigenvector} follows from (3.45) and (3.46) in \cite{alt2018local}. 
Let us verify~\eqref{I+F}. By direct calculation, using Equation~\eqref{eq:phivec} along with the expression of $\bs W$, we have 
\begin{equation}\label{eq:Feigv}
\bs F \bs\Psi^{-1/2} \sqrt{\rvec \vec{\brt}} 
= \bs W \begin{pmatrix} V \bs{\rt} \\ V^\tran \bs r \end{pmatrix} 
 = \bs W \left( \phivec - t \bs 1 \right) = \bs W \left( \bs\Psi^{-1} \vec\brt - t \bs 1 \right)  
 = \bs\Psi^{-1/2} \sqrt{\rvec \vec{\brt}} - t \bs W \bs 1 . 
\end{equation}
Thus, 
\[
( I + \bs F) \bs\Psi^{-1/2} \sqrt{\rvec \vec{\brt}} = 
  2 \bs\Psi^{-1/2} \sqrt{\rvec \vec{\brt}} - t \bs W \bs 1 , 
\]
and applying $(I+ \bs F)^{-1}$ to both sides of this equation, we 
obtain~\eqref{I+F}. 
\end{proof}

We can now manipulate \eqref{eq:dens1}, the expression for the density.
Following \cite{alt2018local}, the technique is based on a factorization of the
term $I - \bs \Psi^{-1/2}  \bs T \bs F \bs  \Psi^{1/2}$. One of the factors
will be dealt with by means of the identity~\eqref{I+F}. In order to be able to
use this identity, we shall have to inject the ``correction'' term $0.5t \bs W
\bs 1$ into the expression~\eqref{eq:dens1} of the density. The following lemma
shows that this can be done safely. 
\begin{lemma} \label{correct(I+F)} 
 $\Bigl| \Bigl\langle \bs\Psi^{1/2} \bs W \bs 1 \, , \, 
   \bs \Psi^{-1/2}  (I-\bs T \bs F)^{-1} \bs \Psi^{1/2}  
  \sqrt{\rvec \vec{\brt}} \Bigr\rangle \Bigr| \lesssim 1$. 
\end{lemma}

Before giving the proof, we state several technical lemmas, from which the above Lemma will immediately follow. The first step is to define the subspace on which the inverse $(I-\bs T \bs F)^{-1}$ is not bounded.  

\begin{lemma}
Let $\it{V}_{-1}$ be spanned by
eigenvectors of $\bs F$ with eigenvalues in $(-1, -1 + Ct]$, that are
additionally of the form $\begin{pmatrix} \bs x \\  - \bs x \end{pmatrix} +
\vec{\bs w}$, where $ \| \vec{\bs w} \| < 2 \| \bs\varepsilon(t) \|$ and $C$ and $\bs\varepsilon(t)$ are from in Lemma~\ref{lemma:altF}. Then the subspace $\it{V}_{-1}$ is spanned by $f_-$.

\end{lemma}

\begin{proof}

From Lemma \ref{lemma:altF}, we have that $f_{-}$ is an eigenvector of $\bs F$,
within an $\| \bs\varepsilon(t) \|$ distance of $ \bs\Psi^{-1/2} \sqrt{\rvec
\vec{\brt}}\bs e_{-}$. Now we show $f_{-}$ spans $\it V_{-1}$.  Let 
$\vec{\bs y} = \begin{pmatrix} \bs y \\  - \bs y \end{pmatrix} +  \begin{pmatrix} \bs w \\  \bs{ \tilde w} \end{pmatrix} 
\in \it V_{-1}$ be a unit vector. The block structure of $\bs F$, then implies 
$F \bs y = \bs y + \bs w + F \bs{ \tilde w}  $. The irreducible matrix $F$ has non-negative entries, with norm $1 -c_t$ and and spectral radius also tending to 1 as $t \to 0$. Additionally $\bs y$, up to an $4 \| \bs\varepsilon(t) \|$ error, saturates this norm bound, so we 
must have that $\bs y = \bs y_1 + \bs y_2$, where the entries of $\bs y_1$ have the same sign and $\|\bs y_2 \| = C_1 \| \bs\varepsilon(t) \|$. Otherwise, setting the entries equal to their absolute values would give a bigger norm.
Finally, as the vectors $f_-$ and $\vec{\bs y} $ are both
$C_1 \| \bs\varepsilon(t) \|$ away from vectors who each have the same sign, we conclude
they cannot be orthogonal for all small $t$, and therefore $f_-$ spans $\it{V_{-1}}$.
\end{proof}

To prove Lemma \ref{correct(I+F)}, we will use the following identity to bound $ ( I -  \bs F \bs T  )^{-1} \bs W \bs 1$:\begin{equation} \label{eq:IFTexp} ( I -  \bs F \bs T  )^{-1} \vec{\bs x}   = \frac{1}{2}  \vec{\bs x}    +   ( I -  \bs F \bs T  )^{-1} \left(\frac{  \bs F \bs T \vec{\bs x} + \vec{\bs x}}{2} \right)    \end{equation}
or any vector $ \vec{\bs x}  $. We will apply this identity with $ \vec{\bs x}  = \left(\frac{  \bs F \bs T + I}{2} \right)^k \bs W \bs 1$, for $k$ a non-negative integer. We now bound the inner product of the final term  and $f_-$. Afterwards, we show this is an effective bound.

\begin{lemma} \label{lem:Fproj} For any positive integer $k$,

\begin{align} \label{eq:fmproj}
\left| \left\langle f_-, \left( \frac{\bs F \bs T + I }{2} \right)^k \bs W \bs 1 \right\rangle   \right| 
&\leq   
\left| \left\langle f_-, \left( \frac{\bs F \bs T + I }{2} \right)^{k-1} \bs W \bs 1 \right\rangle  \right| 
+ \| \bs\varepsilon(t) \|    \left\| \left( \frac{\bs F \bs T + I }{2} \right)^{k-1} \bs W \bs 1  \right\| \\
& \leq | \langle     f_-, \bs W \bs 1 \rangle | +  \| \bs\varepsilon(t)    \|  \sum_{j=0}^{k-1} \left\| \left( \frac{\bs F \bs T + I }{2} \right)^{j} \bs W \bs 1 \right\| \,. \nonumber
\end{align}
Furthermore,
\[ | \langle   f_-, \bs W \bs 1 \rangle  |  \leq \|\bs\varepsilon(t)\| \|\bs W\| \,.  \]

\end{lemma}

\begin{proof} We will prove the inequality in the first line of \eqref{eq:fmproj}, the second line follows by inductively applying the first line.

\begin{align*}
\left\langle f_-, \left( \frac{\bs F \bs T + I }{2} \right)^k \bs W \bs 1 \right\rangle &= 
 \left\langle \left( \frac{\bs T \bs F + I }{2} \right) f_-, \left( \frac{\bs F \bs T + I }{2} \right)^{k-1} \bs W \bs 1 \right\rangle \\
&= \|\bs{F}\| \left\langle f_-, \left( \frac{\bs F \bs T + I }{2} \right)^{k-1} \bs W \bs 1 \right\rangle +
 \|\bs{F}\|\left\langle \left( \frac{  I - \bs T }{2} \right) \bs\varepsilon(t)  , \left( \frac{\bs F \bs T + I }{2} \right)^{k-1} \bs W \bs 1 \right\rangle 
\end{align*}
where we use that 
\[\bs T \bs F f_- =  -\|\bs{F}\| \bs T  f_- =  \|\bs{F}\| f_- +  \|\bs{F}\| (I - \bs T ) \bs\varepsilon(t) \] 
then the desired inequality follows by applying the Cauchy-Schwarz inequality to the second term. The inner product between $\bs W \bs 1 $ and $f_-$ is bounded using~\eqref{eq:Feigenvector} along with the identity\\  $\sum r_i = \sum \tilde r_i $:
\[
| \langle  \bs W \bs 1,  f_- \rangle  | =  
|   \langle \bs r, 1  \rangle -  \langle \brt  , 1 \rangle + \langle  \bs W \bs 1, \bs\varepsilon(t)  \rangle  | \leq \|\bs\varepsilon(t)\| \|\bs W\| .    
\]
\end{proof}

We now show that final term in the identity \eqref{eq:IFTexp} will have smaller norm than vector on the left side.
\begin{lemma} \label{lem:boundFTI}  
There exist a constant $c> 0$ such that, for each non-negative integer $k$, we have
\[ \left\| \left(\frac{  \bs F \bs T + I}{2} \right)^k \bs W \bs 1  \right\| 
 \leq  \left( 1 -  c \epsilon  \right)^{k} \|\bs W\| \,. \]
\end{lemma} 
\begin{proof}
We prove this lemma by induction. If $k=0$ the lemma is trivial.
Let $k>0$ and let $ \vec{\bs x}  = \left(\frac{  \bs F \bs T + I}{2} \right)^{k-1} \bs W \bs 1$. By the induction hypothesis we have
\[\left\| \left(\frac{  \bs F \bs T + I}{2} \right)^{j} \bs W \bs 1  \right\| 
 \leq  \left( 1 - c  \epsilon \right)^{j} \|\bs W\|  \]  for all $0 \leq j\leq k-1. $
 
\begin{align} \label{eq:FTI2}
\left\| \left(\frac{  \bs F \bs T + I}{2} \right) \vec{\bs x}   \right\|^2 
& = \frac{1}{4} \left( \| \vec{\bs x} \|^2 + \|   \bs F \bs T  \vec{\bs x}  \|^2 + 2 \langle     \bs F \bs T  \vec{\bs x}  ,  \vec{\bs x}    \rangle    \right) .
\end{align}
We bound the second term by $ \|   \bs F \bs T  \vec{\bs x}  \| \leq  \|   \bs F \| \|\bs T\| \|  \vec{\bs x}  \| \leq \|  \vec{\bs x}  \|$.
Let $\vec{\bs x} = f_{-} \langle f_{-}, \vec{\bs x} \rangle + \vec{\bs x}' $ be the orthogonal decomposition of $\vec{\bs x}$ onto $f_-$ and its orthogonal complement. Then we expand the final term as
\[ \langle \bs F \bs T  \vec{\bs x}  ,  \vec{\bs x}   \rangle =   \langle \bs F \bs T  \vec{\bs x}  ,  \vec{\bs x}'   \rangle  + \langle \bs F \bs T  \vec{\bs x}  ,  f_{-}  \rangle \langle \vec{\bs x}  ,  f_{-}  \rangle =
 \langle \bs F \bs T  \vec{\bs x}'  ,  \vec{\bs x}'   \rangle 
 +  \langle \bs F \bs T   f_{-}   ,  \vec{\bs x}'   \rangle  \langle \vec{\bs x}  ,  f_{-}  \rangle 
 + \langle \bs F \bs T  \vec{\bs x}  ,  f_{-}  \rangle \langle \vec{\bs x}  ,  f_{-}  \rangle 
 . \]
which we bound by 
\begin{equation} \label{eq:lbxub} - \| \vec{\bs x} \|^2  \leq \langle \bs F \bs T  \vec{\bs x}  ,  \vec{\bs x}   \rangle \leq \langle \bs T \vec{\bs x}'   , \bs  F \vec{\bs x}'   \rangle  + 2\|  \vec{\bs x}  \|  \|  f_{-} \|  \langle \vec{\bs x}  ,  f_{-}  \rangle  \,.
 \end{equation}
From the induction hypothesis along with Lemma \ref{lem:Fproj} we have
\begin{equation} \label{eq:fmxip} | \langle f_{-} ,  \vec{\bs x}   \rangle | 
\leq 2 \| \bs\varepsilon(t) \| \sum_{j=0}^{k-2}(1- c \epsilon )^j   \left\|  \bs W \bs 1 \right\|  
\leq   \frac{2}{c \eps} \| \bs\varepsilon(t) \|  \| \bs W \bs 1 \|  \,.  \end{equation}
To bound $\langle \bs T \vec{\bs x}'   , \bs  F \vec{\bs x}'   \rangle $, let $ \vec{\bs x}'  =  \vec{\bs x}_1 +  \vec{\bs x}_2$ where $  \vec{\bs x}_1$ is the projection onto the eigenspace of $\bs T$ corresponding to the eigenvalue $-1$, and  $\vec{\bs x}_2$ is the projection onto the remaining eigenspaces. We now consider two cases based on the size of $\|  \vec{\bs x}_2\| $ compared to $ \| \vec{\bs x}\|$. In what follows $c_1$ will be an appropriately chosen small constant depending only on $\eps$. 
Case I. If $ \|  \vec{\bs x}_2\|  \leq  c_1 \| \vec{\bs x}'\|$ then we begin by expanding:
\begin{equation} \label{eq:x2c1} \langle \bs T \vec{\bs x}' , \bs F  \vec{\bs x}' \rangle =   -\langle \vec{\bs x}_1 , \bs F  \vec{\bs x}_1 \rangle + \langle \bs T \vec{\bs x}_2 , \bs F  \vec{\bs x}_1 \rangle
+  \langle \bs T \vec{\bs x}' , \bs F  \vec{\bs x}_2 \rangle \,. \end{equation}
To bound $-\langle \vec{\bs x}_1 , \bs F  \vec{\bs x}_1 \rangle$ from above we project $ \vec{\bs x}_1$ onto $f_{-}$ and its orthogonal complement. By choice of $c_1$, we will make the projection onto $f_-$ small. We will bound the orthogonal term by using that it is of the form $\begin{pmatrix} \bs x \\  - \bs x \end{pmatrix} +\vec{\bs w}$ and thus not in $V_{-1}$. Indeed, for $c_1$ is chosen sufficiently small (compared to $\eps$)
\[ |\langle  \vec{\bs x}_1, f_{-} \rangle| =  | \langle \vec{\bs x}', f_{-} \rangle -    \langle \vec{\bs x}_2 , f_{-} \rangle | \leq 0+ c_1 \| \vec{\bs x} \| \| f_{-}\| \] 
and then
\[- \langle  \vec{\bs x}_1 ,  \bs F \vec{\bs x}_1 \rangle =  - \langle  \vec{\bs x}_1 ,\bs F  f_{-}  \rangle \langle \vec{\bs x}_1  , f_{-} \rangle -   \langle  \vec{\bs x}_1 ,  \bs F ( \vec{\bs x}_1 - \langle \vec{\bs x}_1  , f_{-} \rangle  f_{-} ) \rangle  \leq c_1   \| \vec{\bs x}' \|^2 \| f_{-}\|^2 + (1-\eps)   \| \vec{\bs x}' \|^2  .  \]
So we have that there exist a constant $c_2$ such that  
\[ - \langle  \vec{\bs x}_1 ,  \bs F \vec{\bs x}_1 \rangle \leq  (1- c_2 \eps) \|  \vec{\bs x}'  \| \]
and if $c_1$ is chosen smaller, then $c_2$ can be chosen closer to 1. Then continuing from \eqref{eq:x2c1} gives:
\[\langle \bs T \vec{\bs x}' , \bs F  \vec{\bs x}' \rangle  \leq   (1- c_2 \eps) \|  \vec{\bs x}'  \|^2    +2 \|  \vec{\bs x}'  \|  \|  \vec{\bs x}_2  \| .\]
Thus, for a sufficiently small choice of $c_1$, there is a $c_3$ such that
\begin{equation} \langle \bs T \vec{\bs x}'  , \bs  F \vec{\bs x}'   \rangle  \leq (1-c_3 \eps) \| \bs x\|^2 .   \end{equation}
Case II:  If $\|  \vec{\bs x}_2\|   > c_1 \| \vec{\bs x}\|$ From the bound $\|\bs T  \vec{\bs x}_2\| \leq (1-\eps) \| \vec{\bs x}_2  \| $, we have that 
\[  \langle  \bs T \vec{\bs x}' , \bs F  \vec{\bs x}' \rangle  \leq \sqrt{ \| \bs T  \vec{\bs x}_1\|^2 + \| \bs T  \vec{\bs x}_2\|^2  } \| \vec{\bs x}' \|    \leq \sqrt{  \|  \vec{\bs x}_1\|^2 + (1- \eps) \|  \vec{\bs x}_2\|^2  } \| \vec{\bs x}' \|          \leq  \sqrt{1 - c_1^2\eps }  \| \vec{\bs x}'\|^2 .  \] 
Choosing $c'$ to be the smaller of the bounds between the two cases, we have for any possible $\vec{\bs x}'$
\begin{equation} \label{eq:TFip}  \langle \bs T \vec{\bs x}'  , \bs  F \vec{\bs x}'   \rangle  \leq (1-c' \eps) \| \bs x\|^2 .   \end{equation}
So for all $t$ sufficiently small, combining \eqref{eq:lbxub}, \eqref{eq:fmxip}, and \eqref{eq:TFip} gives for some constant $c_4$:
\[ - \| \vec{\bs x} \|^2  \leq \langle \bs F \bs T  \vec{\bs x}  ,  \vec{\bs x} \rangle  \leq (1 - c_4 \eps)  \| \vec{\bs x} \|^2 . \]
Substituting these estimates into \eqref{eq:FTI2} gives, that there exist a $c$ such that
\begin{align*}
\left\| \left(\frac{  \bs F \bs T + I}{2} \right) \vec{\bs x}'   \right\| 
& \leq (1- c\, \epsilon )\|  \vec{\bs x}' \| .
\end{align*}
as desired.
\end{proof}
\begin{proof}[Proof of Lemma \ref{correct(I+F)}]
By taking the adjoint and then applying the Cauchy-Schwarz inequality we have
\[ \Bigl| \Bigl\langle \bs\Psi^{1/2} \bs W \bs 1 \, , \, 
   \bs \Psi^{-1/2}  (I-\bs T \bs F)^{-1} \bs \Psi^{1/2}  
  \sqrt{\rvec \vec{\brt}} \Bigr\rangle \Bigr| \leq  \| (I- \bs F \bs T)^{-1}  \bs W \bs 1   \|   \left\|  \bs \Psi^{1/2}  
  \sqrt{\rvec \vec{\brt}}  \right\| \,.\] 
Then applying \eqref{eq:IFTexp} iteratively gives:
\[   ( I -  \bs F \bs T  )^{-1} \bs W \bs 1  = \sum_{k=0}^\infty   \left( \frac{ I +  \bs F \bs T }{2} \right)^{k}  \frac{1}{2} \bs W \bs 1 \,.   \]
Then applying Lemma \ref{lem:boundFTI} we have
\[ \|   ( I -  \bs F \bs T  )^{-1} \bs W \bs 1 \| \leq  \| \bs W \bs 1\|  \sum_{k=0}^{\infty} (1- c \eps)^k \,.  \]
The desired inequality then follows.
\end{proof}
Now, writing 
$\bs E =\begin{pmatrix} I & I \\ I & I \end{pmatrix} \in \R^{2n\times 2n}$, 
we factor the matrix $\bs \Psi^{-1/2}(I-  \bs T \bs F) \bs  \Psi^{1/2}$ 
as in \cite[Equation 4.16]{alt2018local}, namely 
\[ 
\bs \Psi^{-1/2}(I-  \bs T \bs F ) \bs  \Psi^{1/2}  
= (I - s^2 \bs  \Psi^{1/2}  \bs E \bs F ( I + \bs F)^{-1} \bs  \Psi^{1/2} ) 
 ( I + \bs  \Psi^{-1/2} \bs F  \bs  \Psi^{1/2} ) .  
\]
Using Lemma~\ref{correct(I+F)} to add a correction term and then substituting this relationship gives:
\begin{align*}
f_n(z)  &=   \lim_{t \to 0} \frac{1}{\pi n} 
   \left\langle \sqrt{\vec{\bs{ r}} \vec{\brt}} 
     -0.5t \bs\Psi^{1/2}\bs W \bs 1, 
    \bs \Psi^{-1/2}  (I-\bs T \bs F)^{-1} \bs \Psi^{1/2}  
   \sqrt{\vec{\bs{ r}  } \vec{\brt}  } \right\rangle \\
   &=  \lim_{t \to 0} \frac{1}{\pi n} 
   \left\langle 
  ( I + \bs  \Psi^{1/2} \bs F  \bs  \Psi^{-1/2} )^{-1}  
   (\sqrt{\vec{\bs{ r}} \vec{\brt}} 
     -0.5t \bs\Psi^{1/2}\bs W \bs 1), \right.  \\ 
& \ \ \ \ \ \ \ \  \ \ \ \ \ \ \ \  
 \ \ \ \ \ \ \ \  \ \ \ \ \ \ \ \  
 \left.  
(I - s^2 \bs  \Psi^{1/2}  \bs E \bs F ( I + \bs F)^{-1} \bs  \Psi^{1/2} )^{-1} 
   \sqrt{\vec{\bs{ r}  } \vec{\brt}  } \right\rangle \\
   &=  \lim_{t \to 0} \frac{1}{2 \pi n} 
   \left\langle \sqrt{\vec{\bs{ r}  } \vec{\brt} },  
(I - s^2 \bs  \Psi^{1/2}  \bs E \bs F ( I + \bs F)^{-1} \bs  \Psi^{1/2} )^{-1} 
   \sqrt{\vec{\bs{ r}  } \vec{\brt}  } \right\rangle ,
\end{align*} 
where the final equality uses \eqref{I+F}. After some algebraic manipulations, it is shown 
in~\cite{alt2018local} that 
\[
(I - s^2 \bs  \Psi^{1/2}  \bs E \bs F ( I + \bs F)^{-1} \bs  \Psi^{1/2} )^{-1} 
\begin{pmatrix} x \\ x \end{pmatrix} = 
\begin{pmatrix} 
 ( I - s^2 \Psi^{1/2} B \Psi^{1/2} )^{-1}  x \\ 
 ( I - s^2 \Psi^{1/2} B \Psi^{1/2} )^{-1}  x  
\end{pmatrix}, 
\]
where 
$$
B x =  \begin{pmatrix} I & I \end{pmatrix} \left( \begin{pmatrix} I & 0 \\ 0 & I \end{pmatrix} - \begin{pmatrix} I & F \\ F^{\tran} & I \end{pmatrix}^{-1}  \right)  \begin{pmatrix} x \\ x \end{pmatrix} .
$$ 
We thus obtain that 
\begin{equation} 
\label{eq:dens2}
f_n(z) \quad =\quad\lim_ {t\to 0} \frac{1}{\pi n} 
\left\langle \sqrt{\bs{ r}  \brt } , ( I - s^2 \Psi^{1/2} B \Psi^{1/2} )^{-1} 
   \sqrt{\bs{ r}  \brt } \right\rangle .  
\end{equation}
The matrix $B$ is symmetric. Furthermore, because the spectrum of $F$ is contained in $[-1,1]$ and the vector $s^2 \Psi$ has entries strictly less than $1$ we have the eigenvalues of $s^2 \Psi^{1/2} B \Psi^{1/2} $ are bounded away from 1, uniformly in $t$; see \cite[Eq. (4.20) - (4.22)]{alt2018local} for details (note the matrix $B$ is labeled $A$ there).
To lower bound this expression we begin by noting that if $ \begin{pmatrix}   x \\  x \end{pmatrix} $ is an eigenvector of $\bs F$, with eigenvalue $\lambda$, then 
\begin{equation} \label{eq:Beigv}
B x = \frac{2  \lambda}{1+\lambda}   x.
 \end{equation} 
From Lemma \ref{lemma:altF} we have that $ \begin{pmatrix}  \Psi^{-1/2} \sqrt{\bs{ r}  \brt }   \\ \Psi^{-1/2} \sqrt{\bs{ r}  \brt }  \end{pmatrix}$ is $O(t)$ from an eigenvector of $\bs F$ with eigenvalue $1$. Let $f_+$ be this eigenvector. Since the operator $ ( I - s^2 \Psi^{1/2} B \Psi^{1/2} )^{-1} $ has uniformly bounded norm, we can replace $  \sqrt{\bs{ r}  \brt }$ with $ \Psi^{1/2}  f_+  $, at the cost of an error that goes to zero as $t\to 0$.
 We now have all the elements to provide a lower bound on the density. 
Using the Cauchy-Schwarz inequality 
(with respect to the inner product $\langle \cdot, (s^{-2}\Psi^{-1}- B )^{-1} \cdot\rangle$)
along with \eqref{eq:Beigv}, we have  
\begin{align*} 
\lim_{t\to 0}\langle \sqrt{\bs{ r}  \brt } , ( I -  s^2 \Psi^{1/2} B \Psi^{1/2})^{-1}
 \sqrt{\bs{r}  \brt } \rangle &= 
 \lim_{t\to 0}\langle  \Psi^{-1/2}  f_+ , ( I -  s^2 \Psi^{1/2} B \Psi^{1/2})^{-1}
 \Psi^{-1/2}  f_+\rangle \\
 &= 
\lim_{t\to 0}s^{-2} \langle f_+ , (s^{-2} \Psi^{-1} -  B)^{-1}
f_+ \rangle \\ 
&\geq  \lim_{t\to 0} \frac{\| f_+ \|^2} 
 {s^2 \langle f_+, (s^{-2} \Psi^{-1} -  B)
 f_+ \rangle} \\
  &=  \lim_{t\to 0} \frac{\| f_+ \|^2} 
 {s^2 \langle  f_+ , (s^{-2} \Psi^{-1}  - I )
  f_+  \rangle} \,.
\end{align*}
Taking the limit $t \to 0$ and using that  
$f_+ \to \Psi^{-1/2} \sqrt{\bs{ q}  \bqt } $ as $t \to 0$ gives 
\begin{align*}
\lim_{t\to 0} \frac{\| f_+ \|^2} 
 {s^2 \langle  f_+ , (s^{-2} \Psi^{-1}  - I)
  f_+  \rangle} =
\frac{\| \Psi^{-1/2} \sqrt{\bs{ q} \bqt } \|^2} 
{ s^2 \langle \Psi^{-1} \bs{ q} \bqt  , (s^{-2} \Psi^{-1} \bs 1 - \bs 1) \rangle } .
\end{align*} 
Then using the equalities 
\[   \Psi^{-1} (s^{-2} \Psi^{-1} \bs 1 - \bs 1) = \Psi^{-1} \frac{\bphi \bphit}{s^2} = \frac{ \Psi \bs{ q} \bqt  }{s^2}\  \]
gives
\begin{equation} \label{eq:LBqq}
f_n(z) \quad  \geq  \quad \frac{ \sum_{i=1}^n \Psi_i^{-1} q_i \qt_i  }{  \sum_{i=1}^n \Psi_i q_i^2 \qt_i^2 }.
\end{equation}

From the uniformity in $t$ in Lemma~\ref{lemma:boundsonq}, $q_i, \qt_i$ are upper and lower-bounded and hence Theorem~\ref{th:positive}--\eqref{th:positive-i1} is proven. 
\end{proof}

\subsection{Proof of Corollary \ref{prop:circlaw}}\label{proof:circlaw}
The proof relies on the following theorem by Friedland and Karlin:
\begin{theo}[Theorem 3.1, Equation (1.9) in \cite{FK:spectral-rad}] \label{thm:FKspec}
Let $M$ be an irreducible non-negative matrix with Perron-Frobenius left and right eigenvectors $\bs{u},\bs{v}$ normalized so that $\sum_{i\in [n]} u_i v_i =1$ and $\rho(M) =1$. Let $D$ be a diagonal matrix with positive entries. Then
\begin{equation} \label{eq:specradbound} \rho(M D) \geq \prod_{i}^n d_i^{u_i v_i} \end{equation}
\end{theo}

\begin{proof}[Proof of Corollary \ref{prop:circlaw}]
Without loss of generality we consider $V$ such that $\rho(V)=1$. Proposition \ref{prop:LB}, $\mu_n$ gives the formula for the density at $0$.
By \eqref{eq:sink}, matrix $S: = \diag(\bq) V \diag(\bqt)$ is  doubly stochastic hence with spectral radius 1 and any left or right Perron-Frobenius eigenvector $\bs{u}$ or $\bs{v}$ is proportional to $\bs{1}_n$. In particular, the normalization $\sum_{i\in [n]} u_i v_i=1$ implies $u_iv_i= n^{-1}$.  We now apply Theorem \ref{thm:FKspec} with $M = S$ and $D = \left( \diag(\bqt) \diag(\bq) \right)^{-1}$ to get 
$$
\rho(S \left( \diag(\bqt) \diag(\bq) \right)^{-1}) \quad \ge \quad \prod_{i\in [n]} \left(\frac 1{q_i(0)\widetilde q_i(0)}\right)^{\frac 1n} .
$$
Since $\rho(S D) = \rho(\left(  \diag(\bq) \right)^{-1}  S \left( \diag(\bqt) \right)^{-1}) = \rho( V ) = 1$,
we arrive at
$$ 
1 \quad \leq\quad  \prod_{i\in [n]} \left[ q_i(0)  \widetilde q_i(0) \right]^{\frac 1n} \quad \leq\quad  \frac{1}{n} \sum_{i\in [n]}  q_i(0)  \widetilde q_i(0)\,   ,
$$
where the second inequality is the the AM-GM inequality. We note that equality in the final inequality only occurs if $  q_i(0) \widetilde q_i(0)  =1 $ for all $i\in [n]$. This condition can be rewritten as $\diag(\bq)^{-1}=\diag(\widetilde \bq)$, which, by Remark \ref{rem:necessary-circular}, implies the desired form $V = \diag(\bq)^{-1}\, S\, \diag(\bq)\, $.
\end{proof}

\bibliographystyle{abbrv}
\bibliography{math}

\noindent {\sc Nicholas Cook}\\
Department of Mathematics\\
Duke University\\
Durham, NC 27708\\
e-mail: {\tt nickcook@math.duke.edu}\\

\noindent {\sc Walid Hachem, Jamal Najim},\\
CNRS / Laboratoire d'Informatique Gaspard Monge,\\
Universit\'e Gustave Eiffel, ESIEE\\ 
5, Boulevard Descartes,\\
Champs sur Marne,
F-77454 Marne-la-Vall\'ee, France\\
e-mail: {\tt \{walid.hachem,jamal.najim\}@univ-eiffel.fr}\\

\noindent {\sc David Renfrew}\\
Department of Mathematical Sciences\\
Binghamton University (SUNY)\\
Binghamton, NY 3902-6000\\
e-mail: {\tt renfrew@math.binghamton.edu}\\

\end{document}